\documentclass[a4paper,11pt]{amsart}
\usepackage{graphicx,amssymb,amstext,amsmath,amsfonts,amsthm}
\usepackage{url}
\usepackage{nicefrac}
\usepackage[normalem]{ulem}
\usepackage{geometry}
\usepackage{anysize}
\marginsize{2.5cm}{2.5cm}{2cm}{2cm}
\theoremstyle{plain}
\newtheorem{thm}{Theorem}
\newtheorem{prop}{Proposition}
\newtheorem{hyp}{Hypothesis}
\numberwithin{equation}{section}
\newtheorem{lem}{Lemma}[section]
\newtheorem{cor}[lem]{Corollary}
\newtheorem{defn}[lem]{Definition}
\newtheorem{rem}[lem]{Remark}

\DeclareMathOperator*{\essinf}{ess\,inf}
\newcommand{\al}{\alpha}
\newcommand{\e}{\varepsilon}
\newcommand{\la}{\lambda}
\newcommand{\Om}{\Omega}
\newcommand{\ra}{\rightarrow}
\newcommand{\lra}{\longrightarrow}
\newcommand{\ti}{\tilde}
\newcommand{\<}{\ensuremath{\langle}}
\renewcommand{\>}{\ensuremath{\rangle}}
\newcommand{\ind}{\mathbf{1}}
\newcommand{\lqq}{\leqslant}
\newcommand{\gqq}{\geqslant}
\newcommand{\ud}{\mathrm{d}}
\newcommand{\vt}{\vartheta}
\newcommand\norm[1]{\left\lVert#1\right\rVert_{\mathrm{TV}}}
\newcommand{\bB}{\mathcal{B}}
\newcommand{\cC}{\mathcal{C}}
\newcommand{\fF}{\mathcal{F}}
\newcommand{\gG}{\mathcal{G}}
\newcommand{\hH}{\mathcal{H}}
\newcommand{\nN}{\mathcal{N}}

\newcommand{\oO}{\mathcal{O}}
\newcommand{\tT}{\mathcal{T}}
\newcommand{\vV}{\mathcal{V}}

\newcommand{\xX}{\mathcal{X}}
\newcommand{\fw}{\mathfrak{w}}
\newcommand{\ft}{\mathfrak{t}}
\newcommand{\fX}{\mathfrak{X}}
\newcommand{\fD}{\mathfrak{D}}
\newcommand{\fm}{\mathfrak{m}}
\newcommand{\fb}{\mathfrak{b}}
\newcommand{\fT}{\mathfrak{T}}

\newcommand{\CC}{\mathbb{C}}
\newcommand{\EE}{\mathbb{E}}
\newcommand{\NN}{\mathbb{N}}
\newcommand{\PP}{\mathbb{P}}
\newcommand{\RR}{\mathbb{R}}
\newcommand{\SSS}{\mathbb{S}}
\usepackage[]{color}
\definecolor{DarkBlue}{rgb}{0,0,0.7}
\definecolor{DarkRed}{rgb}{0.95,0,0}

\title[The cutoff phenomenon for nonlinear Langevin systems with small  noise]{The Cutoff phenomenon in total variation for nonlinear Langevin systems with small layered stable noise}
\author{G. Barrera}
\address{University of Helsinki, Department of Mathematics and Statistics. Helsinki, Finland}
\email{gerardo.barreravargas@helsinki.fi}
\author{M.A. H\"ogele}
\address{Universidad de los Andes. Bogot\'a, Colombia}
\email{ma.hoegele@uniandes.edu.co}
\author{J.C. Pardo}
\address{
 CIMAT. Jalisco S/N, Valenciana, CP 36240. Guanajuato, Guanajuato, M\'exico.}
\email{jcpardo@cimat.mx }

\begin{document}
\begin{abstract}
This paper provides an extended case study of 
the cutoff phenomenon for a prototypical class of nonlinear Langevin systems with a single stable state perturbed by an additive pure jump L\'evy noise of small amplitude $\varepsilon>0$, 
where the driving noise process is of layered stable type.
Under a drift coercivity condition the associated family of processes $X^\varepsilon$ turns out to 
be exponentially ergodic with equilibrium distribution $\mu^{\varepsilon}$ in total variation distance which extends a result from \cite{Peng} to arbitrary polynomial moments. 

The main results establish the cutoff phenomenon with respect to the total variation, 
under a sufficient smoothing condition of Blumenthal-Getoor index $\alpha>\frac{3}{2}$. 
That is to say, in this setting we identify a deterministic time scale $\mathfrak{t}_{\varepsilon}^{\mathrm{cut}}$ satisfying $\mathfrak{t}_\varepsilon^{\mathrm{cut}} \rightarrow \infty$, as $\varepsilon \rightarrow 0$, 
and a respective time window, $\mathfrak{t}_\varepsilon^{\mathrm{cut}} \pm o(\mathfrak{t}_\varepsilon^{\mathrm{cut}})$, during which the total variation distance between the current state and its equilibrium $\mu^{\varepsilon}$ essentially collapses as $\varepsilon$ tends to zero. 
In addition, we extend the dynamical characterization 
under which the latter phenomenon can be described 
by the convergence of such distance to a unique profile function 
first established in \cite{BJ1} to the L\'evy case for nonlinear drift. 
This leads to sufficient conditions, which can be verified in examples, 
such as gradient systems subject to small symmetric $\alpha$-stable noise for $\alpha>\frac{3}{2}$.
The proof techniques differ completely from the Gaussian case due to the 
absence of a respective Girsanov transform
which couples the nonlinear equation 
and the linear approximation asymptotically even for short times.
\end{abstract}

\maketitle

\section{\textbf{Exposition}}
\subsection{\textbf{Introduction}} \hfill\\
Roughly speaking the term \textit{cutoff phenomenon} with respect to a distance $d_1$ refers to the following asymptotic dynamics: 
consider the setting of a parametrized family of stochastic processes $(X^{\e})_{ \e>0}$, $X^\e = (X^\e_t)_{t\gqq 0}$, such that for each $\e>0$ the process $X^{\e}$ has a unique limiting distribution~$\mu^{\e}$. 
Then - as $\e$ decreases to $0$ - 
the function $t\mapsto d_\e(X^\e_t, \mu^\e)$ given by a suitably renormalized distance $d_\e$ (of $d_1$) between the law of $X^{\e}_t$ and the corresponding limiting distribution $\mu^{\e}$ essentially resembles the step function $t\mapsto \mbox{diam} \cdot\ind_{[0, \ft_\e^{\mathrm{cut}}]}(t)$. This function descends from the value $\mbox{diam}\in (0, \infty]$ to the value~$0$,  
at a deterministic cutoff time scale $\ft^{\mathrm{cut}}_\e$, which tends to $\infty$ as $\e \ra 0$, 
where $\mbox{diam}=\limsup_{\e\to 0} \mbox{diameter}(d_\e)$ in the respective domain of probability distributions over the state space.
In other words, there exist positive deterministic functions $\e\mapsto \ft_\e^{\mathrm{cut}}$ and
$\e\mapsto \fw^{\mathrm{cut}}_{\e}$
satisfying $\ft^{\mathrm{cut}}_{\e}\ra \infty$ and $\fw^{\mathrm{cut}}_{\e}\ll \ft_\e^{\mathrm{cut}}$ such that
on the interval $(\ft^{\mathrm{cut}}_{\e}-\fw^{\mathrm{cut}}_{\e},\ft^{\mathrm{cut}}_{\e}+\fw^{\mathrm{cut}}_{\e})$ the transition from
$\mbox{diam}$ to $0$ is bound to happen. In general, this transition may depend on subsequences $\e_j\ra 0$ as $j\ra\infty$. In certain situations, a proper limit can be taken, and the limiting function gives rise to a so-called cutoff profile function connecting the asymptotic values $\mbox{diam}$ and $0$ smoothly.

This abrupt convergence phenomenon was first described by  Aldous and  Diaconis  \cite{AD} in the early eighties to conceptualize the collapse of the total variation distance between Markov chain marginals related to card shuffling to its uniform limiting distribution. Since then, this behavior has been studied by numerous authors and in different - mainly discrete - settings.  For instance
 we refer to Diaconis \cite{DI},  Mart\'inez and Ycart \cite{MY} and  Levin et al. \cite{LPW} for the Markov chain setting,  Chen and Saloff-Coste \cite{CSC08} considered some ergodic Markov processes, Lachaud \cite{BL} and Barrera \cite{BA}  for the case of the Ornstein-Uhlenbeck processes driven by a Brownian motion, to name but a few. Further standard texts on the cutoff phenomenon include \cite{Al83, AD87, BHP17, B-HS17, BLY06, BD92, Beresty, BBF08, CSC08, DI87, DGM90, DS,  La16, LanciaNardiScoppola12,  LPW, Meliot14, Trefethenbrothers, Yc99} and the references therein. The newest developments in this active field of research are found in the recent publications \cite{B-HLP19, BQZ20, BCS18, BCS19, CS21, LL19, HS20, HH21}.
 
This article provides a case study on the cutoff phenomenon in the (unnormalized) total variation distance 
for the strong solution process $X^\e$ of a class of  stochastic differential equations with nonlinear coercive vector field $-b$ with a non-degenerate stable state $0$ subject to an additive pure jump L\'evy process $L$ at $\e$-small amplitude
\begin{align}\label{dde0}
\left\{
\begin{array}{r@{\;=\;}l}
\ud X^\e_t & - b(X^\e_t) \ud t + \e \ud L_t \quad \textrm{ for  } t\gqq 0,\\
X^\e_0 & x\in \RR^d. 
\end{array}
\right.
\end{align}
Similar - and in some sense simpler - settings have been studied before: the case of  \textit{nonlinear, coercive} vector fields $(-b)$ subject to \textit{Brownian perturbation $L = W$} with respect to the \textit{total variation} \cite{BJ, BJ1} and two cases of \textit{linear, asymptotically exponentially stable} drifts $- b = - Q$ - that is,  eigenvalues have negative real part, but the matrix is not necessarily coercive, see \cite{Tudoran} - subject to \textit{pure-jump L\'evy noises $L$} \cite{ BHP,BP} in the \textit{total variation} and the \textit{Wasserstein distance}, respectively. 
This paper yields the first results on the cutoff phenomenon for \textit{nonlinear coercive, pure-jump L\'evy} SDEs in the total variation distance, which is  fraught with technical difficulties:
\begin{enumerate}
\item[(a)] It inherits the regularity issues from the linear case \cite{BP} due to the total variation distance. 
\item[(b)] It earns additional challenges due to the nonlinearity.  
In particular, there is a gap in the literature concerning available (short-time) coupling results between the solution of L\'evy SDEs with the nonlinear vector fields and its (linear) Ornstein-Uhlenbeck approximation. 
\end{enumerate}
The regularity issue (a) is overcome by the careful choice of the setting 
of a class of locally layered stable noise processes, by which we generalize 
the notion of  layered stable processes -  introduced by Houdr\'e and Kawai \cite{HOU} -  and the equator condition inspired by \cite{Simon}. 
Regularity results for densities of SDEs which turn out to be crucial for results in the total variation distance have been extensively studied for instance in \cite{ DeFo13, FoPr10,IKT18,ArtKHT,PI}.  

The nonlinear coupling problem (b) is essentially reduced to the control of two partial errors 
of a different nature addressed in Proposition~\ref{prop: stc} and Proposition~\ref{prop: lic}. 
The first error, which is dominated in the statement of Proposition~\ref{prop: stc} 
represents the crucial part of the proof of the main results.  
It directly compares the nonlinear process $X^\e$ with 
its linear inhomogeneous Ornstein-Uhlenbeck approximation 
for short times. While there are very recent short-time couplings for 
SDEs with different (nonlinear) drift under a Brownian driver 
(see Eberle and Zimmer \cite{EBERLE}), to our knowledge 
the literature on respective pure jump counterparts is virtually 
nonexistent. In order to obtain short-time coupling between the 
linear and the nonlinear vector field, we use Plancherel's theorem, 
and appropriate differential inequalities for the characteristic function 
of a strongly localized version of $X^\e$ for Blumenthal-Getoor index $\al>\nicefrac{3}{2}$. 
To the best of our effort it seems hard to derive with this technique 
the correct (exponential) integrability of 
the tails of the characteristic function - even in the linear, 
scalar Gaussian case - and at the same time it is unclear how to relax this condition.
The same  sort of technical difficulties concerning the Fourier approach arises in condition (a) p. 345 of \cite{FoPr10}.
The second error consists of the total variation distance between the short time 
linear inhomogeneous Ornstein-Uhlenbeck (Freidlin-Wentzell first order) approximation 
under linear and nonlinear initial conditions. 
A slight extension of Theorem~3.1 in \cite{HOU} 
provides a stable local limit theorem on the short-range behavior, 
which allows for an appropriate coupling in the proof of Proposition~\ref{prop: lic}.

The difficulty of the nonlinear case studied in this article 
can be informally understood as follows. 
In the linear case $b(x) = - Q x$, it is well-known 
that by the variation-of-constants formula $X^\e_t$ 
can be written as the sum of the deterministic matrix exponential dynamics 
plus the respective stochastic convolution. 
Since the total variation distance is well-behaved under 
deterministic and mutually independent components, 
it can be dominated without too much effort in the linear case. 
This program was carried out in \cite{BP}. 
In the nonlinear, additive noise case $X^\e_t$ 
can be written analogously, but it exhibits an additional error term. 
That is, $X^\e$ is given as the sum of the nonlinear deterministic dynamics, 
its stochastic (nonlinear) convolution with the noise and 
an additional random term representing the (implicit) 
nonlinear residual of the noise, which is neither deterministic nor independent 
from the noise convolution and therefore not easily dominated in total variation. 
Beyond that, the aforementioned random residual term turns out to be a challenge 
since there is no analogue of Slutsky's lemma for the total variation distance even in the case of smooth densities. 
For the sake of completeness and 
since we are not aware of a reference literature,
a counterexample is given in Subsubsection~\ref{subsub:counter}. 
On a more abstract level, the additional difficulties encountered are illustrated for  
the Wasserstein upper bounds of the total variation 
which require additional density gradient estimates 
(see Theorem~2.1 in \cite{CW04}). 

Our results cover the important examples 
of overdamped gradient systems, such as the Fermi-Pasta-Ulam-Tsingou potential, perturbed by 
pure jump L\'evy processes with Blumenthal-Getoor index $\al> \nicefrac{3}{2}$ in the sense of Definition \ref{def:llslm89} and \ref{hyp: layered}, such as symmetric $\al$-stable processes, symmetric tempered $\al$-stable processes in Rosi\'nski \cite{Ros07}
and the symmetric Lamperti-$\al$-stable process \cite{CPP10}. 
If - in addition - the limiting distributions turns out to be rotationally invariant, 
the existence of a cutoff profile is shown to be equivalent to a 
computational linear algebra eigenvector problem 
first established in \cite{BHP} for the easier situation of the Wasserstein distance. 
This characterization is given as a specific orthogonality condition 
of the (generalized) eigenvectors of the linearization $-Db(0)$ of $-b$ in the stable state $0$. 
It allows to carry over several results   
from the linear case under the Wasserstein distance in \cite{BHP}, 
to the case of a nonlinear vector field $-b$ and the total variation distance. 
In physics terminology, 
our results can be restated that 
the existence of a cutoff profile is equivalent to the 
absence of non-normal growth effects in $-Db(0)$
in the case of rotationally 
invariant limiting distributions in the nonlinear setting. 
 
For a complete comparison of the different settings and results 
and in order to avoid a lengthy introduction, 
we refer to the following self-explanatory table. 
 \begin{center}
\begin{footnotesize} 
\begin{tabular}{|l||l|l|l|l|l|}
\hline 
\textbf{Settings} & \cite{BJ} & \cite{BJ1}  & \cite{BP} & \cite{BHP} & \textbf{this article} \\
\hline\hline
\textbf{Dimension} & scalar & multivariate & multivariate & multivariate & \textbf{multivariate} \\
\hline
\textbf{Vector field} & nonlinear  & nonlinear & linear & linear & \textbf{nonlinear}\\
\hline
\textbf{Fixed point} 
& strong & strong 
& neg. real parts 
& neg. real parts 
& \textbf{strong} \\
\textbf{stability} & coercivity & coercivity & of the eigenvalues & of the eigenvalues & \textbf{coercivity} \\
\hline
\textbf{Noise process} & Brownian  & Brownian  & pure jump L\'evy & pure jump L\'evy & \textbf{pure jump L\'evy}\\
& motion & motion & &  & \\
\hline 
\textbf{Noise process} & no & no & no& yes & \textbf{no}\\
\textbf{degeneracy} &  &  & &  & \\
\hline
\textbf{Restrictions} & none & none & finite log-moment & finite moment & \textbf{finite moment}\\
\textbf{on the noise}&&& + Hypothesis~(H) &of order $\beta>0$ & \textbf{of order $\beta>0$,} \\
&&&&&\textbf{+~strongly locally}\\
&&&&&\textbf{layered stable}\\
&&&&&\textbf{$\alpha \in(3/2,2)$}\\
\hline
\textbf{Limiting}   & explicitly & abstract, & characteristic  & d.n.a. due to & \textbf{completely}\\
\textbf{distribution} & known & expansions in $\e$ & function known  & shift linearity of  & \textbf{abstract} \\
&&known \cite{Sheu, Mikami} &&Wasserstein dist.&\\
\hline
\hline 
\textbf{Results} & \cite{BJ} & \cite{BJ1}  & \cite{BP} & \cite{BHP} & \textbf{this article} \\
\hline\hline
\textbf{Distance} & total variation & total variation & total variation & rescaled  & \textbf{total variation} \\
&&&&Wasserstein&\\
\hline
\textbf{Window } & yes & yes & yes & yes & \textbf{yes} \\
\textbf{cutoff} &&&&&\\
\hline
\textbf{Profile} & yes &  dynamical & dynamical & 
 dynamical & \textbf{dynamical}\\
\textbf{cutoff} &    & characterization & characterization & ~characterization  & \textbf{~characterization} \\
&&&&+~normal growth&\textbf{+~normal growth} \\ 
&&&&~characterization&\textbf{~characterization}\\ 
&&&&~(general case) &\textbf{~(rot. inv. case)}\\
\hline
\textbf{Short time}  & Girsanov + & Girsanov + & Fourier inversion  & does not apply  & \textbf{Plancherel}\\
\textbf{coupling} & Pinsker's & Hellinger's &  &  & \textbf{isometry of $L^2$}\\
&inequality & inequality &&&\\
\hline
\end{tabular} 
 \end{footnotesize}
 \end{center}

The nonlinear Wasserstein setting with 
results in the spirit of \cite{BHP} 
are studied in the paper \cite{BHP2}. 

In the manuscript we prove several results of interest in its own right 
which to our knowledge have not been present in the literature:
(1) In Theorem~\ref{ergodicitytheorem} we generalize the strong 
ergodicity result Theorem~4.1 in \cite{Peng} from moments $\beta \gqq 2$ to any $\beta >0$. 
The proof is given in Subsection~ \ref{Appendix D}. 
(2) In Definition \ref{def:llslm89} we introduce the class of locally layered stable process, which are precisely the class 
of processes for which the short-range behavior in Theorem~3.1 in \cite{HOU} remains valid. 
(3) In Proposition~\ref{lem:Holdercontinuity} we give an elementary proof of the local $\beta$-H\"older continuity of the characteristic exponents in case of $\beta\in (0, 1]$-moments in Subsection~\ref{Appendix C.1}. (4) We also provide a complete overview of the behavior of the estimates of matrix exponentials and related flows for an asymmetric matrix in Appendix~\ref{A}, since we are not aware of a reference in the literature.

The manuscript is organized in two large sections and an extended Appendix. 
The first section lays out the setting, the main results formulated as Theorem~\ref{thm: main result} and 
\ref{thm: main result2}, the examples and the skeleton of concluding 
steps in the proof of the main results, 
which boils down to the proofs of Proposition~\ref{prop: windows}, \ref{prop: stc}, \ref{prop: lic} and \ref{prop: equil}. The respective results are proven in the (correspondingly ordered) Subsection~\ref{sec: OUP}, \ref{sec: short coupling}, \ref{sec: linear coupling} and \ref{sec: local}, respectively. 
Subsection~\ref{sec: OUP} shows the cutoff result for the linear inhomogeneous Ornstein-Uhlenbeck process. 
Subsections~\ref{sec: short coupling} and \ref{sec: local} yield the coupling of the inhomogeneous Ornstein-Uhlenbeck and the nonlinear short-time coupling, which exhibits the core difficulties. 
The Appendix is divided in Section \ref{A}, \ref{ap: multiscale}, \ref{Appendix C} and \ref{Appendix D} in which several auxiliary results are shown as a by-product in its own right. Section \ref{A} provides all necessary fine results on the derministic dynamics. Appendix \ref{ap: multiscale} yields a quantitative estimate of the Freidlin-Wentzell first order approximation. Appendix~\ref{Appendix C} gives several auxiliary technical results, some of which we have not been aware in the literature, such as the local $\beta$-H\"older-continuity of a L\'evy process in the presence of arbitrary $\beta$-moments. Appendix~\ref{Appendix D} yields the proof of Theorem~1, which implies the exponential ergodicity of $X^\e$ towards $\mu^\e$, which extends a result by \cite{Peng} to the case of an arbitrary  positive  finite moment.

\bigskip 
\subsection{\textbf{The setting and the main results}} 
\subsubsection{\textbf{The deterministic dynamics $\varphi^x$}} \hfill\\

\noindent Let  $b\in \cC^2(\RR^d, \RR^d)$ be a vector field with $b(0)=0$ 
satisfying the following coercivity condition.
\begin{hyp}[Coercivity]\label{hyp: potential}\hfill\\
Assume that there exists a positive constant $\delta$ such that
\begin{equation}
\label{coercividad}
 \< b(x)-b(y),x-y \> \gqq \delta |x-y|^2 \qquad \mbox{ for all } x, y\in \RR^d,
\end{equation}
where $|\cdot|$ and $\<\cdot, \cdot\>$ denote the Euclidean norm and the standard inner product on $\mathbb{R}^d$, respectively.
\end{hyp}

\noindent In this manuscript we are interested in the stochastically perturbed analogue of 
the dynamical system given as the global solution flow $(\varphi^{\cdot}_t)_{t\gqq 0}$ of the ordinary differential equation 
\begin{align}\label{dde1.1}
\left\{
\begin{array}{r@{\;=\;}l}
\ud \varphi_t^x & -b(\varphi_t^x) \ud t \qquad \textrm{ for any }\quad t\gqq 0,\\
\varphi_0^x & x\in \mathbb{R}^d.
\end{array}
\right.
\end{align}
It is well-known that Hypothesis~\ref{hyp: potential} implies the well-posedness  of \eqref{dde1.1}, see for instance Subsection~2.1 in \cite{BJ1}. 
Furthermore, in our setting inequality \eqref{coercividad} is equivalent to
\[
 \< Db(x)y,y \> \gqq \delta |y|^2 \qquad \mbox{ for all } x, y\in \RR^d,
 \]
where $Db(x)$ denotes the derivative of the vector field $b$ at the point $x$. Moreover, since $b(0)=0$, we have
\[
\frac{\ud }{\ud t} |\varphi^x_t|^2=-2\<\varphi^x_t,b(\varphi^x_t)\>\lqq -2\delta |\varphi^x_t|^2\quad \textrm{ for any } t\gqq 0.
\]
As a consequence $|\varphi^x_t|\lqq e^{-\delta t}|x|$ for any $t\gqq 0$ and $x\in \RR^d$, i.e.
$0$ is an asymptotically exponentially stable fixed point of \eqref{dde1.1}.
For our purposes, however, we need the precise description of the convergence to $0$ in terms of the spectral decomposition of $-Db(0)$. This is the purpose of the  following lemma which characterizes the asymptotics of $\varphi^x_t$ as $t$  tends to $\infty$
 and slightly refines the classical and well-known result by Hartman-Grobman \cite{Grobman,HARTMAN} under Hypothesis~\ref{hyp: potential}. This lemma  turns out to be crucial for the precise shape of the cutoff time and time window.

\begin{lem}[Hartman-Grobman]
\label{asymp}\hfill\\
Consider $(\varphi^x_t)_{t\gqq 0}$ defined by \eqref{dde1.1} under Hypothesis~\ref{hyp: potential}.
Then for  any $x\in \mathbb{R}^{d}\setminus\{0\}$  there exist 
\begin{enumerate}
\item[(i)] positive constants $\quad \lambda:=\lambda_x,\quad \tau:=\tau_x,\quad \ell:=\ell_x, \quad m:=m_x,\quad \ell,m\in \{1,\ldots,d\}$,
\item[(ii)] angles $\quad\theta^1:=\theta^{1}_{x},\dots,\theta^m:=\theta^m_x\in[0,2\pi)$,
where all angles $\theta^k \in(0,2\pi)$ come in pairs
$(\theta^{j_*},\theta^{j_*+1})=(\theta^{j_*}, 2\pi-\theta^{j_*})$
and
\item[(iii)] linearly independent vectors $v^1:=v^1_x,\dots,v^m:=v^m_x$ in $\mathbb{C}^d$ satisfying
$(v^{j_*},v^{j_*+1})=(v^{j_*}, \bar{v}^{j_*})$  whenever $(\theta^{j_*},\theta^{j_*+1})=(\theta^{j_*}, 2\pi-\theta^{j_*})$,
\end{enumerate}
such that
\begin{equation}\label{eq: deterministic exponential convergence}
\lim_{t \to \infty} 
\left| \frac{e^{\lambda t}}{t^{\ell-1}}\cdot \varphi^x_{t+\tau} - \sum_{k=1}^m e^{i\theta_k t}v^k \right|=0.
\end{equation}
Moreover,
\begin{equation}\label{eq: outsidezero}
0<\liminf_{t\rightarrow \infty}\left|\sum_{k=1}^{m} e^{i  t\theta^k} v^k\right|
\lqq 
\limsup_{t\rightarrow \infty}\left|\sum_{k=1}^{m} e^{i  t\theta^k} v^k\right|\lqq 
\sum_{k=1}^{m}  |v^k|.
\end{equation}
\end{lem} 
The  proof of this result is given in Lemma~B.2 of \cite{BJ1}.
\begin{rem}
For $x\in \RR^d, x\neq 0$, $\la_x$  corresponds to a real part of some  eigenvalue of $Db(0)$ and 
$\{v^k, k=1,\ldots, m\}$ are elements of the Jordan decomposition of $Db(0)$
according to the flag of 
eigenspaces (along increasing real parts of the corresponding eigenvalues) containing $x$. 
For any generic choice of $x$, $\la_x$ corresponds to the smallest real part of the eigenvalues of $Db(0)$.  
\end{rem}

\bigskip 
\subsubsection{\textbf{The stochastic perturbation $\e \ud L$}} \label{subsec:stochasticperturbation}\hfill\\

\noindent On a given probability space $(\Om, \fF, \PP)$ consider a L\'evy process 
$L=(L_t)_{t\gqq 0}$ with values in $\RR^d$, i.e. a stochastic process with c\`adl\`ag paths,  independent and stationary increments and issued from $0$. Its marginals are determined by the celebrated L\'evy-Khintchin formula 
\[
\EE\big[e^{i\< u, L_t\>}\big] = e^{t \Psi(u)} \quad \textrm{ for any } u \in \mathbb{R}^d,
\]
with the characteristic exponent
\[
\Psi(u) = i\< a, u\> - \frac{1}{2} \< u, \Sigma u\> + 
\int_{\RR^d} \Big(e^{i\< u, z\>} - 1 - i\< u, z\> \ind_{B_1(0)}(z)\Big) \nu(\ud z),
\]
where $B_1(0)=\{x\in \RR^d~|~|x|< 1\}$, $a\in \RR^d$,  $\Sigma \in \RR^{d \times d}$ is a non-negative definite matrix and $\nu: \bB(\RR^d) \ra [0, \infty]$
is a $\sigma$-finite Borel measure satisfying 
\[
\nu(\{0\}) = 0 \qquad \mbox{ and } \qquad \int_{\RR^d} (1\wedge |z|^2) \nu(\ud z) < \infty.  
\]
Let $(\fF_t)_{t\gqq 0}$ be the enhanced natural filtration of $L$ satisfying the usual conditions of Protter~\cite{Protter}.

The  stochastic analogue of the dynamical system \eqref{dde1.1} 
is described by the following stochastic differential equation. 
For $\e>0$, we consider 
\begin{align}\label{dde1}
\left\{
\begin{array}{r@{\;=\;}l}
\ud X^\e_t & - b(X^\e_t) \ud t + \e \ud L_t \quad \textrm{ for  } t\gqq 0,\\
X^\e_0 & x,
\end{array}
\right.
\end{align}
 which under Hypothesis~\ref{hyp: potential} 
has a unique strong solution $X^{\e,x}=(X^{\e,x}_{t})_{t \gqq 0}$. Such strong solution satisfies the strong Markov property with respect to the filtration $(\fF_t)_{t\gqq 0}$, see for instance p.~1026 in \cite{WANG} and the references therein.\\

\bigskip 
\subsubsection{\textbf{Exponential ergodicity and regularity of the limiting distributions~$\mu^\e$}} \hfill\\

\noindent \textbf{a) Hypotheses on the L\'evy measure: } 
The existence of invariant measures is known to be true for systems with as little as logarithmic moments \cite{Kallianpur}, however we need exponential ergodicity in the total variation distance, which typically needs some (arbitrarily low) finite moments and regularity of the transition kernel for the L\'evy measure, see for instance \cite{KULPAPER}. Both requirements are met by the class of L\'evy measures defined below.

The cutoff results we have in mind can be understood as asymptotically precise small noise formulations of an exponential ergodicity result in total variation distance. Such results typically need some kind of finite positive moments. We refer to a more detailed discussion directly after Theorem \ref{ergodicitytheorem}. 
To our knowledge - apart from dimension $d=1$ in \cite{KULPAPER} - 
there are not exponential ergodicity results available in the literature with weaker moment hypotheses. 

\begin{hyp}[Moment condition]\label{hyp: moment condition}\hfill\\
We assume
\[
\int_{|z|>1} |z|^{\beta} \nu(\ud z) < \infty\quad  \textrm{ for some } \beta>0. 
\]
\end{hyp}

Since we consider a smooth exponentially stable dynamical system with a small random perturbation, it is natural to apply a linearization procedure, which makes it necessary to compare $X^{\e, x}_\cdot$ with 
a suitable linearized process $Y^{\e}_\cdot(x)$. As they have different drift terms, 
this comparison can hold only for short times. In addition, as explained in the introduction, 
$X^{\e, x}_\cdot$ can be understood as $Y^{\e}_\cdot(x)$ plus some short time 
error term, which turns out to be not of independent nature and therefore 
hard to treat in the total variation, since the analogous statement of Slutsky's lemma (for instance \cite{Klenke}, Section 13.2, Theorem 13.18) for the total variation distance is false in general. We are not aware of this result in the literature and hence provide a counterexample in Subsection~\ref{subsub:counter}. 
The resulting difficulty is overcome by a 
short-time local limit theorem. Such a result has been given 
in Theorem~3.1 in \cite{HOU} and requires some kind of regularization in terms of a sufficiently 
steep pole of the L\'evy measure at the origin. 
With this reasoning in mind it comes not as a surprise that our results are shown for a specific class of L\'evy processes with such a property. In what follows, we assume that the L\'evy process $L$ has no Gaussian component and its L\'evy measure belongs to the following class.


\begin{defn}[Locally layered stable L\'evy measure]\label{def:llslm89}\hfill\\
Let $\nu$ be a L\'evy measure on $(\mathbb{R}^d,\mathcal{B}(\mathbb{R}^d))$. 
Then $\nu$
is called a \textbf{locally layered stable L\'evy measure} with parameters 
$(\nu_0,\nu_\infty,\Lambda,q,c_0,\alpha)$ if the following is satisfied.
There exist $\sigma$-finite Borel measures $\nu_0$ and $\nu_\infty$ such that $\nu=\nu_0+\nu_\infty$, where $\nu_\infty$ is a finite measure with support contained in
$\{|z|>1\}$ and
\[
\nu_0(A)=\int_{\mathbb{S}^{d-1}} \Lambda(\ud \theta)\int_{0}^{1}\mathbf{1}_{A}(r\theta)q(r,\theta)\ud r\quad  \textrm{ for any } A\in \mathcal{B}(\mathbb{R}^d), \;0\notin \bar{A},
\]
where $\Lambda$ is a finite positive measure on $\mathbb{S}^{d-1}$ 
(the unit sphere on $\mathbb{R}^d$), and $q:(0, 1]\times \mathbb{S}^{d-1} \to (0,\infty)$ 
is a locally integrable function for which
there exist
a positive function $c_0$ in $L^1(\Lambda)$
and a parameter $\alpha\in (0,2) $
such that 
\begin{equation*}
|r^{1+\alpha}q(r,\theta)-c_0(\theta)|\to 0, \quad \textrm{ as } r\to 0
\end{equation*}
for $\Lambda$-almost all $\theta\in \SSS^{d-1}$.
A pure jump L\'evy process with a locally layered stable L\'evy measure is called a \textbf{locally layered stable L\'evy process}.
\end{defn}

\noindent This notion generalizes naturally the notion of a layered stable L\'evy measure (and the respective L\'evy process) introduced in Definition 2.1 of \cite{HOU} to all L\'evy measures 
for which Theorem~3.1 (Short-range behavior) remains valid under Hypothesis~\ref{hyp: moment condition}. 
They include more general tail measures $\nu_\infty$ than layered stable L\'evy measures given in \cite{HOU}, such as tempered stable L\'evy measures defined in \cite{Ros07} and Lamperti stable L\'evy measures \cite{CPP10}. The following more restrictive notion is tailor-made to strengthen the result of Theorem~3.1 in \cite{HOU} to the convergence in the total variation distance which turns out to be crucial in the proof of Proposition~\ref{prop: lic}. In addition, in Theorem~\ref{Ext PengZhang} in Appendix~\ref{Appendix D} we extend Theorem~4.1 of \cite{Peng} and show that under Hypothesis~\ref{hyp: moment condition} the system \eqref{dde1} is strongly ergodic under the total variation distance.

\begin{defn}[Strongly locally layered stable L\'evy measure]\label{hyp: layered} \hfill\\
Let $\nu$
be  a locally layered stable L\'evy measure with parameters 
$(\nu_0,\nu_\infty,\Lambda,q,c_0,\alpha)$.
If, in addition, we have the small jump symmetry
\begin{equation}\label{eq:locsym}
q(r,\theta)=q(r,-\theta)\quad  \textrm{ for any } r\in (0,1), \;\theta\in \mathbb{S}^{d-1},
\end{equation}
 the uniform convergence
 \begin{equation}\label{layered0}
\sup_{\theta\in \SSS^{d-1}}|r^{1+\alpha}q(r,\theta)-c_0(\theta)|\to 0, \quad \textrm{ as } r\to 0,
\end{equation}  
and the gradient estimate
\begin{equation}\label{eq: gradientbound}
|\nabla\log q(r,\theta)| \lqq  C_1 r^{-1}\quad \textrm{ for some } C_1>0 \textrm{ and all } r\in (0,1),
\end{equation}
we \textit{call $\nu$ a} \textbf{strongly locally layered stable L\'evy measure} with parameters 
$(\nu_0,\nu_\infty,\Lambda,q,c_0,\alpha)$.
A pure jump L\'evy process with a strongly locally layered stable L\'evy measure is called a \textbf{strongly locally layered stable L\'evy process}.
\end{defn}

\begin{rem}
Examples of such processes are symmetric $\alpha$-stable L\'evy processes (see \cite{Appl, Sa}), symmetric tempered $\alpha$-stable process \cite{Ros07} and symmetric Lamperti $\alpha$-stable 
processes. 
\end{rem}

\begin{hyp}[Regularity]\label{hyp: regularity}\hfill\\
We assume that  the L\'evy process $L$ has no Gaussian component and its L\'evy measure $\nu$ is strongly locally layered stable with parameters $(\nu_0,\nu_\infty,\Lambda,q,c_0,\alpha)$.
\end{hyp}

\noindent 
In the sequel, we define sufficient conditions for an abrupt convergence of $X^{\e,x}_t$ to its unique limiting distribution $\mu^\e$ as $\e \to 0$ in the total variation distance.\\

\noindent \textbf{b) The total variation distance $\norm{\cdot}$: }
Before we  introduce the concept of cutoff  formally,  we recall the notion of the total variation distance.
Given two probability measures $\mathbb{P}$ and $\mathbb{Q}$ which are defined on the same measurable space $\left(\Omega,\mathcal{F}\right)$, denote the total variation distance between $\mathbb{P}$ and $\mathbb{Q}$ as follows
\[
\norm{\mathbb{P}-\mathbb{Q}}:=\sup_{A\in \mathcal{F}}{|\mathbb{P}(A)-\mathbb{Q}(A)|}.\] 
For simplicity, in the case of two random vectors $X$ and $Y$ defined on the same probability space $\left(\Omega,\mathcal{F}, \mathbb{P}\right)$ we use the following notation  for its total variation distance
$$\norm{X-Y}:=\norm{\mathcal{L}(X)-\mathcal{L}(Y)}, $$
where $\mathcal{L}(X)$ and $\mathcal{L}(Y)$ denote the law under $\mathbb{P}$ of the random vectors $X$ and $Y$, respectively. 
For the sake of intuitive reasoning and in a conscious abuse of notation we write $\norm{X- \mu_Y}$ instead of $\norm{X- Y}$, where $\mu_Y$
is the distribution of the random vector $Y$.
For a complete understanding of the total variation distance, we refer to Chapter 2 of the monograph of Kulik \cite{KUL} and the references therein. \\

\noindent \textbf{c) Exponential ergodicity with smooth limiting measure.} 

As we mentioned before, we are interested on the  cutoff  under the  total variation distance, which is a rather robust distance for continuous distributions and rather sensitive for discrete distributions. 
It is therefore natural to assume the following additional hypothesis which
with the help of Hypothesis~\ref{hyp: regularity}
 yields smooth densities for the finite time marginals and 
the limiting distribution of \eqref{dde1}. 

\begin{hyp}[Equator condition \cite{Simon}]\label{hyp: blumental} \hfill\\
Let $\nu$ satisfy Hypothesis~\ref{hyp: regularity}.
The support of the  measure $\Lambda$ is not contained in any proper subspace of $\mathbb{R}^{d}$ intersected with $\SSS^{d-1}$.  
Furthermore, we assume 
\begin{equation}\label{essinf}
\underbar{c}_0:=\essinf_{\theta\in \SSS^{d-1}} c_0(\theta)>0,
\end{equation}
where the essential infimum is understood with respect to the spectral measure $\Lambda$ of~$\nu$. 
\end{hyp}

\noindent The equator condition is motivated by the definition given in Simon \cite{Simon}, p.4.  It  provides a non-degeneracy condition on the support of $\Lambda$ on $\mathbb{S}^{d-1}$. 

\begin{rem}
It is not hard to see that Hypothesis~\ref{hyp: blumental} \eqref{essinf} implies
\begin{equation}
\inf_{\bar{v}\in \mathbb{S}^{d-1}}\int_{\mathbb{S}^{d-1}}
\cos^2(\sphericalangle(\bar{v}, \theta))\Lambda(\ud \theta) >0,\quad
\textrm{
 where }\quad\cos(\sphericalangle(\bar{v}, \theta))=\<\bar{v},\theta\>.
\end{equation}
\end{rem}

\noindent
The following lemma links Definition~\ref{hyp: layered} and Hypothesis~\ref{hyp: blumental} 
to the celebrated Orey-Masuda regularity condition, which is used in the proof of Proposition~\ref{prop: stc}. 

\begin{lem}[Orey-Masuda's cone condition]\label{lem oreymasuda}\hfill\\
Let $\nu$ be a strongly locally layered stable L\'evy measure on $\mathbb{R}^d$ with parameters 
$(\nu_0,\nu_\infty,\Lambda,q,c_0,\alpha)$ for $\alpha \in (0,2)$.
Under Hypothesis~\ref{hyp: blumental} 
there exist positive constants $c_{\sphericalangle}$ and $C_{\sphericalangle}$
such that for all $v\in \RR^d$ with $|v|> C_{\sphericalangle}$ we have 
\[
\int_{|\< v, z\>| \lqq 1} |\< v, z\>|^2 \nu(\ud z) \gqq 
\int_{|\< v, z\>|\lqq 1} |\< v, z\>|^2 \nu_0(\ud z)
\gqq c_{\sphericalangle} |v|^{\al}. 
\]
\end{lem}

\begin{proof}
Observe
\[
\begin{split}\label{layered-masuda}
\int_{|\< v, z\>|\lqq 1} |\< v, z\>|^2 \nu(\ud z)
& \gqq \int_{|\< v, z\>|\lqq 1, |z|\lqq 1} |\< v, z\>|^2 \nu(\ud z)\\
&\gqq |v|^2\int_{\SSS^{d-1}}\int_{0}^{1} 
r^2 \< \bar{v}, \theta \>^2
 \mathbf{1}\{{r|v||\cos (\sphericalangle(\bar{v}, \theta))|\lqq 1}\} q(r,\theta)\ud r\Lambda(\ud \theta),
\end{split}
\]
where $\bar{v}=\nicefrac{v}{|v|}$, 
$r=|z|$
and $\theta=\nicefrac{z}{r}$.
By \eqref{layered0} and \eqref{essinf} there exists $r_0>0$ (without loss of generality  $r_0\lqq 1$)  such that 
\[
q(r,\theta)\gqq \frac{\underbar{c}_0}{2r^{1+\alpha}}\quad \textrm{ for any } r\in (0,r_0).
\]
Consequently,  for $|v|>\nicefrac{1}{r_0}\gqq  1$ we have
\[
\begin{split}
&|v|^2\int_{\SSS^{d-1}}\int_{0}^{1} 
r^2 \< \bar{v}, \theta \>^2
 \mathbf{1}\left\{{|\cos (\sphericalangle(\bar{v}, \theta)|\lqq \frac{1}{r|v|}}\right\} q(r,\theta)\ud r\Lambda(\ud \theta)\\
&\qquad\gqq 
 |v|^2\int_{\SSS^{d-1}}\int_{0}^{\nicefrac{1}{|v|}} 
r^2 \cos^2(\sphericalangle(\bar{v}, \theta)) q(r,\theta)\ud r\Lambda(\ud \theta)\\ 
&\qquad\gqq 
\Big(\frac{\underbar{c}_0 }{2(2-\alpha)}
\inf_{\bar{v}\in \SSS^{d-1}}\int_{\SSS^{d-1}}
\cos^2(\sphericalangle(\bar{v}, \theta))\Lambda(\ud \theta)\Big)|v|^\alpha,
\end{split}
\]
which combined with Hypothesis~\ref{hyp: blumental} finishes the proof.
\end{proof}

The following result is a slight generalization of 
Theorem~4.1 in \cite{Peng} and guarantees that under Hypotheses~\ref{hyp: potential},~
\ref{hyp: moment condition},
\ref{hyp: regularity} and~
\ref{hyp: blumental} the system \eqref{dde1} is strongly ergodic under the total variation distance.

\begin{thm}\label{ergodicitytheorem}
Assume Hypotheses~\ref{hyp: potential},~
\ref{hyp: moment condition},~\ref{hyp: regularity} and~\ref{hyp: blumental} for
$\alpha\in (0,2)$ and
$\beta>0$. 
Then for any $\e>0$, there exists a unique invariant distribution $\mu^\e$ and positive constants $C_\e$, $\theta_\e$ such that for all $x\in \RR^d$, the law of the unique strong solution $X^{\e,x}$ of \eqref{dde1} satisfies
\[
\norm{X^{\e, x}_t-\mu^\e}\leq C_\e e^{-\theta_\e t}(1+|x|^{1\wedge \beta})\quad \textrm{ for  any }\quad t\geq 0.
\]
\end{thm}
The proof is a direct corollary of Theorem~\ref{Ext PengZhang} given in Appendix \ref{Appendix D}. The tracking of the dependence $\e\mapsto (\theta_\e,C_\e)$ is typically hard to follow through the discretization procedure laid out by Meyn and Tweedie \cite{MeynTweedie}. In the special case of finite variation, the backtracking of $\e$ can be carried out partially, we refer to \cite{KULPAPER}. \\

We recall that in dimension $d=1$,  a classical result by Kulik (see Proposition~0.1 in \cite{KULPAPER}) implies that the solution of \eqref{dde1} enjoys exponential ergodicity without assumption \eqref{eq: gradientbound} and consequently Theorem~\ref{ergodicitytheorem} holds for general locally layered stable L\'evy measures in this case. Very recently, \cite{NersesyanRaquepas} contains exponential ergodicity by control theoretic methods for multidimensional compound Poisson noise with finite variance.
For higher dimensions, we use the sufficient conditions including
\eqref{eq: gradientbound} in \cite{Peng} and our generalizations of their results given in Appendix~\ref{Appendix D}. We point out that for the special case of symmetric $\alpha$-stable L\'evy processes,  
assumption \eqref{eq: gradientbound} is automatically satisfied
and \cite{WANGLETTERS} yields exponential ergodicity in any dimension.

\bigskip 

\subsubsection{\textbf{The main results: window cutoff (Thm. \ref{thm: main result}) and profile cutoff (Thm. \ref{thm: main result2})}} \hfill\\

\noindent 
Following \cite{BY} and the references therein, 
there are three notions of cutoff phenomenon with increasing strength.
The most restrictive notion is called profile cutoff which 
provides the precise asymptotic  shape of the collapse for the total variation distance.
Profile cutoff implies a weaker concept which is called window cutoff that states abrupt convergence within a precise time interval but losing the precise profile. Window cutoff is generalized  further to the notion of cutoff  in which we retain the abrupt convergence along time scale which corresponds to the center of the interval, however, without a quantification of the error.

\begin{defn}\label{def: cutoff}
For any $\e>0$ and $x\in \mathbb{R}^d$, let $X^{\e,x}$ be the solution of
\eqref{dde1} with a unique limiting distribution $\mu^\e$.
 We say that for $x\in \mathbb{R}^d$ the family $(X^{\e,x})_{\e\in (0,1]}$ exhibits
\begin{itemize}
\item[a)] a {\bf cutoff phenomenon} at the time scale $(t^x_{\e})_{\e\in(0,1]}$,  where 
$t^x_{\e}\to\infty$, as $\e\to 0$, if it satisfies
\begin{eqnarray*}
\lim\limits_{\e\rightarrow 0 }
\norm{{X}^{\e,x}_{\delta \cdot t^x_{\e}}-\mu^{\e}}
= \left\{ \begin{array}{lcc}
             1 &  \textrm{ if }  & \delta\in (0,1), \\
             \\ 0 & \textrm{ if } & \delta\in (1,\infty). \\
             \end{array}
   \right.
\end{eqnarray*}
\item[b)]  a {\bf window cutoff phenomenon} at the enhanced time scale
$(t^x_{\e}, w^x_{\e})_{\e\in(0,1]}$, 
where $t^x_{\e}\to \infty$ and 
\mbox{$\nicefrac{w^x_{\e}}{t^x_{\e}}\to 0$,} as $\e\to 0$, if it satisfies
\[
\lim\limits_{\rho \rightarrow -\infty}{\liminf\limits_{\e\rightarrow 0}
\norm{{X}^{\e,x}_{ t^x_{\e}+\rho\cdot w^x_\e}-\mu^{\e}}}=1\quad \textrm{ and } \quad
\lim\limits_{\rho \rightarrow \infty}{\limsup\limits_{\e\rightarrow 0}
\norm{{X}^{\e,x}_{ t^x_{\e}+\rho\cdot w^x_\e}-\mu^{\e}}}=0.
\]

\item[c)] a {\bf profile cutoff phenomenon} at the enhanced time scale
$(t^x_{\e}, w^x_{\e})_{\e\in(0,1]}$ with the profile function $G_{x}$, where $t^x_{\e}\to \infty$ and 
$\nicefrac{w^x_{\e}}{t^x_{\e}}\to 0$, as $\e \to 0$, if the limit 
\[
 G_x(\rho):=\lim\limits_{\e \rightarrow 0}\norm{{X}^{\e,x}_{ t^x_{\e}+\rho\cdot w^x_\e}-\mu^{\e}}
\]
is well-defined for all $\rho\in \mathbb{R}$ and $G_{x}$ satisfies 
\[
\lim\limits_{\rho \rightarrow -\infty}{G_x(\rho)}=1\quad \textrm{ and } \quad \lim\limits_{\rho \rightarrow \infty}{G_x(\rho)}=0.
\]
\end{itemize}
\end{defn}
\noindent The cut-off time scale $t^x_\e$ is sometimes referred to as the center of the cutoff window  and $w^x_\e$ as its width. As mentioned above  iii) implies ii) and ii) implies i).

The first main result of this study reads as follows. 
\begin{thm}[Generic window cutoff phenomenon]\label{thm: main result}\hfill\\
Assume Hypotheses~\ref{hyp: potential},~\ref{hyp: moment condition},~\ref{hyp: regularity} and ~\ref{hyp: blumental} are satisfied for some $\alpha\in (3/2,2)$ and
 $\beta>0$. 
For any $\e>0$ and $x\in \RR^d\setminus\{0\}$, let  $X^{\e,x}$ be the unique strong solution of
\eqref{dde1} with a unique limiting distribution $\mu^\e$.
Then the family $(X^{\e,x})_{\e\in (0,1]}$ exhibits a window cutoff phenomenon as $\e \to 0$ at the enhanced time scale $(t^x_\e,w^x_\e)$
given by
\begin{equation}\label{eq:tiempos}
t^{x}_\e=\frac{1}{\lambda_x}\ln\left(\nicefrac{1}{\e}\right)+
\frac{\ell_x-1}{\lambda_x}
\ln\left(\ln\left(\nicefrac{1}{\e}\right)\right)
\quad \textrm{ and } \quad
w^{x}_\e=\frac{1}{\lambda_x}+o_\e(1),
\end{equation}
where $\lambda_x>0$ and $\ell_x \in \{1,\ldots,d-1\}$ are the constants appearing in
the Hartman-Grobman decomposition of Lemma~\ref{asymp}.
\end{thm}
Note that $x=0$ in Theorem \ref{thm: main result} is essential. 
\begin{rem}
For $x=0$, there is no cutoff phenomenon since the linearization vanishes and 
intuitively  cannot compete with the ergodicity. For details see Remark \ref{rem:nocutoff0}. For a complete discussion of the easier case of the Wasserstein distance, we refer to Section 3.2 in~\cite{BHP}.
\end{rem}

\noindent
Assume the hypotheses of Theorem~\ref{thm: main result} are satisfied for some $x\in \RR^d\setminus\{0\}$. Let $v(t,x)=\sum_{k=1}^{m}e^{i \theta^k_x t}v^k_x$ and $\la_x$, $\ell_x$, $\theta^1_x,\ldots, \theta^m_x$ and $v^1_x,\ldots, v^m_x$  given  in Lemma~\ref{asymp}.
We define the $\omega$-limit set for the dynamics of $(v(t,x))_{t\gqq 0}$ by
\begin{equation}\label{eq: omegasetx}
\omega(x):=\{v\in \mathbb{R}^d: \textrm{ there exists a sequence $(t_j)\to \infty$ and } \lim\limits_{j\to \infty}v(t_j,x)=v\},
\end{equation}
which due to the left-hand side of \eqref{eq: outsidezero} does not include the null vector, i.e. $ 0\not \in \omega(x)$.

\begin{rem}
Note that $\omega(x)\neq \emptyset$. Indeed,  
a Cantor diagonal argument 
for any limiting sequence in \eqref{eq: deterministic exponential convergence} 
yields the existence of a subsequence $(t_j)_{j\in \NN}$ with $t_j\ra \infty$, as $j \ra \infty$, such that for any $k=1,\ldots,m$ the limit 
$\lim\limits_{j\to \infty }e^{it_j\theta^k}= \vartheta_k$ exist. Moreover, $|\vartheta_k|=1$ for all $k$.
Since $v^1,\ldots,v^m$ are linearly independent vector in $\CC^d$, we deduce 
$v=\sum_{j=1}^{m}\vartheta_j v^j\in \omega(x)$.
\end{rem}

\noindent In an abuse of notation let $Z_\infty$ denote a parametrization of the unique 
invariant distribution of the Ornstein-Uhlenbeck process
\[
\ud Z_t=-Db(0)Z_t\ud t +\ud L_t.
\]
We have the following characterization of profile cutoff.

\begin{thm}[A dynamical characterization of a profile cutoff phenomenon]\label{thm: main result2}\hfill\\
Assume the hypotheses of Theorem~\ref{thm: main result} are satisfied for some $x\in \RR^d\setminus\{0\}$. 
Recall the $\omega$-limit set $\omega(x)$ given in \eqref{eq: omegasetx}.
Then the family $(X^{\e,x})_{\e\in (0,1]}$ exhibits a profile cutoff phenomenon as $\e \to 0$ at the enhanced time scale $(t^x_\e,w^x_\e)$
given by Theorem~\ref{thm: main result} with profile function 
\[
G_{x}(\rho)=
\norm{\left(
e^{-\rho}\cdot \frac{e^{-\la_x \tau_x}}{
\lambda_x^{\ell_x-1}}
v+ Z_\infty \right)- Z_\infty}
\quad \textrm{ for any } \rho\in \RR,\;  v\in \omega(x)
\]
if and only if for any  $a>0$ the map
\begin{equation}\label{eq:charac}
\omega(x)\ni v\mapsto 
\norm{(a v+ Z_\infty) - Z_\infty}\quad
\textrm{ is constant}.  
\end{equation}
\end{thm}
\noindent Observe that $\omega(x)=\{v_x\}$ immediately implies profile cutoff by 
the preceding theorem. The latter, indeed, is satisfied in the subsequent case of a gradient potential.

The following special case of gradient systems is particularly of interest in applications, such as for instance the Fermi-Ulam-Pasta-Tsingou-potential treated in Subsection \ref{ss:Fermi}.
\begin{cor} 
\label{cor:profile20}
Let the assumptions of Theorem~\ref{thm: main result} be satisfied and 
assume $b(x)=\nabla \vV (x)$, $x\in \RR^d$, for a potential function $\vV:\RR^d\to [0,\infty)$.
Then the family $(X^{\e,x})_{\e\in (0,1]}$ exhibits a profile cutoff as $\e \to 0$ at the enhanced time scale $(t^x_\e,w^x_\e)$
given by
\[
t^{x}_\e=\frac{1}{\lambda_x}\ln\left(\nicefrac{1}{\e}\right)
\quad \textrm{ and } \quad
w^{x}_\e=\frac{1}{\lambda_x}+o_\e(1),
\]
where $\lambda_x>0$ and $\tau^x$ are the positive constants in the Hartman-Grobman decomposition of Lemma~\ref{asymp} such that
\begin{equation}\label{limitgrad}
\lim\limits_{t\to \infty }e^{\lambda_x t}\varphi^{x}_{t+\tau^x}=v^{x}\not=0
\end{equation}
and the profile function is given by
\[
G_x(\rho)=\norm{(e^{-\rho}\cdot e^{-\la_x \tau_x}v^x+Z_\infty)-Z_\infty}, \qquad \rho \in \RR.
\]
\end{cor}

\begin{rem}
Note that the dependence of $\lambda_x$ of $x$
can be complicated, however,
 it is rather weak in the following qualitative sense: $\lambda_x=\lambda$ for Lebesgue almost every $x \in\RR^d$, where $\lambda$ is the smallest eigenvalue of the positive definite symmetric matrix $D^2 \vV(0)$. 
\end{rem}
We give a more general sufficient conditions for the existence of a cutoff profile in terms of 
a symmetry condition. 
\begin{cor}\label{cor: profilelineal}
Assume the hypotheses of Theorem~\ref{thm: main result} are satisfied for some $x\in \RR^d\setminus\{0\}$. If there exists an invertible $d\times d$-square matrix $M$  such that the distribution of $MZ_\infty$ is rotationally invariant and the image set satisfies $M\omega(x)\subset \{|z|=r\}$ for some $r=r_x>0$, then
the family $(X^{\e,x})_{\e\in (0,1]}$ exhibits a profile cutoff phenomenon as $\e \to 0$ at the enhanced time scale $(t^x_\e,w^x_\e)$.
\end{cor}
In the Gaussian case we have the following picture. 
\begin{rem}
For the non-degenerate Gaussian case we refer to Lemma A.2 in \cite{BJ}. There, the law of
$Z_\infty$ is $\nN(0,\Sigma)$, where $\Sigma$ satisfies
\[
Db(0)\Sigma+\Sigma Db(0)^*=I_d.
\]
The choice of $M=\Sigma^{-1/2}$ yields that $MZ_\infty\stackrel{d}= \nN(0,I_d)$ is rotationally invariant. 
Hence the sphere condition $M \omega(x)\subset \{|z|=r\}$ for some $r=r_x>0$ is equivalent to the profile cutoff, see Corollary 2.11 in \cite{BJ}.
 However, in the generic L\'evy case, no symmetry
on the law of $Z_\infty$ can be expected. 
Note that we always find an invertible bi-measurable map $\tT:\RR^d\to \RR^d$
such that the push-forward $\tT(Z_\infty)$ is rotationally invariant (for instance $\nN(0,I_d)$),
however, it is  highly nonlinear and irregular, and therefore the proof of Corollary~\ref{cor: profilelineal} breaks down.\\
\end{rem}
A sufficient condition for the hypotheses of Corollary \ref{cor: profilelineal} to be satisfied 
can be given in terms of the following density condition on the invariant limiting measure of the Ornstein-Uhlenbeck process $Z$. 
\begin{cor}[Geometric profile characterization under rotational invariant $Z_\infty$]\label{cor:charprofile}\hfill\\
Assume the hypotheses of Theorem~\ref{thm: main result} are satisfied for some $x\in \RR^d\setminus\{0\}$. If in addition, the law of $Z_\infty$ is rotationally invariant 
and its density  $f\in \mathcal{C}^1(\RR^d,(0,\infty))$  is unimodal in the sense that 
$f(z) = g(|z|)$
for some function $g\in \mathcal{C}^1((0, \infty), (0, \infty))$  with  $g'(s) < 0$ for all $s > 0$  and  $g'\in L^1(\RR^d)$.
Then the image set satisfies $\omega(x)\subset \{|z|=r\}$ for some $r=r_x>0$ if and only if
the family $(X^{\e,x})_{\e\in (0,1]}$ exhibits a profile cutoff as $\e \to 0$ at the enhanced time scale $(t^x_\e,w^x_\e)$.
\end{cor}
In case of a pure jump L\'evy noise $L$ the sufficient condition of Corollary \ref{cor: profilelineal} 
can be almost characterized (up to a non-resonance condition) in terms of the following normal growth condition, 
which is discussed in detail in \cite{BHP}. 
\begin{rem}[Generic normal growth profile characterization]\label{rem:linalg}
In the sequel, we characterize when the function
\[
\omega(x)\ni u\mapsto |u|
\]
is constant for the generic case of the setting in Corollary~\ref{cor:charprofile}. We enumerate $v^1, \dots, v^m$ given in Lemma~\ref{asymp} as follows. Without loss of generality we assume that $\theta^1=0$. Otherwise we take $v^1=0$ and eliminate it from the sum $\sum_{k=1}^{m}e^{i\theta^k t}v^k$. Without loss of generality let $m=2n+1$ for some $n\in \NN$.
We assume that 
$v^k$ and $v^{k+1}=\bar v^k$ are complex conjugate for all even number $k\in \{2,\ldots,m\}$.
For $k\in \{2,\ldots,m\}$ we write $
v^k=\hat v^k+i\check v^k$ where 
$\hat v^k,\check v^k\in \RR^d$. 
\begin{enumerate}
 \item  If the real parts and the imaginary parts of the (complex) vectors $v^{2},v^4, \dots, v^{2n}$ in the Hartman-Grobman Lemma~\ref{asymp} form an orthogonal family and $|\mathsf{Re}(v^{2k})| =|\mathsf{Im}(v^{2k})|$ for all~$k$.  
Then Lemma~E.1 in \cite{BHP} implies that $\omega(x)\subset \{|z|=r\}$ for some $r=r_x>0$ and 
hence Corollary~\ref{cor: profilelineal} yields a profile cutoff. 
 \item Assume the angles $\theta^2,\theta^4, \dots, \theta^{2n}$ given in the Hartman-Grobman Lemma~\ref{asymp} are rationally independent from $2\pi$. If $\omega(x)\subset \{|z|=r\}$ for some $r=r_x>0$ then Lemma~E.2 implies that the real parts and the imaginary parts of the (complex) vectors $v^{2},v^{4}, \dots, v^{2n}$ in the Hartman-Grobman Lemma~\ref{asymp} form an orthogonal family and $|\mathsf{Re}(v^{2k})| =|\mathsf{Im}(v^{2k})|$ for all $k$. 
 \end{enumerate}
\end{rem}

\begin{proof}[Proof of Corollary~\ref{cor: profilelineal}:]
We apply the characterization given in Theorem~\ref{thm: main result2}.
Let $v_1,v_2\in \omega(x)$ and $a>0$. For $M$ given in the statement,
we have $|Mv_1|=|Mv_2|=r$. Then there exists an orthogonal matrix $\oO$ such that $\oO(Mv_1)=Mv_2$.
Theorem~5.2 of \cite{DEVLUG} and $\oO$,$M$ being invertible implies
\begin{align*}
\norm{\big(a v_1+ Z_\infty\big) - Z_\infty}&=\norm{\big(a Mv_1+ MZ_\infty\big) - MZ_\infty}\\
&=\norm{\big(a \oO (M v_1)+ \oO MZ_\infty\big) - \oO MZ_\infty}.
\end{align*}
Since $\oO(Mv_1)=Mv_2$ and $\oO$ is orthogonal, the rotational invariance of  $MZ_\infty$ implies
\begin{align*}
\norm{\big(a \oO (M v_1)+ \oO MZ_\infty\big) - \oO MZ_\infty}=\norm{\big(a Mv_2+ MZ_\infty\big) -MZ_\infty}.
\end{align*}
Again, Theorem~5.2 of \cite{DEVLUG} yields 
\begin{align*}
\norm{\big(a M v_1+ MZ_\infty\big) - MZ_\infty}=\norm{\big(a v_2+ Z_\infty\big) -Z_\infty}.
\end{align*}
Combining the preceding equalities we obtain 
\[
\norm{\big(a v_1+ Z_\infty\big) -Z_\infty}=\norm{\big(a v_2+ Z_\infty\big) -Z_\infty}
\]
for any $v_1,v_2\in \omega(x)$ and $a>0$ which yields \eqref{eq:charac} and hence the desired profile cutoff.
\end{proof}

\begin{proof}[Proof of Corollary \ref{cor:charprofile}:]
By Corollary \ref{cor: profilelineal} ($M = I_d$) it is enough to prove the converse implication.
Since the family $(X^{\e,x})_{\e\in (0,1]}$ exhibits a profile cutoff as $\e \to 0$ at the enhanced time scale $(t^x_\e,w^x_\e)$, Theorem~\ref{thm: main result2} implies for all $a>0$ that the map
$v\in \omega(x)\mapsto 
\norm{(a v+ Z_\infty) - Z_\infty}$
is constant.
Since the law of $Z_\infty$ is rotationally invariant, we have
\[
\norm{(a v+ Z_\infty) - Z_\infty}=\norm{(a|v|e_1+ Z_\infty) - Z_\infty},
\]
where $e_1=(1,0,\ldots,0)^*$.
By Lemma \ref{lem:monotoniaTV} in Appendix~\ref{Appendix C} we have that 
$\omega(x)\ni v\mapsto a|v|$ is constant. 
That is to say,  $\omega(x)\subset \{|z|=r_x\}$ for some $r_x>0$. This finishes the proof.
\end{proof}

\bigskip 

\subsection{\textbf{Examples}}

\subsubsection{\textbf{More general linear dynamics}}\hfill\\

\noindent When the vector field is given by $b(x)=Qx$, $x\in \RR^d$ for  a general deterministic  $d\times d $ matrix $Q$ whose eigenvalues have positive real parts, the cutoff phenomenon is completely discussed in \cite{BP}, Theorem~2.3 under Hypothesis~(H), which is covered by Hypothesis~\ref{hyp: blumental}.
It is well-known that such linear systems are more general than linear systems satisfying Hypothesis~\ref{hyp: potential}.
For instance, the classical linear oscillator with friction $\gamma>0$ has negative real parts in $(-\infty,-\gamma/2]$  but fails to be coercive, \cite{BHP}.

The case of pure Brownian motion is covered in detail in Section 3.3 in \cite{BJ1}.
For degenerate driving noise processes $L$ and general cutoff results in the Wasserstein distance, we refer to \cite{BHP}. There, complex systems of linear oscillators in a thermal bath are covered. 

\bigskip 

\subsubsection{\textbf{Gradient systems: Fermi-Ulam-Pasta-Tsingou}}\label{ss:Fermi}\hfill\\

\noindent In the sequel, we consider the generalized Fermi-Ulam-Pasta-Tsingou potential
\cite{FERMI,DAUX}
\begin{equation}
\vV(x)=|Ax|^2/2+|Bx|^4/4+\eta(x), \qquad x\in \RR^d, 
\end{equation}
where $A$ and $B$ are $d\times d$ deterministic matrices satisfying for some $\delta_1>0$ 
\begin{equation}\label{eq:AB matrices}
\<Ax,x\>\gqq \delta_1 |x|^2
\qquad \textrm{ and } \qquad
\<Bx,x\>\gqq 0\qquad \textrm{ for all }x\in \RR^d
\end{equation}
and some $\eta:\RR^d \to \RR$ with
$\eta\in C^2_b$, $\nabla \eta(0)=0$, and for $H_\eta$ being the Hessian of $\eta$ 
\[
\<H_\eta(x)y,y\>\gqq -\delta_2 |y|^2 
\]
for all $x,y\in \RR^d$ and some $\delta_2<\delta_1$.
Note that $\eta$ needs not be convex.
Set $b(x)= \nabla \vV(x), x\in \RR^d$. 
Then for all $x\in \RR^d$ we have
\begin{equation}\label{eq:exampleb}
b(x)=A^*A x+\<Bx,Bx\>B^*Bx+\nabla \eta(x)
\end{equation}
and  satisfies Hypothesis~\ref{hyp: potential}.
Indeed, the Jacobian of $b$ at $x$ is given by
\[
Db(x)=A^*A+3\<Bx,Bx\>B^*B+H_{\eta}(x),
\]
where $H_\eta(x)$ denote the Hessian matrix at $x$.
By \eqref{eq:AB matrices} we obtain
for any $x,y\in \RR^d$
\[
\<y,Db(x)y\>=
\<Ay,Ay\>+3\<Bx,Bx\>\<By,By\>
+\<H_\eta(x)y,y\>
\gqq \delta|y|^2,
\]
where $\delta=\delta_1-\delta_2>0$. Hence the vector field $b = \nabla \vV$ satisfies Hypothesis~\ref{hyp: potential}. We consider the solution of \eqref{dde1.1} with vector field $b = \nabla \vV$. Note that in this case $|\omega(x)| = 1$, where $\omega(x)$ is given in \eqref{eq: omegasetx}. 

Note that generically equation \eqref{dde1}  
does not have an explicitly known solution for the Kolmogorov forward equation of the densities, not even in simplest case of 
$d=1$, $L=W$, a standard Wiener process, $A=B=1$ and $\eta\equiv 0$.
While for any dimension $d$, $L=W$ a standard Brownian motion the invariant density is well-known to be proportional to 
$
\exp\left(-\nicefrac{2\vV(x)}{\e^2}\right)$, 
for a complete discussion, see for instance Section 2.2 in \cite{Siegert}. 
For the case of dimension $d=1$, nonlinear $b$ satisfying Hypothesis~\ref{hyp: potential}, $L=W$ a standard Brownian motion the authors prove profile cutoff for \eqref{dde1} in \cite{BJ}. For higher dimensions, window cutoff is established in this case and the existence of profile cutoff is characterized, we refer to \cite{BJ1}. We remark that the authors  strongly use the hypo-ellipticity property and the resulting regularization by the generator of the Brownian diffusion.

For a strongly locally layered stable noise $L$ satisfying Hypotheses~\ref{hyp: moment condition}, \ref{hyp: regularity} and \ref{hyp: blumental},   Corollary~\ref{cor:profile20} implies for $\al>3/2$ the presence of a cutoff profile. In particular, the system exhibits cutoff in the sense of equation (18.3) in Chapter 18 of the monograph \cite{LPW} as follows  
 \[
\lim_{\e\to 0}\frac{{\tT^{x,\e}_{\mathrm{mix}}}(\eta)}{{\tT^{x,\e}_{\mathrm{mix}}}(1-\eta)}=1
\]
for any $\eta\in (0,1)$, where the mixing time is given by
\begin{align*}
{\tT^{x,\e}_{\mathrm{mix}}}(\eta)=\inf\{t\gqq 0: \norm{X^\e_t(x)-\mu^\e}\lqq \eta\}.\\
\end{align*}

\bigskip 
\subsubsection{\textbf{Profile vs Window cutoff for nonlinear oscillations}}
\label{ss:nonlinearosc}\hfill\\

\noindent In the sequel we analyze a class of nonlinear oscillators for which the existence of a cutoff profile is studied in detail.
We consider the nonlinear system \eqref{dde1} in
 $\RR^2$,
where  $b:\RR^2 \to \RR^2$ is given by
\begin{equation*}
b(x_1, x_2)=\left(
\begin{array}{c}
\eta x_2+\partial_1 \hH(x_1,x_2) \\
-\eta x_1+\partial_2 \hH(x_1,x_2)
\end{array}
\right), 
\end{equation*}
for  some $\eta\in \RR$, 
\begin{align*}
\hH(x_1,x_2)=\delta_1x^2_1+\delta_2 x^2_2+\gG(x_1,x_2)
\end{align*}
for any $x_1,x_2\in \RR$
and some positive constant $\delta_1,\delta_2$, and $\gG\in C^2(\RR^2,\RR)$.
Assume that $b(0,0) = (0,0)^*$. 
We verify that $b$ is 
a non-gradient vector field. 
The Jacobian matrix of $b$ is given by 
\begin{equation}\label{eq:Jacobian}
Db(x_1, x_2)
 =\left(
\begin{array}{cc} 
2\delta_1 +\partial_{11} \gG(x_1,x_2) &\eta+ \partial_{12} \gG(x_1,x_2) \\
-\eta +\partial_{12} \gG(x_1,x_2)  & 2\delta_2+\partial_{22} \gG(x_1,x_2)
 \end{array}
 \right).
\end{equation}
Since for any $\eta\neq 0$ the Jacobian $Db$ matrix is asymmetric, there is no $\mathcal{C}^2$-function $\vV:\RR^2\to \RR$ such that 
$b(x_1, x_2) =\nabla \vV(x_1, x_2)$ and consequently $b$ is non-gradient.
Under the assumption that
\begin{align*}
x_1^2 \partial^2_{11} \gG(u_1,u_2)+x_2^2 \partial^2_{22} \gG(u_1,u_2) 
+2x_1x_2\partial^2_{12} \gG(u_1,u_2)\gqq 
-\delta_3 (x^2_1+x^2_2)
\end{align*}
for some 
$\delta_3<2\min\{\delta_1,\delta_2\}$
and any $x_1,x_2,u_1,u_2\in \RR$, the vector field $b$ satisfies Hypothesis~\ref{hyp: potential}
\begin{align*}
(x_1,x_2)&Db(u_1, u_2)(x_1,x_2)^*\\ 
&=
2\delta_1 x_1^2 +2\delta_2 x_2^2+
x_1^2 \partial^2_{11} \gG(u_1,u_2)+x_2^2 \partial^2_{22} \gG(u_1,u_2) 
+2x_1x_2\partial^2_{12} \gG(u_1,u_2) \\
&\gqq (2\delta_1-\delta_3)x^2_1+(2\delta_2-\delta_3)x^2_2=\delta(x^2_1+x^2_2).
\end{align*}
For a rotationally invariant
$\alpha$-stable noise $L$ in $\RR^2$
with $\alpha>\nicefrac{3}{2}$, 
 Theorem~\ref{thm: main result} yields window cutoff.
In the sequel, we study the presence of a cutoff profile.
We claim that $Z_\infty$ is rotationally invariant. Indeed, there is $K_\alpha>0$ such that the characteristic function of $Z_\infty$ reads
\begin{align*}
z\mapsto &\exp(-K_\alpha\int_{0}^{\infty} |e^{-Db(0,0)t} z|^\alpha \ud t)
=\exp(-K_\alpha |z|^\alpha\int_{0}^{\infty} e^{-2\delta_1 \alpha t} \ud t)=\exp\left(-\frac{K_\alpha}{2\delta_1 \alpha} |z|^\alpha\right).
\end{align*}
For $a:=\partial_{11} \gG(0,0)=\partial_{22} \gG(0,0)$ and $\partial_{12} \gG(0,0)=0$ we have
\begin{equation}\label{eq:Jacobian1}
Db(0, 0)
 =\left(
\begin{array}{cc} 
2\delta_1 +a &\eta
 \\
-\eta   & 2\delta_2+a
 \end{array}
 \right).
\end{equation}
Assume a negative discriminant $
\Delta:=(2\delta_2-2\delta_1)^2-4\eta^2<0$
 and $\delta_1+\delta_2+a>0$.
Then the complex eigenvectors associated to the eigenvalues  
\[
\lambda_1=\delta_1+\delta_2+a+
\frac{\sqrt{\Delta}}{2}\quad  \textrm{ and }\quad
\lambda_2=\delta_1+\delta_2+a-\frac{\sqrt{\Delta}}{2}
\]
are 
given by
\[
v_1=\Big(1, \frac{2(\delta_2-\delta_1)+\sqrt{\Delta}}{2\eta}\Big)\quad  \textrm{ and }\quad
v_2=\Big(1, \frac{2(\delta_2-\delta_1)-\sqrt{\Delta}}{2\eta}\Big).
\]
The respective family real and imaginary part vectors are given 
\[
\hat v_1=\hat v_2=\Big(1,\frac{\delta_2-\delta_1}{\eta}\Big)\quad \textrm{and} \quad 
\check v_1=-\check v_2=\Big(0,\frac{\sqrt{|\Delta|}}{2\eta}\Big).
\]
\begin{enumerate}
\item \textbf{Nonlinear nongradient system with a cutoff profile:} For $\delta_1=\delta_2$ we obtain
 \[
e^{(2\delta_1+a) t} |e^{-Db(0,0)t}x|=|\oO(\eta t)x|=|x|
 \]
 for all $t\gqq 0$, where $\oO(\eta t)$ is an orthogonal matrix.
Therefore, whenever  $2\delta_1+a>0$, Corollary~\ref{cor: profilelineal} for $M=I_2$ yields profile cutoff.
\item \textbf{Nonlinear counterexample to a cutoff profile:}
Note that $\hat v_2$ is orthogonal to $\check v_2$ if and only if $\delta_1=\delta_2$.
Define $\theta_2=\arg(\lambda_2)$.
For $\delta_1<\delta_2$, 
Corollary \ref{cor:charprofile} and Remark \ref{rem:linalg} for $\theta_2\not \in \mathbb{Q}\cdot \pi $ yield the absence of a cutoff profile.
For further examples in the linear case for the Wasserstein distance we refer to \cite{BHP}.
\end{enumerate}

\bigskip 
\subsubsection{\textbf{{The shape of cutoff profiles: Gaussian vs $\alpha$-stable}}}\hfill\\

\noindent Since $Z_\infty$ is the limiting distribution as $t\to \infty$ of the Ornstein-Uhlenbeck process $(Z_t)_{t\gqq 0}$,
Lemma~\ref{lem oreymasuda} implies that $Z_\infty$ has a $\cC^\infty$ density $f_\infty$. 
In the sequel we study the unidimensional case. 
Theorem~53.1 in \cite{Sa} yields that
$f_\infty$ is unimodal with mode $\fm$, 
that is to say, it is increasing on $(-\infty, \fm)$ and decreasing on 
$(\fm,+\infty)$ and hence $f_\infty(\fm)>0$.
In the sequel, we determine the asymptotics of the profile function $\rho \mapsto \norm{(e^{-\rho}e^{-\la_x\tau_x}v^x+Z_\infty)-Z_\infty}$ in zero and  at infinity for some special cases.
The density $f_\infty$ of $Z_\infty$ is explicitly  accessible  only in a limited number of cases. 

We start with the asymptotics  for $\rho\ll -1$.
Without loss of generality we assume that $f_\infty$ is smooth. 
Then for any $z>0$ there exists $\fm_z\in (\fm-z, \fm)$ such that 
\begin{align}\label{eq:fasympinf}
1-\norm{(z+Z_\infty)-Z_\infty}&=\PP(Z_\infty\lqq   \fm_z)+
\PP(Z_\infty\gqq   \fm_z+z)\nonumber\\
&=\int_{-\infty}^{\fm_z}f_\infty(u)\ud u
+\int_{\fm_z+z}^{\infty}f_\infty(u)\ud u.
\end{align}
Assume that $f_\infty(u)>0$ for all $u\in \RR$.
By Scheff\'e's lemma for densities, see Lemma~3.3.1 in \cite{Reiss}, the left-hand side of the preceding equality tends to zero as $z\to \infty$. Hence the right-hand side implies $\fm_z\to -\infty$ and $\fm_z+z\to \infty$, as $z\to \infty$. We have
\begin{align*}
&1-\norm{(e^{-\rho}e^{-\la_x\tau_x}v^x+Z_\infty)-Z_\infty} \nonumber\\
&\qquad\qquad= F_{\infty}(
\fm(e^{-\rho}e^{-\la_x\tau_x}|v^x|))+
(1-F_{\infty}(
\fm(e^{-\rho}e^{-\la_x\tau_x}|v^x|)+e^{-\rho}e^{-\la_x\tau_x}|v^x|)),
\end{align*}
 which reduces in the symmetric case to
\begin{align*}
1-\norm{(e^{-\rho}e^{-\la_x\tau_x}v^x+Z_\infty)-Z_\infty}= 2(1-F_{\infty}(e^{-\rho}e^{-\la_x\tau_x}|v^x|/2)),
\end{align*}
where $F_\infty$ is the cumulative function of $Z_\infty$. \\

\noindent We compare the prototypical shapes of the tails of the profile functions. 
\begin{enumerate}
 \item[I)] For the \textit{symmetric $\alpha$-stable} process $L$ we obtain the exponential profile function
\begin{align*}
1-\norm{(e^{-(\rho+\la_x\tau_x)}v^x+Z_\infty)-Z_\infty}&\sim \frac{2^{\alpha+1}C_\alpha 
e^{(\rho+\la_x\tau_x)  \alpha}
}{|v^x|^\alpha}
 \propto e^{\rho \alpha}
,\quad \textrm{ as } \rho\to -\infty,
\end{align*}
where $C_\alpha$ is an explicit constant.
 \item[II)] The asymptotically doubly exponential shape of the profile for 
the case of \textit{Gaussian tails $F_\infty$} is discussed in Remark 2.3 in \cite{BA} which reads  in our setting as follows
\begin{align}\label{eq:formulagausiana}
2(1-F_{\infty}(e^{-(\rho+\la_x\tau_x)}|v^x|/2))& \sim \frac{4}{\sqrt{2\pi}|v^x|}
\exp\left(
-\exp
\big(-2(\rho+\la_x\tau_x) \big)|v^x|^2/8+
\rho+\la_x\tau_x
\right)
\end{align}
for $\rho\to -\infty$.
In particular, \eqref{eq:formulagausiana} yields the doubly exponential asymptotic ($\rho \to -\infty$) leading term
\[
\frac{2\sqrt{2}}{\sqrt{\pi}|v^x|}
\exp\left(
-\exp
\big(-2(\rho+\la_x\tau_x) \big)|v^x|^2/8
\right)
\propto
\exp\left(-K_x
\exp
\big(-2\rho \big)
\right)
\]  
for some positive $K_x$.
\end{enumerate}
\noindent We continue with the asymptotics at zero of the profile function and show 
\begin{align*}
\frac{\norm{(z+Z_\infty)-Z_\infty}}{|z|}\to f_\infty(\fm), \quad \textrm{ as } z\to 0,
\end{align*}
where $\fm=\frac{a}{b^\prime(0)}$, where $b$ is the vector field of \eqref{dde1.1} and  $(a,0,\nu)$ is the characteristic triplet of~$L$. 
By \eqref{eq:fasympinf} we have 
\begin{align*}
\norm{(z+Z_\infty)-Z_\infty}&=\int_{\fm_z}^{\fm_z+z}f_{\infty}(u)\ud u,
\end{align*}
where $\fm_z\in (\fm-z, \fm)$. 
Since $\fm_z\to \fm$ as $z\to 0$ and $\fm$ is a Lebesgue point of $f_\infty$, it follows
\[
\frac{\norm{(z+Z_\infty)-Z_\infty}}{z}\to f_\infty (\fm), \quad \textrm{ as } z\to 0.
\]
As a consequence of the preceding limit we obtain
\begin{align*}
\lim\limits_{\rho \to \infty}\frac{\norm{(e^{-(\rho+\la_x\tau_x)}v^x+Z_\infty)-Z_\infty}}{e^{-(\rho+\la_x\tau_x)}|v^x|}=f_\infty(\fm). 
\end{align*}
In particular, as $\rho\ra \infty$, the profile is asymptotically proportional to the respective Wasserstein profile \cite{BHP}. 

\bigskip 
\subsubsection{\textbf{Counterexample  to Slutsky's lemma in total variation distance}}
\label{subsub:counter}
\hfill\\

\noindent
The following example is
the main motivation for Hypothesis \ref{hyp: regularity}. It is given for completeness since we are not aware of a reference in the 
 literature.
It is based on private communication with 
professors M. Jara (IMPA) and R. Imbuzeiro Oliveira (IMPA).
\begin{lem}
Let $(U_n)_{n\in\NN}$ be a sequence of random variable with  the discrete uniform  distribution supported on the set 
$\{\nicefrac{j}{n}: j=1,\ldots, n\}$.
Let 
$(R_n)_{n\in \mathbb{N}}$ be a sequence of random variables independent of $(U_n)_{n\in\NN}$ with the continuous uniform distribution supported on $[0,a_n]$, where $(a_n)_{n\in \mathbb{N}}$ is any sequence of positive numbers such that $n\cdot a_n\to 0$ and $a_n\to 0$, as $n\to \infty$.
For each $n\in \mathbb{N}$, we define $X_n=U_n+R_n$ and $Y_n=-R_n$. 
Then we have:
\begin{enumerate}
\item $X_n$ and $Y_n$ are absolutely continuous with respect to the Lebesgue measure on $\RR$.
\item $\lim\limits_{n\to \infty}U_n\stackrel{d}=U$, where $U$ is   (continuously) uniformly distributed on $[0,1]$.
\item  $Y_n\to 0$, as $n\to
 \infty$
 in
probability.
\item 
  $\lim\limits_{n\to \infty}
 \norm{X_n-U}=0$.
 \item
 $\norm{(X_n+Y_n)-U}=1$ for all $n\in \mathbb{N}$.
\end{enumerate}
\end{lem} 
\begin{proof}
Items (1), (2), (3) and (5) are straightforward.
In the sequel we verify (4).
Since  $U_n$ and $R_n$ are independent, the  convolution formula yields that the density of $X_n$, $f_{n}$, is given by
\[
\RR\ni z\mapsto
f_{n}(z)=\frac{1}{n}\sum\limits_{j=1}^{n}g(z-\nicefrac{j}{n}), \quad \textrm{ where }\quad  g(z)=(\nicefrac{1}{a_n})\,\ind_{[0,a_n]}(z).
\]
First, for $z\lqq 0$, it follows that $f_n(z)=0$ for all $n\in \mathbb{N}$.
Next, for $z>1$ there exists $n_0=n_0(z)\in \mathbb{N}$ such that $0<a_n<z-1$ for all $n\gqq n_0$. Hence,
for all $j=1,\ldots, n$ we have 
 $0<a_n<z-\nicefrac{j}{n}$ for all $n\gqq n_0$. Consequently, 
  $f_n(z)=0$ for all $n\gqq n_0$.
We continue with the case $z\in (0,1]$.
Then there exists $n_1:=n_1(z)\in \mathbb{N}$ such that $\nicefrac{z}{2}<z-a_n<z$ for all $n\gqq n_1$.
Then we have for all $n\gqq n_1$
\[
f_n(z)=\frac{1}{n}\sum\limits_{j= \lceil n(z-a_n)\rceil 
}^{\lfloor nz \rfloor}g(z-\nicefrac{j}{n})=
\frac{na_n}{na_n}+\frac{C_n(z)}{na_n}
=1+\frac{C_n(z)}{na_n},
\]
where $0<C_n(z)\lqq 2$. Since $na_n\to \infty$, $n\to \infty$, we have  $f_n(z)\to 1$, as $n\to \infty$.
In summary, it is shown for all $z\in \RR$ that $f_n(z)\to \ind_{(0,1]}(z)$, as $n\to \infty$.
Scheff\'e's lemma for densities implies 
$\norm{X_n-U}\to 0$, as $n\to \infty$.
\end{proof}

\bigskip

\subsection{\textbf{{Global steps of the proofs of Theorem~\ref{thm: main result} and Theorem~\ref{thm: main result2}}}}\label{subsec: outline}
\hfill\\

\noindent The fundamental idea of the proofs of Theorem~\ref{thm: main result} and Theorem~\ref{thm: main result2} is to carry out a quantitative asymptotic expansion in $\e$ by probabilistic
methods. It turns out that the hyperbolic contracting nature of the underlying deterministic dynamics 
$\varphi^{x}$ can be used to show that the correct first order expansion of $X^{\e,x}$
of the sense of Freidlin-Wentzell \cite{FW} Chapter~2.2
 given by 
the inhomogeneous Ornstein-Uhlenbeck  defined in \eqref{eq: Yxte} provides an asymptotic 
description of $X^{\e,x}$ which is effective for time scales beyond
the cutoff time scale.

\bigskip 
\subsubsection{\textbf{Freidlin-Wentzell first order expansion}}\hfill\\

\noindent It is not hard to see that for any $\eta>0$ and $t\gqq 0$ the law of large numbers implies
\begin{equation}\label{eq:zeroapprox}
\PP\Big(\sup_{0\lqq s\lqq t}|X^{\e,x}_s-\varphi^x_s|\gqq \eta\Big)\to 0,\quad \textrm{ as } \e\to 0.
\end{equation}
In the sequel, we analyze the asymptotic fluctuations of $X^{\e,x}_t-\varphi^{x}_t$.
Let 
\[
Z^{\e,x}_t:=\frac{X^{\e,x}_t-\varphi^{x}_t}{\e}, \quad t\gqq 0.
\] 
Then the process $(Z^{\e,x}_t)_{t\gqq 0}$ is the unique strong solution of the stochastic differential equation
\begin{align*}
\left\{
\begin{array}{r@{\;=\;}l}
\ud Z^{\e,x}_t & -\frac{1}{\e}\left( b(X^{\e,x}_t)-b(\varphi^{x}_t) \right)\ud t+\ud L_t \quad \textrm{ for any } t\gqq 0,\\
Z^{\e,x}_0 & 0.
\end{array}
\right.
\end{align*}
The mean value theorem yields
\begin{align}\label{eq:appearlinear}
\ud Z^{\e,x}_t=-\Big(\int_{0}^{1} Db(\varphi^{x}_t+\theta (X^{\e,x}_t-\varphi^{x}_t))\ud \theta\Big) Z^{\e,x}_t\ud t+\ud L_t\quad \textrm{ for any } t\gqq 0.
\end{align}
By construction, $X^{\e,x}_t=\varphi^{x}_t+\e Z^{\e,x}_t$ for any $t\gqq 0$. However,  \eqref{eq:appearlinear} has the same level of complexity as \eqref{dde1}.
Using \eqref{eq:zeroapprox} in \eqref{eq:appearlinear}
we derive the linear inhomogeneous approximation of \eqref{eq:appearlinear} as follows. Let $(Y^{x}_t)_{t\gqq 0}$ be the unique strong solution of the linear inhomogeneous stochastic differential equation
\begin{align}\label{eq: Yxt}
\left\{
\begin{array}{r@{\;=\;}l}
\ud Y^{x}_t & -Db(\varphi^{x}_t) Y^{x}_t\ud t+\ud L_t \quad \textrm{ for any } t\gqq 0,\\
Y^{x}_0 & 0.
\end{array}
\right.
\end{align}
Instead of \eqref{eq:zeroapprox} we claim the following stronger result, that is, the first order approximation in the sense of Section 2, Chapter 2 in \cite{FW}
\[
\mathbb{P}\big(|X^{\e,x}_t- (\varphi^{x}_t+\e Y^{x}_t)|\gqq \e^{3/2}\big)\to 0, \quad \textrm{ as } \e\to 0
\]
for times $t\gg t^x_\e$, where $t^x_\e$ is given in \eqref{eq:tiempos}.
For a concise quantification of the approximation, see Lemma~\ref{prop: secondorder} in 
Appendix~\ref{ap: multiscale}. 
Next, we define the first order approximation
\begin{equation}\label{eq: Yxte}
Y^{\e}_t(x): = \varphi_t^x + \e Y^x_t \quad \textrm{ for any } t\gqq 0.
\end{equation}
It is not hard to see that for any $\e$ there exists a limiting distribution $\mu^{\e}_*$ such that for any $x\in \RR^d$, the process $(Y^{\e}_t(x))_{t\gqq 0}$ converges  to $\mu^{\e}_*$ in the total variation distance  as $t$ tends to infinity. For further details see Lemma~\ref{lem: cvtlineal} in Appendix~\ref{Appendix C}. Moreover, it is shown there that $\mu^{\e}_*$ is  the unique invariant distribution of the homogeneous Ornstein-Uhlenbeck process
\[
\ud Z^{\e}_t=-Db(0)Z^{\e}_t\ud t+\e \ud L_t,
\]
and has a $\cC^{\infty}$-density with respect to the Lebesgue measure on $\RR^d$.
Note that $(Y^{\e}_t(x))_{t\gqq 0}$ satisfies the inhomogeneous equation
\begin{equation}\label{eq: SDE Y}
\ud Y^{\e}_t(x)=\left(-b(\varphi^{x}_t)+Db(\varphi^x_t)\varphi^x_t-Db(\varphi^x_t)Y^{\e}_t(x)\right)\ud t+\e\ud L_t, \quad Y^{\e}_0(x)=x.
\end{equation}
Since we need to compare solutions of  stochastic differential equations with different initial conditions, we introduce the following notation.
Let $\fT$ be a positive number and
$\xi $ be a given random vector on $\mathbb R^d$. 
We assume that $\xi$ is $\mathcal{F}_{\fT}$-measurable for $(\mathcal{F}_t)_{t\gqq 0}$ defined in Subsection~\ref{subsec:stochasticperturbation}.
Let $(Y^{\e,x}(t;\fT,\xi))_{t\gqq 0}$ be the unique strong solution of the stochastic differential equation
\begin{align}\label{eq: linealhomogenea}
\left\{
\begin{array}{r@{\;=\;}l}
\ud Y^{\e,x}(t;\fT,\xi) & \left(-b(\varphi^{x}_{t+\fT})+Db(\varphi^x_{t+\fT})\varphi^x_{t+\fT}-Db(\varphi^x_{t+\fT})Y^{\e,x}(t;\fT,\xi)\right)\ud t+\e\ud L_{t+\fT},\\
Y^{\e,x}(0;\fT,\xi) & \xi.
\end{array}
\right.
\end{align}
Let $\Delta_\e>0$ (independently of $x$) with $\lim\limits_{\e \ra 0} \Delta_\e=0$.
For any $\rho\in \mathbb{R}$, we define 
\[
T^x_\e:=t^x_\e-\Delta_\e+\rho\cdot w^x_\e,
\] where  $t^x_\e$ and $w^x_\e$  are given in Theorem~\ref{thm: main result}.
In what follows, we always take $\fT=T^x_\e$. 
Then $T^x_\e>0$ for $0<\e\ll 1$.

\bigskip 
\subsubsection{\textbf{Key cutoff estimate}} \hfill\\

\noindent The proofs of the main results Theorem~\ref{thm: main result} and 
Theorem~\ref{thm: main result2} are based on the following fundamental inequality. 
On the one hand, note that  for any $\rho\in \mathbb{R}$ and $\e$ small enough we have
\begin{align*}
\norm{X^{\e}
_{t^x_\e+\rho\cdot w^x_\e}(x)-\mu^{\e}}
&=
\norm{X^{\e}_{\Delta_\e}(X^{\e}_{T^x_\e}(x))-\mu^{\e}}\\
&\lqq 
 \norm{X^{\e}_{\Delta_\e}(X^{\e}_{T^x_\e}(x))-
Y^{\e,x}(\Delta_\e;T^x_\e, X^{\e}_{T^x_\e}(x))}\\
&\qquad +\norm{Y^{\e,x}(\Delta_\e;T^x_\e,
X^{\e}_{T^x_\e}(x))-
Y^{\e,x}(\Delta_\e;T^x_\e,Y^{\e,x}(T^x_\e;0,x))}\\
&\qquad +\norm{Y^{\e,x}(\Delta_\e;T^x_\e,Y^{\e,x}(T^x_\e;0,x))-\mu^{\e}_*}
+\norm{\mu^{\e}_*-\mu^{\e}}.
\end{align*}
Conversely, we obtain
\begin{align*}
&\norm{Y^{\e,x}(\Delta_\e;T^x_\e,Y^{\e,x}(T^x_\e;0,x))-\mu^{\e}_*} \\
&\quad \lqq  
 \norm{Y^{\e,x}(\Delta_\e;T^x_\e,Y^{\e,x}(T^x_\e;0,x))-Y^{\e,x}(\Delta_\e;T^x_\e,X^{\e}_{T^x_\e}(x))}\\
&\qquad +\norm{Y^{\e,x}(\Delta_\e;T^x_\e,X^{\e}_{T^x_\e}(x))-X^{\e}_{\Delta_\e}(X^{\e}_{T^x_\e}(x))
} +\norm{X^{\e}_{\Delta_\e}(X^{\e}_{T^x_\e}(x))-\mu^{\e}}
+\norm{\mu^{\e}-\mu^{\e}_*}.
\end{align*}
Note that $
Y^{\e,x}(\Delta_\e, T^x_\e,Y^{\e,x}(T^x_\e;0,x))=Y^{\e,x}(t^x_\e+\rho\cdot w^x_\e;0,x)$.
Combining both preceding inequalities we deduce 
\begin{align}\label{ine5}
&\left|\norm{X^{\e}_{t^x_\e+\rho\cdot w^x_\e}(x)-\mu^{\e}}-
\norm{Y^{\e,x}(t^x_\e+\rho\cdot w^x_\e;0,x)-\mu^{\e}_*}\right|\nonumber\\
&\qquad \qquad\lqq \norm{X^{\e}_{\Delta_\e}(X^{\e}_{T^x_\e}(x))
-Y^{\e,x}(\Delta_\e;T^x_\e,X^{\e}_{T^x_\e}(x))}\\
&\quad \qquad \qquad+ \norm{Y^{\e,x}(\Delta_\e;T^x_\e,X^{\e}_{T^x_\e}(x))
-
Y^{\e,x}(\Delta_\e;T^x_\e,Y^{\e,x}(T^x_\e;0,x))}+\norm{\mu^{\e}_*-\mu^{\e}}.\nonumber\\
& \qquad \qquad=E_1+E_2+E_3,\nonumber
\end{align}
where 
\begin{align*}
E_1:=&\norm{X^{\e}_{\Delta_\e}(X^{\e}_{T^x_\e}(x))
-Y^{\e,x}(\Delta_\e;T^x_\e,X^{\e}_{T^x_\e}(x))},\\
E_2:=&\norm{Y^{\e,x}(\Delta_\e;T^x_\e,X^{\e}_{T^x_\e}(x))
-
Y^{\e,x}(\Delta_\e;T^x_\e,Y^{\e,x}(T^x_\e;0,x))},\\
E_3:=&\norm{\mu^{\e}_*-\mu^{\e}}.
\end{align*}
Roughly speaking, 
it turns out that the processes $(X^{\e}_t(x))_{t\gqq 0}$
and $(Y^{\e}_t(x))_{t\gqq 0}$ are close enough for time scales of order
$\mathcal{O}(\ln(\nicefrac{1}{\e}))$  in order to carry out the following quantitative coupling procedure.
Since $(X^{\e}_t(x))_{t\gqq 0}$ and $(Y^{\e}_t(x))_{t\gqq 0}$ have different (inhomogeneous) drifts, couplings which dominate the total variation distance 
typically only hold for short-time horizons. For an excellent introduction on the subject in the diffusive case we refer to \cite{EBERLE}.
Since the process $(Y^{\e}_t(x))_{t\gqq 0}$ is linear, the precise cutoff behavior (cutoff, window cutoff and profile cutoff) is derived from it in the spirit of \cite{BP}. However, it is inhomogeneous such that the results of \cite{BP} cannot be applied directly. They are adapted in Subsection~\ref{sec: OUP}. Recall that $\mu^{\e}_*$ is the limiting distribution of the process
$(Y^{\e}(t;0,x))_{t\gqq 0}$.

\begin{prop}[Window and profile cutoff phenomenon for the first order approximation $Y^{\e}$]\label{prop: windows}
\noindent
Assume Hypotheses~\ref{hyp: potential},
~\ref{hyp: moment condition}, 
~\ref{hyp: regularity} and ~\ref{hyp: blumental}  are satisfied for
$\alpha\in (0,2)$,
 $\beta>0$ and $x\in \RR^d\setminus \{0\}$. Let 
 $(Y^{\e,x})_{\e\in (0,1]}$ be the 
 family of inhomogeneous Ornstein-Uhlenbeck processes given by
 $Y^{\e,x}:=(Y^{\e}(t;0,x))_{t\gqq 0}$ in  \eqref{eq: Yxte}.
\begin{enumerate}
\item Then
$(Y^{\e,x})_{\e\in (0,1]}$ exhibits a window cutoff phenomenon with respect to $\mu^{\e}_*$ as $\e \to 0$ 
at the enhanced time scale $(t^x_\e,w^x_\e)$
given by
\begin{equation}\label{eq:escalastw}
t^{x}_\e=\frac{1}{\lambda_x}\ln\left(\nicefrac{1}{\e}\right)+
\frac{\ell_x-1}{\lambda_x}
\ln\left(\ln\left(\nicefrac{1}{\e}\right)\right)
\quad \textrm{ and } \quad
w^{x}_\e=\frac{1}{\lambda_x}+o_{\e}(1), 
\end{equation}
where $\lambda_x>0$ and $\ell_x \in \{1,\ldots,d-1\}$ are the constants appearing in
the Hartman-Grobman decomposition of Lemma~\ref{asymp}.
\item Then $(Y^{\e,x})_{\e\in (0,1]}$ exhibits a profile cutoff phenomenon
with respect to $\mu^{\e}_*$
 as $\e \to 0$ at the enhanced time scale $(t^x_\e,w^x_\e)$
given by \eqref{eq:escalastw} with profile function 
\[
G_{x}(\rho)=
\norm{\left(
e^{-\rho}\cdot \frac{e^{-\la_x \tau_x}}{\lambda_x^{\ell_x-1}}
v+ Z_\infty \right)- Z_\infty}
\quad \textrm{ for any } \rho\in \RR,\;  v\in \omega(x),
\]
where $\tau_x$ is given in Lemma \ref{asymp} and $\omega(x)$ is defined in \eqref{eq: omegasetx} 
if and only if for any $a>0$ the map
\begin{equation*}
\omega(x)\ni v\mapsto 
\norm{(a v+ Z_\infty) - Z_\infty}\quad
\textrm{ is constant}. 
\end{equation*} 
\end{enumerate}
\end{prop} 
\noindent The proof is given in Subsection~\ref{sec: OUP} and relies on the Hartman-Grobman decomposition of Lemma~\ref{asymp}.
In what follows, we  argue that the upper bound of inequality \eqref{ine5} tends to zero as $\e\rightarrow 0$. To be precise, we  show the following.

\begin{prop}[Error term $E_1$: the nonlinear short time coupling]\label{prop: stc}
Assume Hypotheses~\ref{hyp: potential},
~\ref{hyp: moment condition}, 
~\ref{hyp: regularity} and ~\ref{hyp: blumental} are satisfied for $\alpha\in (3/2,2)$ and $\beta>0$.
Let $\Delta_\e=\e^{\nicefrac{\alpha}{2}}$. For any $x\in \RR^d$ it follows
\[
\lim\limits_{\e\rightarrow 0}
\norm{X^{\e}_{\Delta_\e}(X^{\e}_{T^x_\e}(x))
-Y^{\e,x}(\Delta_\e;T^x_\e,X^{\e}_{T^x_\e}(x))}
=0.
\]
\end{prop}
\noindent
The complete proof can be found in Subsection~\ref{sec: short coupling}
and it is based on the local limit theorem for strongly locally layered stable L\'evy measures on the short-time scale $\Delta_\e\to 0$. The limitation of $\alpha\in (3/2,2)$ is due to the tail integrability of the characteristic function of $X^\e_{\Delta_\e}(x)$. It is of technical nature, but it seems difficult to remove.  

\begin{prop}[Error term $E_2$: the linear inhomogeneous coupling]\label{prop: lic}
Assume Hypotheses~\ref{hyp: potential},
~\ref{hyp: moment condition}, 
~\ref{hyp: regularity} and ~\ref{hyp: blumental} are satisfied for
$\alpha\in (0,2)$ and
 $\beta>0$. 
Let $\Delta_\e=\e^{\nicefrac{\alpha}{2}}$. For any $x\in \RR^d$ it follows
\[
\lim\limits_{\e\rightarrow 0}\norm{Y^{\e,x}(\Delta_\e;T^x_\e,X^{\e}_{T^x_\e}(x))
-
Y^{\e,x}(\Delta_\e;T^x_\e,Y^{\e,x}(T^x_\e;0,x))}
=0.
\]
\end{prop}
The proof is given Subsection~\ref{sec: linear coupling} and relies on a version of the local limit theorem by \cite{HOU} for strongly locally layered stable distributions and small times $\Delta_\e$. 

\noindent We approximate the invariant distribution $\mu^{\e}$ of $X^{\e}_\cdot(x)$ by the limiting distribution $\mu^{\e}_*$ of the inhomogeneous Ornstein-Uhlenbeck $Y^{\e}_\cdot(x)$ in the total variation distance.
\begin{prop}[Error term $E_3$: the equilibrium asymptotics]\label{prop: equil}
Assume Hypotheses~\ref{hyp: potential},
~\ref{hyp: moment condition}, 
~\ref{hyp: regularity} and ~\ref{hyp: blumental} are satisfied for
$\alpha\in (3/2,2)$ and
 $\beta>0$.
It follows
\[
\lim\limits_{\e\rightarrow 0}\norm{\mu^{\e}_*-\mu^{\e}}=0.
\]
\end{prop}
\noindent The proof is given in Subsection~\ref{sec: local}. 

\noindent
\begin{proof}[\textbf{Proof of Theorem~\ref{thm: main result} and Theorem~\ref{thm: main result2}}:]
We apply
Propositions~\ref{prop: stc}, ~\ref{prop: lic} and~\ref{prop: equil} to the key estimate \eqref{ine5}  and obtain
\[
\lim\limits_{\e\to 0}\left|\norm{X^{\e}_{t^x_\e+\rho\cdot w^x_\e}(x)-\mu^{\e}}-
\norm{Y^{\e,x}(t^x_\e+\rho\cdot w^x_\e;0,x)-\mu^{\e}_*}\right|=0.
\]
Finally, Proposition~\ref{prop: windows} implies the main result in Theorem~\ref{thm: main result} and Theorem~\ref{thm: main result2}. 
\end{proof}

\bigskip 
\section{\textbf{The local results (Prop. 1 - 4) in the proofs of Theorem~\ref{thm: main result} and Theorem~\ref{thm: main result2}}}

\subsection{\textbf{Cutoff for the inhomogeneous linearization (Proposition~\ref{prop: windows})}}\label{sec: OUP}

\subsubsection{\textbf{{Cutoff linearization via Hartman-Grobman}}}\hfill\\

By Lemma~\ref{lem:convergenceindist} in Appendix~ \ref{Appendix C} we see that  $\mu^{\e}_*$ is the distribution  of $\e Z_\infty$, where $Z_\infty$ is the unique invariant distribution of the homogeneous Ornstein-Uhlenbeck process $Z = (Z_t)_{t\gqq 0}$ given by
\begin{equation*}
\ud Z_t = - Db(0) Z_t \ud t +\ud L_t.
\end{equation*}
As a consequence $\mu_*^\e \sim \e Z_\infty$ for any $\e\in (0,1)$. 
We start  with the observation that $Z_\infty$ is absolutely continuous. Indeed,
let $\zeta_t$ be the characteristic function of $Z_t$ and $\zeta_\infty$ be the characteristic function of $Z_\infty$. By Theorem~3.1 in Sato and Yamazato \cite{SaYa} we have for any $t\gqq 0$
\[
|\zeta_t(\theta)|=\exp\left(\int_{0}^{t} \mathsf{Re}(\psi(e^{-Db(0)s}\theta))\ud s \right), 
\quad \theta\in \RR^d,
\]
where 
\[
\mathsf{Re}(\psi(\vartheta))=\int_{\RR^d} \left(\cos(\<\vartheta,u\>)-1\right)\nu(\ud u)\lqq 0, 
\quad \vartheta\in \RR^d.
\] 
Hence,
$|\zeta_\infty(\theta)|\lqq |\zeta_t(\theta)|$ for all $\theta\in \mathbb{R}^d$. 
Then Item 3. in Section 4 of \cite{BP} implies that $Z_\infty$ has a bounded $\cC^{\infty}$-density  with respect to the Lebesgue measure on $\mathbb{R}^d$,
where
we take $\kappa(v)=c_{\sphericalangle}\;|v|^{\alpha}$ in their notation, and 
$c_{\sphericalangle}$, $\alpha$ being given in Lemma~\ref{lem oreymasuda}.
In particular, $Z_\infty$ is absolutely continuous on $\mathbb{R}^d$.

The following lemma reduces the cutoff phenomenon for the non-homogeneous linearization $Y^{\e}_\cdot(x)$ of $X^{\e, x}$ to the homogeneous linearization $Z$. 

\begin{lem}[Elimination of the inhomogeneity in the cutoff linearization]\label{lem:cutoff-linearization}
Let the hypotheses of Proposition~\ref{prop: windows} be satisfied for some $x\neq 0$, $\alpha\in (0,2)$ and $\beta>0$. 
We define 
\[
 d^{\e, x}(t) :=\norm{Y^{\e}_t(x) - \e Z_\infty} 
\]
and 
\begin{align*}
\tilde{D}^{\e, x}(t)&:= \norm{\Big(\frac{(t-\tau_x)^{\ell_x-1}e^{-\la_x (t-\tau_x)}}{\e}v(t-\tau_x,x) + Z_\infty\Big) - Z_\infty}, \quad t\gqq \tau_x,
\end{align*}
where $v(t,x)=\sum_{k=1}^{m}e^{i \theta^k_x t}v^k_x$ and $\la_x$, $\ell_x$, $\tau_x$, $\theta^1_x,\ldots, \theta^m_x$ and $v^1_x,\ldots, v^m_x$ are the quantities given by the Hartman-Grobman decomposition in Lemma~\ref{asymp}.\\
Then for any $\rho\in \RR$ 
\begin{align*}
\limsup_{\e\ra 0} d^{\e, x}(t^x_\e + \rho \cdot w_\e) &=  \limsup_{\e\ra 0} \tilde{D}^{\e, x}(t^x_\e + \rho \cdot w_\e)\quad \mbox{ and }\\
\liminf_{\e\ra 0} d^{\e, x}(t^x_\e + \rho \cdot w_\e) &=  \liminf_{\e\ra 0} \tilde{D}^{\e, x}(t^x_\e + \rho \cdot w_\e).
\end{align*}
\end{lem}

\begin{proof}[Proof of Lemma~\ref{lem:cutoff-linearization}:]
\noindent Let $\e\in (0,1)$. We observe that $Y^{\e,x}(t;0,x)=Y^{\e}_t(x)$, $t\gqq 0$,
where $Y^{\e}_t(x)=\varphi^x_t+\e Y^x_t$
and $(Y^x_t)_{t\gqq 0}$ is the unique strong solution of \eqref{eq: Yxt}.
Due to scale and (deterministic) shift invariance of the total variation distance 
given in part ii) of Lemma A.1 of \cite{BP}, it follows for all $t\gqq 0$
\begin{align}
 d^{\e, x}(t) 
 &=\norm{Y^{\e}_t(x) - \mu^{\e}_*}=\norm{Y^{\e}_t(x) - \e Z_\infty} \nonumber\\[2mm]
 &\lqq \norm{(\varphi_t^x + \e Y^x_t) - (\varphi_t^x + \e Z_\infty)} 
 + \norm{(\varphi_t^x + \e Z_\infty) - \e Z_\infty}  \nonumber\\[2mm]
 &= \norm{Y^x_t - Z_\infty} + \underbrace{\norm{(\nicefrac{\varphi^x_t}{\e} + Z_\infty) - Z_\infty}}_{=:D^{\e, x}(t)} \label{ine: fine}.
\end{align}
That is,
$
 d^{\e, x}(t) - D^{\e, x}(t) \lqq \norm{Y^x_t - Z_\infty}.
$
Analogously, we obtain
\begin{align*}
D^{\e, x}(t) 
&= \norm{(\varphi^x_t + \e Z_\infty) - \e Z_\infty}\\
&\lqq \norm{(\varphi^x_t + \e Z_\infty) - (\varphi^x_t + \e Y^x_t)} + \norm{(\varphi^x_t + \e Y^x_t) - \e Z_\infty} \\
&= \norm{Y^x_t - Z_\infty} + d^{\e, x}(t), 
\end{align*}
and deduce that
\begin{equation}\label{eq: approx1}
|d^{\e, x}(t) - D^{\e, x}(t)| \lqq \norm{Y^x_t - Z_\infty}\quad \textrm{ for all } t\gqq 0.
\end{equation}
\begin{rem}
\label{rem:nocutoff0}
 Note that for $x=0$, $\varphi^{x}_t=0$ for any $t\gqq 0$ and consequently,
$D^{\e,x}(t)=0$. Since the right-hand side of inequality \eqref{eq: approx1} does not depend on $\e$ and tends to zero for $t\to \infty$, we have for any time scale $(s_\e)_{\e\in (0,1)}$ such that 
$s_\e\to \infty$, as $\e\to 0$,  the limit
$
\lim\limits_{\e\to 0}d^{\e,x}(s_\e)=0.
$
Hence the family $(Y^{\e,x})$ does not exhibit a cutoff phenomenon for any time scale.
\end{rem}
\noindent
As a consequence we continue with 
$x\neq 0$ and recall that
$
D^{\e, x}(t) = \norm{\left(\nicefrac{\varphi^x_{t}}{\e} + Z_\infty\right) - Z_\infty}.
$
In addition, let
\[
R^{\e,x}(t):=\norm{\Big(\nicefrac{\varphi^x_{t}}{\e} + Z_\infty\Big) - \Big(\frac{(t-\tau_x)^{\ell_x-1}e^{-\la_x (t-\tau_x)}}{\e}v(t-\tau_x,x)+Z_\infty\Big)}, \quad t\gqq \tau_x.
\]
By the triangle inequality it follows
\begin{align*}
 D^{\e,x}(t) &\lqq \norm{\big(\nicefrac{\varphi^x_{t}}{\e} + Z_\infty\big)-\Big(\frac{(t-\tau_x)^{\ell_x-1}e^{-\la_x (t-\tau_x)}}{\e}v(t-\tau_x,x) + Z_\infty\Big)}\\
 &\quad+
\norm{\Big(\frac{(t-\tau_x)^{\ell_x-1}e^{-\la_x (t-\tau_x)}}{\e}v(t-\tau_x,x) + Z_\infty\Big)-Z_\infty}=R^{\e,x}(t)+\tilde{D}^{\e,x}(t),
\end{align*}
and analogously 
$\tilde{D}^{\e,x}(t)\lqq R^{\e,x}(t)+{D}^{\e,x}(t)$ which yields 
\begin{equation}\label{ine: desi}
|D^{\e, x}(t)-\tilde{D}^{\e, x}(t)|\lqq R^{\e,x}(t)\quad \textrm{ for all } t\gqq \tau_x.
\end{equation}
Combining \eqref{eq: approx1} and \eqref{ine: desi}
we obtain
\begin{equation}\label{eq: master}
|d^{\e, x}(t)-\tilde{D}^{\e, x}(t)|\lqq R^{\e,x}(t)+\norm{Y^{x}_t-Z_\infty}\quad \textrm{ for all } t\gqq \tau_x.
\end{equation}
The limit \eqref{eq: tvd22} in Lemma~\ref{lem: cvtlineal} in Appendix~\ref{Appendix C} shows that
\[
\lim\limits_{t\to \infty}\norm{Y^{x}_t-Z_\infty}=0.
\]
In particular, for any $\rho\in \mathbb{R}$ we obtain
\begin{equation}\label{eq: master1}
\lim\limits_{\e\to 0}\norm{Y^{x}_{t^x_\e+\rho\cdot w^x_\e}-Z_\infty}=0.
\end{equation}
\textbf{Claim:} For any $\rho\in \mathbb{R}$ we have 
\begin{equation}\label{eq: master2}
\lim\limits_{\e\to 0}R^{\e,x}(t^x_\e+\rho\cdot w^x_\e)=0.
\end{equation}
First, we note that the scale and shift invariance of the total variation distance imply for all $t\gqq 0$
\[
R^{\e,x}(t)=\norm{\Big(\frac{(t-\tau_x)^{\ell_x-1}e^{-\la_x (t-\tau_x)}}{\e}\Big(\frac{e^{\la_x (t-\tau_x)}}{(t-\tau_x)^{\ell_x-1}}\varphi^x_{t}-v(t-\tau_x,x)\Big) + Z_\infty\Big) -Z_\infty}.
\]
Secondly, the Hartman-Grobman decomposition in Lemma~\ref{asymp} states
\begin{equation}\label{eq: limiting}
\lim\limits_{\e \to 0} \Big|\frac{e^{\la_x (t_\e+\rho\cdot w_\e-\tau_x)}}{(t_\e+\rho\cdot w_\e-\tau_x)^{\ell_x-1}}\varphi^x_{t_\e+\rho\cdot w_\e}-v(t_\e+\rho\cdot w_\e-\tau_x,x)\Big|=0,
\end{equation}
and  the very definition of $t^x_\e$ and $w^x_\e$ yields
\begin{equation}\label{eq: limite}
\lim\limits_{\e \to 0}\frac{(t^x_\e+\rho\cdot w^x_\e-\tau_x)^{\ell_x-1}e^{-\la_x (t^x_\e+\rho\cdot w^x_\e-\tau_x)}}{\e}=\frac{e^{-\rho-\la_x \tau_x}}{\la_x^{\ell_x-1}}.
\end{equation}
Combining  \eqref{eq: limiting}, \eqref{eq: limite} and the absolute continuity of $Z_\infty$
with the  Scheff\'e lemma for densities implies that
\mbox{$R^{\e,x}(t^x_\e+\rho\cdot w^x_\e)$} tends to zero as $\e \to 0$.
Joining \eqref{eq: master}, \eqref{eq: master1} and \eqref{eq: master2} yields that
any  cutoff phenomenon in the sense of Definition~\ref{def: cutoff} can be read off from
the simpler term $\tilde{D}^{\e,x}$.
\end{proof} 

\bigskip 

\subsubsection{\textbf{Window cutoff for the inhomogeneous O-U process (Proposition~\ref{prop: windows}, Item (1))}}\label{ss:Item1}\hfill\\

\begin{proof}[Proof of Proposition \ref{prop: windows}, Item (1):] 
\noindent By Lemma \ref{lem:cutoff-linearization} it is enough to show the window cutoff phenomenon for $\tilde{D}^{\e, x}$. We observe that $\lim\limits_{t\to \infty}v(t,x)$ may not exist in general.
Set
\[
\bar{D}^x_{\rho}:=\limsup\limits_{\e\to 0}\tilde{D}^{\e, x}(t^x_\e+\rho \cdot w^x_\e).
\]
Then there exists a subsequence $(t^x_{\e_j}+\rho \cdot w^x_{\e_j})_{j\in \NN}$ of $(t^x_\e+\rho\cdot w^x_\e)_{\e\in (0,1]}$ such that $\e_j\to 0$ as $j\to \infty$ and for which 
\[
\bar{D}^x_{\rho}=\lim\limits_{j\to \infty}\tilde{D}^{\e_j, x}(t^x_{\e_j}+\rho \cdot w^x_{\e_j}).
\]
Notice that the sequence  $(v(t^x_{\e_j}+\rho \cdot w^x_{\e_j}-\tau_x))_{j\in \NN}$ is bounded by $\sum_{k=1}^m |v^k|$. Then the Bolzano-Weierstrass theorem yields
the existence of a subsequence  $(\e_{j_n})_{n\in \NN}$ of
$(\e_{j})_{j\in \NN}$ such that 
\begin{equation}\label{eq: limitset}
\lim\limits_{n\to \infty}v(\e_{j_n},x)=:\hat{v}_{\rho}(x) \quad \textrm{ exists.}
\end{equation} 
By construction $\hat{v}_{\rho}(x)\in \omega(x)$.
Combining \eqref{eq: limite} and \eqref{eq: limitset} and using that the law of $Z_\infty$ is absolutely continuous with respect to the Lebesgue measure on $\RR^d$, Scheff\'e's lemma for densities implies
\begin{equation}
\begin{split}
&\bar{D}^x_\rho =\lim\limits_{n\to \infty}\tilde{D}^{\e_{j_n}, x}(t^x_{\e_{j_n}}+\rho \cdot w^x_{\e_{j_n}})\\
&=
\lim\limits_{n\to \infty}\norm{\Big(\frac{(t^x_{\e_{j_n}}+\rho \cdot w^x_{\e_{j_n}}-\tau_x)^{\ell_x-1}e^{-\la_x (t^x_{\e_{j_n}}+\rho \cdot w^x_{\e_{j_n}}-\tau_x)}}{\e_{j_n}}v(t^x_{\e_{j_n}}+\rho \cdot w^x_{\e_{j_n}}-\tau_x,x) + Z_\infty\Big) - Z_\infty}\\
&=\norm{\Big(\frac{e^{-\rho-\la_x \tau_x}}{\la_x^{\ell_x-1}}\hat{v}_{\rho}(x)+ Z_\infty\Big) - Z_\infty}.\label{eq:limsupd} 
\end{split}
\end{equation}
Analogously, we deduce
\begin{equation}\label{eq:liminfd}
\underbar{D}^x_{\rho}:=\liminf\limits_{\e\to 0}\tilde{D}^{\e, x}(t^x_\e+\rho \cdot w^x_\e)
=\norm{\Big(\frac{e^{-\rho-\la_x \tau_x}}{\la_x^{\ell_x-1}}\check{v}_{\rho}(x)+ Z_\infty\Big) - Z_\infty},
\end{equation} 
where $\check{v}_{\rho}(x)\in \omega(x)$. 
Let $\rho>0$.
In the sequel we send $\rho\to \infty$.
We observe that the upper limiting vector $\hat{v}_{\rho}(x)$ depends on $\rho$, however, it is  uniformly bounded by $\sum_{k=1}^{m}|v^k|$.
Hence, $e^{-\rho}\hat{v}_{\rho}(x)$ tends to zero as $\rho\to \infty$.
With the help of Scheff\'e's lemma for densities we obtain
\begin{equation}\label{eq: limsup0}
\lim\limits_{\rho\to \infty} \bar{D}^x_{\rho}=0.
\end{equation}
For $\rho<0$, by \eqref{eq: outsidezero} we observe that $|\check{v}_{\rho}(x)|\gqq \liminf\limits_{t\to \infty}|v(t,x)|>0$, where the right-hand side does not depend on $\rho$. Hence, $e^{-\rho}|\check{v}_{\rho}(x)|\to \infty$ as $\rho\to -\infty$.  
A standard version of Scheff\'e's lemma for densities with diverging drift (see Lemma A.3 in \cite{BP}) implies
\begin{equation}\label{eq: liminf1}
\lim\limits_{\rho\to -\infty} \underbar{D}^x_{\rho}=1.
\end{equation}
Combining \eqref{eq: limsup0} and \eqref{eq: liminf1} shows the window cutoff limits for $\tilde{D}^{\e,x}$ and hence the window cutoff phenomenon for the family $(Y^{\e}(x))_{\e\in (0,1)}$. 
\end{proof}

\bigskip 
\subsubsection{\textbf{Profile cutoff for the inhomogeneous O-U process (Proposition~\ref{prop: windows}, Item (2))}}\label{ss:Item2}\hfill\\

\begin{proof}[Proof of Proposition \ref{prop: windows}, Item (2):] 
\noindent By Lemma \ref{lem:cutoff-linearization} it is enough to show the window cutoff phenomenon for $\tilde{D}^{\e, x}$. 
By \eqref{eq:limsupd} and \eqref{eq:liminfd} we have for any $\rho\in \RR$
\[
\limsup\limits_{\e\to 0}\tilde{D}^{\e, x}(t^x_\e+\rho \cdot w^x_\e)=\norm{\Big(\frac{e^{-\rho-\la_x \tau_x}}{\la_x^{\ell_x-1}}\hat{v}_{\rho}(x)+ Z_\infty\Big) - Z_\infty}
\]
and 
\[
\liminf\limits_{\e\to 0}\tilde{D}^{\e, x}(t^x_\e+\rho \cdot w^x_\e)=\norm{\Big(\frac{e^{-\rho -\la_x \tau_x}}{\la_x^{\ell_x-1}}\check{v}_{\rho}(x)+ Z_\infty\Big) - Z_\infty},
\]
where $\hat{v}_{\rho}(x),\check{v}_{\rho}(x)\in \omega(x)$ defined in \eqref{eq: omegasetx}. 
The limit 
\[
\lim\limits_{\e\to 0}\tilde{D}^{\e, x}(t^x_\e+\rho \cdot w^x_\e)\quad \textrm{exists}
\]
if and only if 
\[
\norm{\Big(\frac{e^{-\rho-\la_x\tau_x}}{\la_x^{\ell_x-1}}\hat{v}_{\rho}(x)+ Z_\infty\Big) - Z_\infty}=
\norm{\Big(\frac{e^{-\rho-\la_x\tau_x}}{\la_x^{\ell_x-1}}\check{v}_{\rho}(x)+ Z_\infty\Big) - Z_\infty}.
\]
We start with the necessary condition for profile cutoff in Theorem~\ref{thm: main result2}.
If for  any  $a >0$ the map
\[
v\in \omega(x)\mapsto 
\norm{\Big(a  v+ Z_\infty\Big) - Z_\infty}\quad
\textrm{ is constant},
\]
then we have
\[
\lim\limits_{\e\to 0}\tilde{D}^{\e, x}(t^x_\e+\rho \cdot w^x_\e)=
\norm{\Big(\frac{e^{-\rho-\la_x\tau_x}}{\la_x^{\ell_x-1}}v+ Z_\infty\Big) - Z_\infty},
\]
where $v$ is any representative of $\omega(x)$.
This yields the desired profile cutoff phenomenon for the family $(Y^{\e,x})_{\e\in (0,1)}$.\\

\noindent We continue with the sufficient condition for profile cutoff in Theorem~\ref{thm: main result2}.
Let $v\in \omega(x)$, i.e. there exists a subsequence $(t_{j})_{j\in \NN}$ such that
\[
\lim\limits_{j\to \infty}v(t_j,x)=v.
\]
For any $x\in \RR^d$ and $\rho\in \RR$ 
consider the parametrization  $\e\mapsto  t^{x}_\e+\rho \cdot w^{x}_\e-\tau_x$ and set 
$t_{j}:=t^{x}_{\e_j}+\rho \cdot w^{x}_{\e_j}-\tau_x$ for all $j\in \NN$.
Limit \eqref{eq: limite} and Scheff\'e's lemma for densities imply
\begin{align}\label{eq: omegav}
\lim_{j\to \infty}\tilde{D}^{\e_{j}, x}(t^{x}_{\e_j}+\rho \cdot w^{x}_{\e_j})=
 \norm{\Big(\frac{e^{-\rho-\la_x\tau_x}}{\la_x^{\ell_x-1}} v + Z_\infty\Big) - Z_\infty}.
\end{align}
Since we are assuming profile cutoff, it follows
\begin{align*}
\lim_{\e\to 0}\tilde{D}^{\e, x}(t^{x}_{\e}+\rho \cdot w^{x}_{\e})&=
\norm{\Big(\frac{e^{-\rho-\la_x\tau_x}}{\la_x^{\ell_x-1}}\hat{v}_{\rho}(x)+ Z_\infty\Big) - Z_\infty}\\
&=
\norm{\Big(\frac{e^{-\rho-\la_x\tau_x}}{\la_x^{\ell_x-1}}\check{v}_{\rho}(x)+ Z_\infty\Big) - Z_\infty},
\end{align*}
where $\hat{v}_{\rho}(x),\check{v}_{\rho}(x)\in \omega(x)$.
That is, the function
\begin{align*}
v\in \omega(x)\mapsto\norm{\Big(e^{-\rho}\cdot \frac{e^{-\la_x\tau_x}}{\la_x^{\ell_x-1}}v+ Z_\infty\Big) - Z_\infty}\quad
\textrm{ is constant}. \\
\end{align*}
\end{proof}
\begin{proof}[Proof of Proposition \ref{prop: windows}: ]
Combining Lemma \ref{lem:cutoff-linearization} with Subsubsection \ref{ss:Item1} and \ref{ss:Item2} yields the Item (1) and (2) of Proposition \ref{prop: windows}. 
\end{proof}

\bigskip
\subsection{\textbf{Coupling for the inhomogeneous O-U processes (Proposition~\ref{prop: lic})}}\label{sec: linear coupling}
\hfill\\

\noindent We keep the notation introduced in Subsection~\ref{subsec: outline}.
Let $\Delta_\e=\e^{\nicefrac{\alpha}{2}}$.
For any $\rho\in \mathbb{R}$ and $x\neq 0$, recall that
$T^x_\e=t^x_\e-\Delta_\e+\rho\cdot w^x_\e$, where  $t^x_\e$ and $w^x_\e$  are given in Theorem~\ref{thm: main result}. For $x= 0$ any time $T^x_\e = O(|\ln(\e)|^2)$ can be taken (see Lemma \ref{lem: dependence}).  
We show the following limit
\begin{equation}\label{eq: limiteinhomogeneous}
\lim\limits_{\e\rightarrow 0}\norm{Y^{\e,x}(\Delta_\e;T^x_\e,X^{\e}_{T^x_\e}(x))-
Y^{\e,x}(\Delta_\e;T^x_\e,Y^{\e,x}(T^x_\e;0,x))}
=0.
\end{equation}

\bigskip 
\subsubsection{\textbf{{Coupling by the local limit theorem for locally layered stable drivers}}} \hfill \\

\noindent
We recall that $(\varphi^x_t)_{t\gqq 0}$ is the solution of \eqref{dde1.1}. By \eqref{eq: linealhomogenea} and
the variation of constants formula yields the explicit representation
\begin{align}\label{eq: representation}
Y^{\e,x}(t;T^x_\e,z)=(\Phi^\e_t(x))^{-1}z+
(\Phi^\e_t(x))^{-1}
\int_{0}^{\Delta_\e}
\Phi^\e_s(x)\big(Db(\varphi^x_{T^x_\e+s})\varphi^x_{T^x_\e+s}-b(\varphi^{x}_{T^x_\e+s})\big)\ud s+
\e U^{x}_\e,
\end{align}
where $(\Phi^\e_t(x))_{t\gqq 0}$ is the solution of the  matrix valued inhomogeneous differential equation 
\begin{equation}\label{d:SGInhom}
\frac{\ud}{\ud t} \Phi_t = \Phi_t\, Db(\varphi_{T^x_\e+t}^x), \qquad \Phi_0=I_d,
\end{equation}
and 
\begin{equation}\label{d:SConvInhom}
U^x_\e:=(\Phi^\e_{\Delta_\e}(x))^{-1}
\int_{0}^{\Delta_\e}  \Phi^\e_s(x) \ud L_{T^x_\e+s}.
\end{equation}
Since $\varphi^x_{T^x_\e+t}\ra 0$, as $\e \ra 0$,
$U^x_\e$ resembles  the respective homogeneous Ornstein-Uhlenbeck process.
We claim that there exists a scale $\gamma_\e$ (independent of $x$) and a deterministic vector $a^x_\e$ such that $\gamma_\e U^x_\e+a^x_\e$ converges in total variation distance to an absolutely continuous random vector as $\e \ra 0$. To be precise, we state it as Proposition \ref{tomate} below. 
\begin{rem}\label{rem: generalhoudre}
Assume that the L\'evy measure $\nu$ is strongly locally layered stable
 in the sense of Definition~\ref{hyp: layered}  with 
parameters $(\nu_0,\nu_\infty,\Lambda,q,c_0,\alpha)$.
Let $\alpha\in (0,2)$ and $\beta>0$, where $\alpha$ is given in Definition~\ref{hyp: layered}
and $\beta$ is given in Hypothesis~\ref{hyp: moment condition}.
It is not hard to adapt the proof of Theorem~3.1 in \cite{HOU} to deduce that
\begin{equation}\label{eq:slys0}
(h^{-\nicefrac{1}{\alpha}}(L_{sh}+sh\eta_{\alpha,\beta})-s \fb_{\alpha,\beta})_{s\gqq 0}\stackrel{d}{\longrightarrow} \mathcal{S}_{\alpha}(\Lambda_1)\quad \textrm{as } h\to 0,\; h>0,
\end{equation}
where $\mathcal{S}_{\alpha}(\Lambda_1)$ is a strictly $\alpha$-stable process with spectral density $\Lambda_1(\ud \theta)=c_0(\theta)\Lambda(\ud \theta)$. 
If in addition, we assume \eqref{eq:locsym} and \eqref{layered0}, then  $c_0$ is a  symmetric function and therefore  $\Lambda_1$, too.
The vectors $\eta_{\alpha,\beta}$ and $\fb_{\alpha,\beta}$ are explicit and their formulas are 
given in the statement  of Theorem~3.1 in \cite{HOU}.
To be precise, the authors in \cite{HOU} state the stronger tail condition (3.3) on the L\'evy 
measure~$\nu$. However,~in their proof of Theorem~3.1 in \cite{HOU} which treats short-range behavior, 
it is only used to guarantee the following (according to their notation):
for $f$ being a bounded continuous function  vanishing in a neighbourhood of the origin, and $h>0$, 
$\e>0$, that the iterated integral below
is bounded independently of $h$
\begin{align*}
\int_{\SSS^{d-1}} \Lambda(\ud \xi)\int_{\e}^{\infty} f(h^{-1/\al}\xi)q(r,\xi)\ud r.
\end{align*}
In our setting of Definition~\ref{hyp: layered}, it is bounded by 
\begin{align*}
|f|_\infty\int_{\SSS^{d-1}} \Lambda(\ud \xi)\int_{\e}^{1}  q(r,\xi)\ud r+
|f|_\infty \nu_\infty(B^c_1(0)),
\end{align*}
which is finite for any $\e\in (0,1)$. As a consequence, \eqref{eq:slys0} for $L$ being strongly locally layered stable.
\end{rem}
\begin{prop}[Local limit theorem for the inhomogeneous O-U approximation]\label{tomate}\hfill\\
Assume that  $\nu$ is a strongly locally layered stable L\'evy measure in the sense of Definition~\ref{hyp: layered} with 
parameters $(\nu_0,\nu_\infty,\Lambda,q,c_0,\alpha)$.
Let $\alpha\in (0,2)$ and $\beta>0$, where $\alpha$ and $\beta$ are given in Definition~\ref{hyp: layered}. 
Then for any $K>0$ we have 
\[
\lim\limits_{\e\ra 0} \sup_{|x|\lqq K} \norm{(\gamma_\e U^x_{\e}+a^x_\e)  - U} = 0,
\]
where $\gamma_\e:=\Delta^{-\nicefrac{1}{\alpha}}_\e$, the random vector $U$ has a symmetric $\alpha$-stable distribution
with spectral density $\Lambda_1(\ud \theta)=c_0(\theta)\Lambda(\ud \theta)$,
 and the deterministic vector $a^x_\e$ is given by
\begin{equation}\label{vectortomate}
a^x_\e=\Delta^{1-\nicefrac{1}{\alpha}}_{\e}\eta_{\alpha,\beta}-\fb_{\alpha,\beta}-
\gamma_\e \Big(\eta_{\alpha,\beta}-\frac{\fb_{\alpha,\beta}}{\Delta^{1-\nicefrac{1}{\alpha}}_\e}\Big)(\Phi^{\e}_{\Delta_\e}(x))^{-1}
\int_0^{\Delta_\e}\Phi^{\e}_s (x)Db(\varphi^x_{T^x_\e+s}) 
s \ud s.
\end{equation}
In particular, for $x = 0$ we have $\lim\limits_{\e\ra 0} \norm{(\gamma_\e U^0_{\e}+a^0_\e)  - U} = 0$. 
\end{prop}
\begin{proof}[Proof of Proposition~\ref{tomate}:] By the continuity shown in Lemma~\ref{lem: continuidad} in Appendix~\ref{Appendix C} we have for any $\e>0$ a point $x_\e \in \RR^d$ with $|x_\e|\lqq K$ such that 
\[
\sup_{|x|\lqq K} \norm{(\gamma_\e U^x_{\e}+a^x_\e)  - U} 
= \norm{(\gamma_\e U^{x_\e}_{\e}+a^{x_\e}_\e)  - U}.
\]
In the sequel, we show that the right-hand side tends to $0$ as $\e\ra 0$. 
For simplicity, we drop the $\e$-dependence of $x_\e$ which is denoted by $x$. 
We stress that in the proof below the dependence of $x$ only enters in terms of 
$|x|$, which is uniformly bounded by $K$.  

\noindent We show the existence of the distributional limit $\lim_{\e\ra 0}(\gamma_\e U^x_\e+a^x_\e)$ for a suitable deterministic scale $\gamma_\e$ such that $\lim_{\e\ra 0}\gamma_\e=\infty$ and a deterministic vector $a^x_\e$.
By \eqref{d:SConvInhom} and since the process $L$ is additive,
it is not hard to deduce that its characteristic function has the following shape
\[
z\mapsto
\mathbb{E}\left[e^{i\<z,U^x_\e\>}\right]=\exp\left(\int_{0}^{\Delta_\e}\psi\left((\Phi^{\e}_s(x))^{*}((\Phi^{\e}_{\Delta_\e}(x))^{-1})^*z\right) \ud s\right),
\quad z\in \mathbb{R}^d.
\]
The translation invariance of the Lebesgue integral in the preceding exponent  implies that in distribution 
$U^x_\e=(\Phi^\e_{\Delta_\e}(x))^{-1}
\int_{0}^{\Delta_\e}  \Phi^\e_s(x) \ud L_{s}$. Integration by parts yields
\begin{align}\label{distri}
U^x_\e &\stackrel{d}{=}
 L_{\Delta_\e} - (\Phi^{\e}_{\Delta_\e}(x))^{-1} \int_0^{\Delta_\e} 
 \dot{\Phi}^{\e}_s(x) L_s \ud s \nonumber\\
&= L_{\Delta_\e} -( \Phi^{\e}_{\Delta_\e}(x))^{-1} \int_0^{\Delta_\e} \Phi^{\e}_s(x) Db(\varphi^x_{T^x_\e+s}) L_s \ud s=J_1 -J_2, 
\end{align}
where 
\begin{align*}
J_1:= L_{\Delta_\e} \qquad \mbox{ and }\qquad J_2:=  ( \Phi^{\e}_{\Delta_\e}(x))^{-1} \int_0^{\Delta_\e} \Phi^{\e}_s(x) Db(\varphi^x_{T^x_\e+s}) L_s \ud s.
\end{align*}
We start with the second term. Since $|\varphi^x_t|\lqq |x|$ for any $t\gqq 0$, it follows that
\begin{align}\label{eq: J2}
&\left|J_2+ \int_0^{\Delta_\e}(\Phi^{\e}_{\Delta_\e}(x))^{-1}\Phi^{\e}_s(x) Db(\varphi^x_{T^x_\e+s}) 
\Big(s\eta_{\alpha,\beta}-\frac{s}{\Delta^{1-\nicefrac{1}{\alpha}}_\e}\fb_{\alpha,\beta}\Big) \ud s \right|\nonumber\\
&\qquad\lqq  \int_0^{\Delta_\e} 
|(\Phi^{\e}_{\Delta_\e}(x))^{-1}\Phi^{\e}_s(x)Db(\varphi^x_{T^x_\e+s})| \big|L_s +s \eta_{\alpha,\beta}-\frac{s}{\Delta^{1-\nicefrac{1}{\alpha}}_\e}\fb_{\alpha,\beta}\big|\ud s \nonumber\\
&\qquad \lqq C(|x|)\Delta_\e\sup_{s\in [0,\Delta_\e]}|L_s +s \eta_{\alpha,\beta}-\frac{s}{\Delta^{1-\nicefrac{1}{\alpha}}_\e}\fb_{\alpha,\beta}|,\nonumber\\
&\qquad  = C(|x|)\Delta_\e\sup_{s\in [0,1]}|L_{s \Delta_\e } +s \Delta_\e  \eta_{\alpha,\beta}-\frac{s}{\Delta^{-\nicefrac{1}{\alpha}}_\e} \fb_{\alpha,\beta}|,
\end{align}
where the last inequality follows from inequality \eqref{cotacontractiva} in Lemma~\ref{lem: cotainfsup} in Appendix \ref{A}. 
Since $\gamma_\e=\Delta^{-1/\al}_\e$, we obtain
\begin{align}\label{eq: limitehou3}
\gamma_\e&\Big|J_2+ \int_0^{\Delta_\e}(\Phi^{\e}_{\Delta_\e}(x))^{-1}\Phi^{\e}_s(x) Db(\varphi^x_{T^x_\e+s}) 
\Big(s\eta_{\alpha,\beta}-\frac{s}{\Delta^{1-\nicefrac{1}{\alpha}}_\e}\fb_{\alpha,\beta}\Big) \ud s \Big| \nonumber\\[3mm]
&\hspace{5cm}\lqq 
C(|x|)\Delta_\e\sup_{s\in [0,1]}|\gamma_\e (L_{s \Delta_\e } +s \Delta_\e  \eta_{\alpha,\beta})-s \fb_{\alpha,\beta}|.
\end{align}
By Remark~\ref{rem: generalhoudre} we have
\begin{equation}\label{eq: limitehoudre2}
(\gamma_\e(L_{s\Delta_\e}+s\Delta_\e \eta_{\alpha,\beta})-s\fb_{\alpha,\beta})_{s\gqq 0}\stackrel{d}{\longrightarrow} \mathcal{S}_{\alpha}(\Lambda_1), \quad \e\to 0,
\end{equation}
 where $\mathcal{S}_{\alpha}(\Lambda_1)$ is a symmetric $\alpha$-stable process with spectral measure $\Lambda_1$.
 It is well-known in the literature that  
the supremum norm is continuous with respect to the Skorokhod topology, 
see Theorem~7.4.1 in Chapter 7 of \cite{WHITT}. 
Hence the continuous mapping theorem implies
\[
\sup_{s\in [0,1]}
 |\gamma_\e(L_{s\Delta_\e}+s\Delta_\e \eta_{\alpha,\beta})-s\fb_{\alpha,\beta}|
\stackrel{d}{\longrightarrow} \sup_{[0,1]}|\mathcal{S}_{\alpha}(\Lambda_1)|,\quad \e \to 0.\]
Since $\Delta_\e \to 0$, Slutsky's lemma yields
\begin{equation}\label{eq28}
\Delta_\e\cdot \sup_{s\in [0,1]}
 |\gamma_\e(L_{s\Delta_\e}+s\Delta_\e \eta_{\alpha,\beta})-s\fb_{\alpha,\beta}|
\stackrel{d}{\longrightarrow} 0,\quad \e \to 0.
\end{equation}
As a consequence the right-hand side of \eqref{eq: limitehou3} tends to zero, as $\e \to 0$.

We continue with the first term $J_1$.
Since $J_1=L_{\Delta_\e}$, limit \eqref{eq: limitehoudre2} implies
\begin{equation}\label{layer}
\lim\limits_{\e \to 0}\big(\Delta^{-\nicefrac{1}{\alpha}}_{\e}L_{\Delta_\e}+\Delta^{1-\nicefrac{1}{\alpha}}_{\e}\eta_{\alpha,\beta}-\fb_{\alpha,\beta}\big) \stackrel{d}{=} U,
\end{equation}
where $\eta_{\alpha,\beta}$ and $\fb_{\alpha,\beta}$ are deterministic vectors on $\mathbb{R}^d$,
and $U$ has a symmetric $\alpha$-stable distribution with spectral measure $\Lambda_1$.
By \eqref{vectortomate} and \eqref{distri} we obtain
\[
\begin{split}
&\gamma_\e U^x_\e+a^{x}_\e\\
&\quad=\gamma_\e U^x_\e+\Delta^{1-\nicefrac{1}{\alpha}}_{\e}\eta_{\alpha,\beta}-\fb_{\alpha,\beta}-
\gamma_\e\int_0^{\Delta_\e}(\Phi^{\e}_{\Delta_\e}(x))^{-1}\Phi^{\e}_s(x) Db(\varphi^x_{T^x_\e+s}) 
\Big(s\eta_{\alpha,\beta}-\frac{s}{\Delta^{1-\nicefrac{1}{\alpha}}_\e}\fb_{\alpha,\beta}\Big) \ud s
\\
&\quad\stackrel{d}{=}
(\gamma_\e L_{\Delta_\e}+\Delta^{1-\nicefrac{1}{\alpha}}_{\e}\eta_{\alpha,\beta}-\fb_{\alpha,\beta})\\
& \qquad -\gamma_\e\Big( J_2+
\int_0^{\Delta_\e}(\Phi^{\e}_{\Delta_\e}(x))^{-1}\Phi^{\e}_s(x) Db(\varphi^x_{T^x_\e+s}) 
\Big(s\eta_{\alpha,\beta}-\frac{s}{\Delta^{1-\nicefrac{1}{\alpha}}_\e}\fb_{\alpha,\beta}\Big) \ud s\Big).
\end{split}
\]
By \eqref{eq: limitehou3}, \eqref{eq28} and \eqref{layer} we deduce with the help of Slutsky's lemma
\begin{equation}\label{eq: limiteresultado}
\lim\limits_{\e \to 0}(\gamma_\e U^x_\e+a^x_\e)\stackrel{d}{=}
U, 
\end{equation}
where $a^x_\e$ is given in\eqref{vectortomate}. 
We stress that the dependence of $x$ in the preceding limit only enters 
via $C(|x|)$ in \eqref{eq: limitehou3} and holds uniformly for $|x|\lqq K$. 

Finally, we strengthen the convergence in distribution in \eqref{eq: limiteresultado}
to the convergence in total variation distance, using the regularity of the densities and showing their convergence in $L^1{(\RR^d)}$. 
This can be carried out using the Fourier inversion formula of the explicit 
characteristic function of the linear process $\gamma_\e U^x_\e+a^x_\e$
and the Orey-Masuda condition in Lemma~\ref{lem oreymasuda}, analogously as in the proof of Lemma~\ref{lem: cvtlineal} in Appendix~\ref{Appendix C}. Since this procedure is spelt out in full detail in Lemma~\ref{lem: cvtlineal} for the limit $\lim\limits_{t\to \infty}Y^{\e}_t(x)\stackrel{d}= Z_\infty$ established in Lemma~\ref{lem:convergenceindist}  in Appendix~\ref{Appendix C} we refrain from repeating it here. 
\end{proof}

\bigskip 

\subsubsection{\textbf{{Proof of Proposition~\ref{prop: lic}}}}\hfill\\

\noindent In this subsection we establish an upper bound  of
\[\norm{Y^{\e,x}(\Delta_\e;T^x_\e,X^{\e}_{T^x_\e}(x)) - Y^{\e,x}(\Delta_\e; T^x_\e,
Y^{\e,x}(T^x_\e;0,x))}\] 
with the help of Proposition~\ref{tomate}, 
which tends to zero as $\e\to 0$. 

\begin{proof}[Proof of Proposition~\ref{prop: lic}:]
\noindent For short, let $z=X^{\e}_{T^x_\e}(x)$ and $\ti z = Y^{\e,x}(T^x_\e;0,x)$.  
The shift and scale invariance of the total variation distance and representation \eqref{eq: representation} yield
\begin{align}
&\norm{Y^{\e,x}(\Delta_\e;T^x_\e,z) - Y^{\e,x}(\Delta_\e;T^x_\e,\ti z)}\nonumber\\
&\qquad=
\norm{\Big(\frac{\gamma_\e}{\e }(\Phi^\e_{\Delta_\e}(x))^{-1}z+ \gamma_{\e} U^{x}_\e+
a^{x}_\e \Big) - \Big(\frac{\gamma_\e}{\e }(\Phi^\e_{\Delta_\e}(x))^{-1}\ti z+ \gamma_\e U^{x}_\e +a^{x}_\e \Big)}=:I_1,\label{e:P5I1}
\end{align}
where $a^x_\e$ is given in \eqref{vectortomate} and $\gamma_\e$ being given in Proposition~\ref{tomate}.
The triangle inequality yields
\begin{align*}
I_1& \lqq \norm{\Big(\frac{\gamma_\e}{\e }(\Phi^\e_{\Delta_\e}(x))^{-1}z+ \gamma_{\e} U^{x}_\e+
a^{x}_\e \Big)-\Big(\frac{\gamma_\e}{\e }(\Phi^\e_{\Delta_\e}(x))^{-1}z+ U \Big)}\\
&\quad+\norm{\Big(\frac{\gamma_\e}{\e }(\Phi^\e_{\Delta_\e}(x))^{-1}z+ U \Big)-\Big(\frac{\gamma_\e}{\e }(\Phi^\e_{\Delta_\e}(x))^{-1}\ti z+ U \Big)}\\
&\quad+\norm{\Big(\frac{\gamma_\e}{\e }(\Phi^\e_{\Delta_\e}(x))^{-1}\ti z+ U \Big)-\Big(\frac{\gamma_\e}{\e }(\Phi^\e_{\Delta_\e}(x))^{-1}\ti z+ \gamma_{\e} U^{x}_\e+
a^{x}_\e \Big)},
\end{align*}
where $U$ has a $\mathcal{S}_{\alpha}(\Lambda_1)$ distribution given in Proposition~\ref{tomate}.
The independence of the increments of $L$ yields that $(\Phi^\e_t(x))^{-1} z$ and $(\Phi^\e_t(x))^{-1}\ti z$ are independent of  $U^x_\e$ and $U$, respectively. 
Then
the cancellation property of independent shifts in the total variation distance
given in Item ii) of Lemma~A.2 of \cite{BP} yields
\begin{align}
I_1& \lqq 2\norm{\Big(\gamma_{\e} U^{x}_\e+
a^{x}_\e\Big)- U}
+\norm{\Big(\frac{\gamma_\e}{\e }(\Phi^\e_{\Delta_\e}(x))^{-1}z+ U \Big)-\Big(\frac{\gamma_\e}{\e }(\Phi^\e_{\Delta_\e}(x))^{-1}\ti z+ U \Big)}.\label{e:P5I11}
\end{align}
We prove that the right-hand side of the preceding inequality tends to zero as $\e \ra 0$. By Proposition~\ref{tomate} it  remains to prove that 
\[
\lim\limits_{\e \ra 0}\norm{\Big(\frac{\gamma_\e}{\e }(\Phi^\e_{\Delta_\e}(x))^{-1}z+ U \Big)-
\Big(\frac{\gamma_\e}{\e }(\Phi^\e_{\Delta_\e}(x))^{-1}\ti z+ U \Big)}=0.
\]
Let $\PP^x_\e(\ud u, \ud \tilde{u})$ denote the joint probability measure  
$\PP\left(X^{\e}_{T^x_\e}(x)\in \ud u, Y^{\e,x}({T^x_\e};0, x)\in \ud\ti u\right)$ and keep the notation 
$z=X^{\e}_{T^x_\e}(x)$ and $\ti z = Y^{\e,x}(T^x_\e; 0,x)$. 
Since $z$ and $\ti z$ are nondegenerate and mutually dependent random variables the {\it{shift property}} for the total variation distance cannot be applied directly. Nevertheless, the Markov property  and the shift invariance allow to disintegrate $\PP^x_\e$  as follows
\begin{align}\label{chile}
&\norm{\Big(\frac{\gamma_\e}{\e }(\Phi^\e_{\Delta_\e}(x))^{-1}z+ U \Big)-
\Big(\frac{\gamma_\e}{\e }(\Phi^\e_{\Delta_\e}(x))^{-1}\ti z+ U \Big)}\nonumber\\
&\hspace{1cm}\lqq \int A^{\e}_1(\zeta,\ti \zeta)
\PP^x_\e (\ud \zeta, \ud \ti \zeta)=\int A^{\e}_2(\zeta,\ti \zeta)
\PP^x_\e (\ud \zeta, \ud \ti \zeta),
\end{align}
where 
\begin{align*}
&A^{\e}_1(\zeta,\ti \zeta):=\norm{\Big(\frac{\gamma_\e}{\e }(\Phi^\e_{\Delta_\e}(x))^{-1}\zeta+ U \Big)-
\Big(\frac{\gamma_\e}{\e }(\Phi^\e_{\Delta_\e}(x))^{-1}\ti \zeta+ U \Big)},\\
&A^{\e}_2(\zeta,\ti \zeta):=\norm{\Big(\frac{\gamma_\e}{\e }(\Phi^\e_{\Delta_\e}(x))^{-1}(\zeta-\ti \zeta)+ U \Big)-
U }.
\end{align*}
We continue with the following split.  For any $ \eta>0$ we consider
\begin{equation}\label{fresa}
A^{\e}_2(\zeta,\ti \zeta)=A^{\e}_2(\zeta,\ti \zeta)\ind\{\gamma_\e |{\zeta - {\ti \zeta}}| > \eta \e \}
+A^{\e}_2(\zeta,\ti \zeta)\ind\{\gamma_\e |{\zeta - {\ti \zeta}}| \lqq \eta \e \}.
\end{equation}
We start with the second term on the right-hand side of \eqref{fresa}. Since the shift operator is continuous at $0$ in $L^1(\mathbb{R}^d)$ for any $\rho>0$, there exists $\eta=\eta(\rho)>0$ such that 
\begin{equation}\label{e: fresa}
\norm{(h+U)-U}<\rho\quad \textrm{whenever}\quad |h|\lqq \sqrt{d} \eta.
\end{equation}
By Lemma~\ref{lem: cotainfsup} in Appendix~\ref{A} we obtain $|(\Phi^\e_{\Delta_\e})^{-1}(x)|\lqq \sqrt{d}$ for any $\e\in (0,1]$ and $x\in \RR^d$, where  $|\cdot|$ denote the standard matrix $2$-norm which in abuse of notation we also denote by $|\cdot|$.
By  Hypothesis~\ref{hyp: potential}
the event $\{\gamma_\e |{\zeta - \ti \zeta}|\lqq  \eta \e \}$ implies 
\[
\Big|\frac{\gamma_\e}{\e }(\Phi^\e_{\Delta_\e}(x))^{-1}(\zeta-\ti \zeta)\Big|\lqq \sqrt{d} \eta
\quad\textrm{ for any } \e.
\]
The preceding  estimate implies
\begin{align}\label{cafe}
\int A^{\e}_2(\zeta,\ti \zeta)
\ind\{\gamma_\e |{\zeta - \ti \zeta}|\lqq  \eta \e \}
\PP^x_\e (\ud \zeta, \ud \ti \zeta) 
\lqq \rho\,\PP^x_\e(\gamma_\e |{\zeta - \ti \zeta}|\lqq  \eta \e)\lqq \rho.
\end{align}
On the other hand, for any $\eta>0$ we have
\begin{align}\label{cafecito}
\int A^{\e}_2(\zeta,\ti \zeta)
\ind\{\gamma_\e |{\zeta - \ti \zeta}|>  \eta \e \}
\PP^x_\e (\ud \zeta, \ud \ti \zeta)
\lqq  \PP^x_\e(\gamma_\e |{\zeta - \ti \zeta}|> \eta \e).
\end{align}
Combining \eqref{chile}-\eqref{cafecito} we obtain
\begin{align*}
\limsup_{\e \ra 0} \norm{\Big(\frac{\gamma_\e}{\e }(\Phi^\e_{\Delta_\e}(x))^{-1}z+ U \Big)-
\Big(\frac{\gamma_\e}{\e }(\Phi^\e_{\Delta_\e}(x))^{-1}\ti z+ U \Big)}\lqq \limsup_{\e \ra 0}\PP(\gamma_\e |z - \ti z|> \eta \e)  + \rho
\end{align*}
for any $\rho>0$. Note that $\eta$ depends on $\rho$.
Sending $\rho \ra 0$ we obtain
\begin{align}\label{cerveza}
\limsup_{\e \ra 0} \norm{\Big(\frac{\gamma_\e}{\e }(\Phi^\e_{\Delta_\e}(x))^{-1}z+ U \Big)-
\Big(\frac{\gamma_\e}{\e }(\Phi^\e_{\Delta_\e}(x))^{-1}\ti z+ U \Big)}\nonumber\\
\hspace{1cm}\lqq 
\limsup_{\rho \ra 0}\limsup_{\e \ra 0}\PP(\gamma_\e |z - \ti z|> \eta \e).
\end{align}
By Lemma \ref{lem: dependence} stated below we obtain that the upper bound on the right-hand side of \eqref{cerveza} tends to zero, $\e \to 0$, which together with inequality \eqref{e:P5I1}, \eqref{e:P5I11} and Proposition~\ref{tomate} implies \eqref{eq: limiteinhomogeneous}.
\end{proof}

\begin{lem}\label{lem: dependence}
Let $\gamma_\e={\Delta_\e}^{-\nicefrac{1}{\alpha}}$, where $\Delta_\e=\e^{\nicefrac{\alpha}{2}}$.
Then it follows
\[
\limsup_{\e \ra 0}\PP(\gamma_\e |z - \ti z|> \eta \e)=0 \quad \textrm{ for any } \eta>0.
\]
\end{lem}
\begin{proof}
Let $\eta>0$ and recall $\gamma_\e={\Delta_\e}^{-\nicefrac{1}{\alpha}}$ for some $\alpha\in (0,2)$ as 
in Proposition~\ref{tomate}. Then we observe
\begin{align}\label{tomillo}
\PP\left(|z - {\ti z}|>  \nicefrac{\eta\e }{\gamma_\e}\right)
= \PP\big(| X^{\e}_{T^x_\e}(x) - Y^{\e,x}(T^x_\e;0,x)|> \eta \e {\Delta_\e}^{\nicefrac{1}{\alpha}} \big)
\to 0, \quad \textrm{as } \e\to 0.
\end{align}
Since $T^x_\e={O}(|\ln(\e)|)$, Proposition~\ref{prop: secondorder} in Appendix \ref{ap: multiscale} yields
\[
\limsup_{\e \ra 0}\PP(\gamma_\e |z - \ti z|> \eta \e)=0 \quad \textrm{ for any } \eta>0. 
\]
\end{proof}

\bigskip

\subsection{\textbf{Nonlinear short-time coupling (Proposition~\ref{prop: stc})}}\label{sec: short coupling}
\hfill\\

\noindent We keep the notation introduced in Subsection~\ref{subsec: outline}.
Let $\Delta_\e=\e^{\nicefrac{\alpha}{2}}$.
For any $\rho\in \mathbb{R}$, recall that
$T^x_\e=t^x_\e-\Delta_\e+\rho\cdot w^x_\e$, where  $t^x_\e$ and $w^x_\e$  are given in Theorem~\ref{thm: main result}.
We show the following:
\[
\lim\limits_{\e\to 0}
\norm{X^{\e}_{\Delta_\e}(X^{\e}_{T^x_\e}(x))-
Y^{\e,x}(\Delta_\e;T^x_\e, X^{\e}_{T^x_\e}(x))}=0.
\]
\noindent
Recall that $(\varphi^x_t)_{t\gqq 0}$ is the solution of \eqref{dde1.1}. By \eqref{eq: linealhomogenea}
the variation of constants formula yields the explicit representation
\begin{align}\label{eq: representation0}
Y^{\e,x}(t;T^x_\e,z)=(\Phi^\e_t(x))^{-1}z+
(\Phi^\e_t(x))^{-1}
\int_{0}^{\Delta_\e}
\Phi^\e_s(x)\big(Db(\varphi^x_{T^x_\e+s})\varphi^x_{T^x_\e+s}-b(\varphi^{x}_{T^x_\e+s})\big)\ud s+
\e U^{x}_\e,
\end{align}
where $(\Phi^\e_t(x))_{t\gqq 0}$ is the solution of the  matrix valued inhomogeneous differential equation given in \eqref{d:SGInhom} 
and the random vector $U^x_\e$ is defined by \eqref{d:SConvInhom}.
For any $z\in \RR^d$ we consider the unique strong solution $(Z^{\e}_t(z))_{t\gqq 0}$ of
\[
\ud Z^{\e}_t=-Db(0)Z^{\e}_t\ud t+\e \ud L_t\quad \textrm{ with } Z^{\e}_0=z.
\]
The variation of constant formula yields the representation
\[
Z^\e_{\Delta_\e}(X^\e_{T^x_\e}(x)) 
= \Psi_{\Delta_\e}^{-1} X^\e_{T^x_\e}(x) + \e\ti U^x_\e,  \quad \textrm{ where } \quad
\ti U^x_\e = (\Psi_{\Delta_\e}^{-1}) \int_0^{\Delta_\e} \Psi_s \ud L_{T^x_\e +s}
\]
and $\Psi_t = e^{Db(0) t}, t\in \RR$. It is easily seen that $\Psi^{-1}_t = \Psi_{-t}$.   
We start with the estimate 
\begin{align}
 &\norm{X^\e_{\Delta_\e}(X^\e_{T^x_\e}(x)) - Y^{\e,x}(\Delta_\e;T^x_\e,X^\e_{T^x_\e}(x))} \nonumber\\
 &\qquad \lqq \norm{X^\e_{\Delta_\e}(X^\e_{T^x_\e}(x)) - Z^\e_{\Delta_\e}(X^\e_{T^x_\e}(x))}
 + \norm{Z^\e_{\Delta_\e}(X^\e_{T^x_\e}(x)) - Y^{\e,x}(\Delta_\e;T^x_\e,X^\e_{T^x_\e}(x))}\nonumber\\[2mm]
 &\qquad = G_1 +G_2,\label{e:G1G2}
 \end{align}
where 
\begin{align*}
G_1&:=\norm{X^\e_{\Delta_\e}(X^\e_{T^x_\e}(x)) - Z^\e_{\Delta_\e}(X^\e_{T^x_\e}(x))},\\
G_2&:=\norm{Z^\e_{\Delta_\e}(X^\e_{T^x_\e}(x)) - Y^{\e,x}(\Delta_\e;T^x_\e,X^\e_{T^x_\e}(x))}.
\end{align*}

\bigskip 
\subsubsection{\textbf{{Step 1: Domination of the error term $G_2$}}}\label{sss: Step1} \hfill\\

\noindent In this subsubsection  
we estimate the second term on the right-hand side of \eqref{e:G1G2}. 

\begin{lem}\label{lem:G2} $G_2\ra 0$ as $\e \ra 0$. \end{lem}

\begin{proof}[Proof of Lemma~\ref{lem:G2}:] Let 
\begin{equation}\label{def:Gamma}
\Gamma^x_{\e}:=
(\Phi^{\e}_{\Delta_\e}(x))^{-1}\int_{0}^{\Delta_\e}
\Phi^{\e}_s(x)\big(
b(\varphi^{x}_{T^x_\e+s})-Db(\varphi^{x}_{T^x_\e+s}) \varphi^x_{T^x_\e+s} \big)\ud s.
\end{equation}
By disintegration combined with the translation and scale invariance of the total variation distance, we obtain 
\begin{align*}
&G_2 =\norm{\Big(\Psi_{\Delta_\e}^{-1} X^\e_{T^x_\e}(x) + \e \ti U^x_\e \Big) - 
  \Big((\Phi^{\e}_{\Delta_\e}(x))^{-1}X^\e_{T^x_\e}(x)-\Gamma^x_{\e} + \e U^x_\e \Big)}\\
&\qquad\lqq \int_{\RR^d } \norm{\Big(\Psi_{\Delta_\e}^{-1} z+ \e \ti U^x_\e \Big) - 
  \Big((\Phi^{\e}_{\Delta_\e}(x))^{-1}z-\Gamma^x_{\e} + \e U^x_\e \Big)} \PP^x_\e(\ud z)\\
&\qquad=\int_{\RR^d} \norm{\bigg(
\frac{\big(\Psi_{\Delta_\e}^{-1}-(\Phi^{\e}_{\Delta_\e}(x))^{-1}\big) z+\Gamma^x_{\e}}{\e}+ \ti U^x_\e \bigg) - 
  U^x_\e} \PP^x_\e(\ud z), 
  \end{align*}
where $\PP^x_\e(\ud z):= \PP(X^\e_{T^x_\e}(x) \in \ud z)$. 
By Proposition~\ref{tomate} 
there exists a random variable 
$U\stackrel{d}= \mathcal{S}_\alpha(\Lambda_1)$ and the deterministic vector  
$a^x_\e\in \RR^d$  defined in \eqref{vectortomate} such that
\begin{equation}\label{driftsintilde}
\norm{(\Delta_\e^{-\nicefrac{1}{\alpha}} U^x_\e+a^x_{\e})- U}\to 0,\quad
\textrm{ as } \e \ra 0.
\end{equation}
Repeating the same argument of  Proposition~\ref{tomate}, we have that
there exists a random variable $\tilde{U}\stackrel{d}= \mathcal{S}_\alpha(\Lambda_1)$ and the deterministic vector  
$a^0_\e\in \RR^d$  given by
\[
a^0_\e=\Delta^{1-\nicefrac{1}{\alpha}}_{\e}\eta_{\alpha,\beta}-\fb_{\alpha,\beta}-
\gamma_\e \Big(\eta_{\alpha,\beta}-\frac{\fb_{\alpha,\beta}}{\Delta^{1-\nicefrac{1}{\alpha}}_\e}\Big)(\Psi_{\Delta_\e})^{-1}
\int_0^{\Delta_\e}\Psi_sDb(0) 
s \ud s
\]
 such that
\begin{equation}\label{drifttilde}
\norm{(\Delta_\e^{-\nicefrac{1}{\alpha}} \tilde{U}^x_\e+a^0_{\e})- \tilde{U}}\to 0,\quad
\textrm{ as } \e \ra 0.
\end{equation} 
We define the deterministic function
\begin{equation}\label{fg}
g^{x}_\e (z)=\frac{
\left(
(\Psi_{\Delta_\e})^{-1}
-(\Phi^{\e}_{\Delta_\e}(x))^{-1}
\right) z+\Gamma^x_{\e}}
{\e},\qquad z \in \RR^d
\end{equation}
and the pivotal terms
\begin{align*}
&B^{\e}_0(z):=\norm{\Big(g^x_{\e}(z)+\ti U^x_\e\Big) - 
  U^x_\e },\\
&B^{\e}_1(z):=\norm{\Big(\Delta_\e^{-\nicefrac{1}{\alpha}}g^x_{\e}(z)+\Delta_\e^{-\nicefrac{1}{\alpha}}\ti U^x_\e+a^0_\e\Big) - 
  \Big(\Delta_\e^{-\nicefrac{1}{\alpha}}U^x_\e+ a^0_\e\Big)},\\ 
&B^{\e}_2(z):= \norm{\Big(\Delta_\e^{-\nicefrac{1}{\alpha}}g^x_{\e}(z)+\Delta_\e^{-\nicefrac{1}{\alpha}}\ti U^x_\e+a^0_\e\Big) -
\Big(\Delta_\e^{-\nicefrac{1}{\alpha}}g^x_{\e}(z)+ \ti U\Big)},\\  
&B^{\e}_3(z):= \norm{
\Big(\Delta_\e^{-\nicefrac{1}{\alpha}}g^x_{\e}(z)+ \ti U\Big)
-U
},\\  
&B^{\e}_4(z):= \norm{U-
\Big(\Delta_\e^{-\nicefrac{1}{\alpha}}U^x_\e+a^0_\e\Big)}.
\end{align*}
The scale and shift invariance of the total variation distance combined with the triangle inequality yield
\begin{align}\label{e:pivpiv}
B^{\e}_0(z)=B^{\e}_1(z)\lqq B^{\e}_2(z)+B^{\e}_3(z)+B^{\e}_4(z).
\end{align}
\noindent
\textbf{Estimate of $B^{\e}_2(z)$ in \eqref{e:pivpiv}.}
By the cancellation property of independent increments in the total variation distance we have
\[
B^{\e}_2(z)\lqq  \norm{\Big(\Delta_\e^{-\nicefrac{1}{\alpha}}\ti U^x_\e+a^0_\e\Big) - \ti U^x}\to 0,\quad \textrm{ as } \e\to 0,
\]
due to \eqref{drifttilde}.
As a consequence, we have
\[
\int_{\RR^d} B^{\e}_2(z)\PP^x_{\e}(\ud z)\lqq \norm{\Big(\Delta_\e^{-\nicefrac{1}{\alpha}}\ti U_\e+a^0_\e\Big) - \ti U}\to 0,\quad \textrm{ as } \e\to 0.
\]
\noindent
\textbf{Estimate of $B^{\e}_4(z)$ in \eqref{e:pivpiv}.}
Analogously to $B^\e_2(z)$ we have
\begin{align*}
B^{\e}_4(z)&=  \norm{\Big(\Delta_\e^{-\nicefrac{1}{\alpha}}U^x_\e+ a^0_\e\Big) -  U}=
\norm{\Big(\Delta_\e^{-\nicefrac{1}{\alpha}}U^x_\e+a^x_\e+a^0_\e\Big) - \Big(U+a^x_\e\Big)}\\
&\lqq \norm{\Big(\Delta_\e^{-\nicefrac{1}{\alpha}}U^x_\e+a^x_\e+a^0_\e\Big) - \Big(U+a^0_\e\Big)}+
\norm{\Big(U+ a^0_\e\Big)-\Big(U+a^x_\e\Big)}\\
&= \norm{\Big(\Delta_\e^{-\nicefrac{1}{\alpha}}U^x_\e+a^x_\e\Big) -U}+
\norm{\Big(a^0_\e-a^x_\e+U\Big)-U}.
\end{align*}
Due to \eqref{driftsintilde} we obtain 
\[
\norm{\Big(\Delta_\e^{-\nicefrac{1}{\alpha}}U^x_\e+a_\e^x\Big) -U}\to 0,\quad \textrm{ as  } \e \to 0.
\]
Proposition~\ref{tomate} yields
\begin{align*}
&a^x_\e-a^0_\e\\
&=
-\Delta_\e^{-\nicefrac{1}{\alpha}}
\Big(\eta_{\alpha,\beta}-\frac{1}{\Delta^{1-\nicefrac{1}{\alpha}}_\e}\fb_{\alpha,\beta}\Big)
\int_0^{\Delta_\e} s
\Big((\Phi^{\e}_{\Delta_\e}(x))^{-1}\Phi^{\e}_s(x) Db(\varphi^x_{T^x_{\e}+s})-(\Psi_{\Delta_\e})^{-1}\Psi_s Db(0)\Big)\ud s,
\end{align*}
such that    
\begin{equation}\label{est:desi}
\begin{split}
&|a_\e-a^0_\e|\\
&\lqq \Delta_\e^{1-\nicefrac{1}{\alpha}}
\Big |\eta_{\alpha,\beta}-\frac{1}{\Delta^{1-\nicefrac{1}{\alpha}}_\e}\fb_{\alpha,\beta}\Big|
\int_0^{\Delta_\e}
\Big|(\Phi^{\e}_{\Delta_\e}(x))^{-1}\Phi^{\e}_s(x) Db(\varphi^x_{T^x_{\e}+s})-(\Psi_{\Delta_\e})^{-1}\Psi_s Db(0)\Big| \ud s \\
& \lqq \Delta_\e^{1-\nicefrac{1}{\alpha}}
\Big |\eta_{\alpha,\beta}-\frac{1}{\Delta^{1-\nicefrac{1}{\alpha}}_\e}\fb_{\alpha,\beta}\Big|\int_0^{\Delta_\e}
\Big|(\Phi^{\e}_{\Delta_\e}(x))^{-1}\Phi^{\e}_s(x)\Big| \Big|Db(\varphi^x_{T^x_{\e}+s})-Db(0)\Big| \ud s\\
&\quad+\Delta_\e^{1-\nicefrac{1}{\alpha}}
\Big |\eta_{\alpha,\beta}-\frac{1}{\Delta^{1-\nicefrac{1}{\alpha}}_\e}\fb_{\alpha,\beta}\Big||Db(0)|
\int_0^{\Delta_\e}
\Big|(\Phi^{\e}_{\Delta_\e}(x))^{-1}\Phi^{\e}_s(x) -(\Psi_{\Delta_\e})^{-1}\Psi_s\Big| \ud s.
\end{split}
\end{equation}
We start with the estimate of the first term on the right-hand side. 
By Lemma~\ref{lem:ordenepsilon} in Appendix~\ref{A} there exists a positive constant $C(|x|)$ depending continuously on $|x|$
such that
\begin{equation}\label{eq: epsilonC}
|\varphi^x_{T^x_\e}|\lqq C(|x|)\e\quad \textrm{ for all } \e\ll 1.
\end{equation}
With the help of inequality \eqref{cotacontractiva} in Lemma~\ref{lem: cotainfsup} in Appendix~\ref{A}, the mean value theorem and the fact that 
$|\varphi^{x}_t|\lqq |x|$, $t\gqq 0$, we have
\begin{align*}
&\Delta_\e^{1-\nicefrac{1}{\alpha}}
\left |\eta_{\alpha,\beta}-\frac{\fb_{\alpha,\beta}}{\Delta^{1-\nicefrac{1}{\alpha}}_\e}\right|\int_0^{\Delta_\e}
\left|(\Phi^{\e}_{\Delta_\e}(x))^{-1}\Phi^{\e}_s(x)\right| \left|Db(\varphi^x_{T^x_{\e}+s})-Db(0)\right| \ud s\\
&\hspace{1cm} \lqq \sqrt{d}\Delta_\e^{1-\nicefrac{1}{\alpha}}
\left |\eta_{\alpha,\beta}-\frac{\fb_{\alpha,\beta}}{\Delta^{1-\nicefrac{1}{\alpha}}_\e}\right|
\int_0^{\Delta_\e} \left|Db(\varphi^x_{T^x_{\e}+s})-Db(0)\right| \ud s\\
&\hspace{1cm}\lqq 
C(|x|,d)\e \Delta_\e^{2-\nicefrac{1}{\alpha}}
\left |\eta_{\alpha,\beta}-\frac{\fb_{\alpha,\beta}}{\Delta^{1-\nicefrac{1}{\alpha}}_\e}\right| 
\\
&\hspace{1cm}\lqq 
C(|x|,d)\e \Delta_\e^{2-\nicefrac{1}{\alpha}}
|\eta_{\alpha,\beta}|+C(|x|,d)\e \Delta_\e|\fb_{\alpha,\beta}|.
\end{align*}
Since $\Delta_\e=\e^{\al/2}$, both preceding terms on the right-hand side tend to zero as $\e \to 0$.

\noindent
We continue with the second term on the right-hand side of \eqref{est:desi}. 
By Lemma~\ref{lem: cotainfsup}.v) in Appendix~\ref{A} we have  for $\e$ sufficiently small that
\begin{align*}
|\Phi^{-1}_{\Delta_\e}(x)\Phi_s(x) -\Psi^{-1}_{\Delta_\e}\Psi_s |& \lqq 
\frac{C_1(|x|) d^2}{2\delta}
|\varphi^x_{T^x_\e}|
e^{-\frac{\delta}{2}{\Delta_\e}}\sqrt{1-e^{-4\delta ({\Delta_\e}-s)}}\\
&\lqq 
C_1(|x|)\e,
\quad s\in [0,\Delta_\e],
\end{align*}
where $C_1(|x|)$ is a constant that depends continuously on $|x|$. 
Then for small values of $\Delta_\e$ we have
\begin{align*}
&\Delta_\e^{1-\nicefrac{1}{\alpha}}
\left |\eta_{\alpha,\beta}-\frac{\fb_{\alpha,\beta}}{\Delta^{1-\nicefrac{1}{\alpha}}_\e}\right||Db(0)|
\int_0^{\Delta_\e}
\left|(\Phi^{\e}_{\Delta_\e})^{-1}\Phi^{\e}_s -(\Psi_{\Delta_\e})^{-1}\Psi_s\right| \ud s\\
&\hspace{1cm}\lqq |Db(0)|C_1(|x|)\e\Delta_\e^{2-\nicefrac{1}{\alpha}}
\left |\eta_{\alpha,\beta}-\frac{\fb_{\alpha,\beta}}{\Delta^{1-\nicefrac{1}{\alpha}}_\e}\right|\\
&\hspace{1cm}\lqq |Db(0)|C_1(|x|)\e\Delta_\e^{2-\nicefrac{1}{\alpha}}|\eta_{\alpha,\beta}|+
|Db(0)|C_1(|x|)\e\Delta_\e|\fb_{\alpha,\beta}|.
\end{align*}
Since $\Delta_\e=\e^{\al/2}$, we have  $|a_\e-a^0_\e|\to 0$, as $\e \to 0$ and consequently by the Scheff\'e lemma for densities we obtain
\[
\norm{\big(a_\e-a^0_\e+U\big)-U}\to 0, \quad \textrm{ as } \e\to 0.
\]
With the same reasoning we get 
\[
\int_{\RR^d} B^{\e}_4(z)\PP^x_{\e}(\ud z)\to 0,\quad \textrm{ as } \e\to 0.
\]
\noindent
\textbf{Estimate of $B^{\e}_3(z)$ in \eqref{e:pivpiv}.}
The remainder of Step 1 is dedicated to show that
\[
\int_{\RR^d} B^{\e}_3(z)\PP^x_{\e}(\ud z)\to 0,\quad \textrm{ as } \e\to 0.
\]
For $\vartheta\in (0,\nicefrac{1}{4})$ we define $r_\e:=\e^{1-\vt}$ and estimate
\[
\int_{\RR^d} B^\e_3(z)\PP^x_{\e}(\ud z)\lqq \int_{|z|\lqq r_\e} B^\e_3(z)\PP^x_{\e}(\ud z)+ \mathbb{P}\big(|X^{\e}_{T^x_\e}(x)|> r_\e\big).
\]
By Lemma~\ref{cor: demomento} in Appendix~\ref{Appendix D} we have for the second term 
\[
\mathbb{P}\big(|X^{\e}_{T^x_\e}(x)|\gqq r_\e\big)\to 0, \quad
\textrm{ as } \e\to 0.
\]
We continue with the first term of the right-hand side of the preceding inequality. 
Recall that
\begin{align*}
B^{\e}_3(z)&= \norm{
\big(\Delta^{-\nicefrac{1}{\alpha}}_\e g^x_{\e}(z)+ \ti U\big)
-U},
\end{align*}
where $U$ and $\ti U$ are $\mathcal{S}_\alpha(\Lambda_1)$ distributed, and 
\[
g^x_{\e}(z)=\frac{\left((\Psi_{\Delta_\e})^{-1}-(\Phi^{\e}_{\Delta_\e}(x))^{-1}\right) z+\Gamma^x_{\e}}{\e},\qquad z\in \mathbb{R}^d,
\]
where $\Gamma^x_{\e}$ was defined in \eqref{def:Gamma}.
By Lemma~\ref{lem: cotainfsup}.v) in Appendix~\ref{A} there exists a positive constant $C(|x|)$ depending continuously on $|x|$ such that 
\[
|(\Phi^{\e}_{\Delta_\e}(x))^{-1} -\Psi^{-1}_{\Delta_\e}|\lqq 
\frac{C(|x|) d^2}{2\delta}
|\varphi^x_{T^x_\e}|
e^{-\frac{\delta}{2} \Delta_\e}(1-e^{-4\delta \Delta_\e}).
\]
The preceding inequality combined with inequality \eqref{eq: epsilonC} 
yields for $\e$ sufficiently small
\[
|(\Phi^{\e}_{\Delta_\e}(x))^{-1}-(\Psi_{\Delta_\e})^{-1}|\lqq C_1(|x|)\e\Delta_\e.
\]
Therefore,
\begin{align}\label{eq:estim1}
&\sup_{|z|\lqq r_\e}\frac{\Delta^{-\nicefrac{1}{\alpha}}_\e
|(\Phi^{\e}_{\Delta_\e}(x))^{-1}z-(\Psi_{\Delta_\e})^{-1}z|}{\e}\nonumber\\
&\qquad\lqq 
\sup_{|z|\lqq r_\e}\frac{\Delta^{-\nicefrac{1}{\alpha}}_\e
|(\Phi^{\e}_{\Delta_\e}(x))^{-1}-(\Psi_{\Delta_\e})^{-1}|
|z|}{\e}\lqq C_1(|x|)\Delta^{1-\nicefrac{1}{\alpha}}_\e r_\e.
\end{align}
It remains to estimate
\[
\begin{split}
&|\Gamma^x_\e|=\Big|
\int_{0}^{\Delta_\e}
\big[(\Phi^{\e}_{\Delta_\e}(x))^{-1}\Phi^{\e}_s(x)-(\Psi_{\Delta_\e})^{-1}\Psi_s\big]\big[
b(\varphi^{x}_{T^x_\e+s})-Db(\varphi^{x}_{T^x_\e+s}) \varphi^x_{T^x_\e+s} \big]\ud s 
\Big|
\\
& \lqq C(|x|,d)\int_{0}^{\Delta_\e}
\big|
b(\varphi^{x}_{T^x_\e+s})-Db(\varphi^{x}_{T^x_\e+s}) \varphi^x_{T^x_\e+s} \big|\ud s \lqq C_1(|x|,d)\int_{0}^{\Delta_\e}
\big|\varphi^x_{T^x_\e+s} \big|^2\ud s\lqq C\e^2 \Delta_\e,
\end{split}
\]
where the last inequality follows from Lemma~\ref{lem:ordenepsilon} in Appendix~\ref{A}.
As a consequence we have
\begin{equation}\label{eq:estim2}
\frac{\Delta^{-\nicefrac{1}{\alpha}}_\e |\Gamma^x_\e|}{\e}\lqq C_1(|x|,d)\e\Delta^{1-\nicefrac{1}{\alpha}}_\e.
\end{equation}
Finally we estimate 
\[
\int_{|z|\lqq r_\e} B^\e_3(z)\mathbb{P}^x_{\e}(\ud z)\lqq 
\sup_{|z|\lqq r_\e}B^\e_3(z).
\]
The  continuity of the shift operator in $L^{1}$ and the compactness of the Euclidean closed ball imply
\[
\sup_{|z|\lqq r_\e}B^\e_3(z)=B^{\e}_{3}(z_\e)\quad  \textrm{ for some } |z_\e| \lqq r_\e.
\]
Since $\Delta_\e=\e^{\al/2}$, the preceding inequality combined with estimates \eqref{eq:estim1} and
\eqref{eq:estim2} yields 
\[
\int_{|z|\lqq r_\e} B^\e_3(z)\mathbb{P}^x_{\e}(\ud z)\to 0, \quad \textrm{ as }
\e\to 0.
\]
This finishes the proof of Lemma~\ref{lem:G2}. 
\end{proof}

\bigskip
\subsubsection{\textbf{Step 2: Domination of the error term $G_1$ up to a term in distribution}}\hfill\\

In the sequel we treat the error term $G_1$ in two consecutive steps (Step 2 and Step 3). 
By the end of Step 3 (Subsubsection \ref{ss:Step3}) we obtain the desired result $G_1 \ra 0$ as $\e \ra 0$ by a suitable localization procedure combined with the Fourier inversion technique applied to the result of Step 2. 

\noindent First note that by disintegration we have 
\begin{equation} \label{eq: pivot disintegration}
\begin{split}
G_1&=\norm{X^{\e}_{\Delta_\e}(X^{\e}_{T^x_\e}(x))-Z^{\e}_{\Delta_\e}(X^{\e}_{T^x_\e}(x))}\lqq 
\int_{\RR^{d}}
\norm{X^{\e}_{\Delta_\e}(z)-Z^{\e}_{\Delta_\e}(z) }\PP^x_{\e}(\ud z)\\
&=
\int_{|z|\lqq r_\e}
\norm{X^{\e}_{\Delta_\e}(z)-Z^{\e}_{\Delta_\e}(z) }\PP^x_{\e}(\ud z)+
\int_{|z|>r_\e}
\norm{X^{\e}_{\Delta_\e}(z)-Z^{\e}_{\Delta_\e}(z) }\PP^x_{\e}(\ud z)\\
&\lqq 
\sup_{|z|\lqq r_\e}
\norm{X^{\e}_{\Delta_\e}(z)-Z^{\e}_{\Delta_\e}(z) }+
\PP(|X^{\e}_{T^x_\e}(x)|\gqq r_\e),
\end{split}
\end{equation}
where 
$\PP^x_{\e}(\ud z)=\PP(X^{\e}_{T^x_\e}(x)
\in \ud z)$ and $r_\e=\e^{1-\vt}$, $\vartheta\in (0,\nicefrac{1}{4})$. 
By Lemma~\ref{cor: demomento} in Appendix~\ref{Appendix D} we have 
\begin{equation}\label{e:secondorder}
\PP(|X^{\e}_{T^x_\e}(x)|\gqq r_\e)\to 0,\quad \textrm{ as } \e\to 0.
 \end{equation}
 It remains to treat the first term on the right-hand side of \eqref{eq: pivot disintegration}. 
 The variation of constant formula yields
\begin{align}
X^{\e}_{t}(z)&=\Psi^{-1}_{t}z+\Psi^{-1}_{t}\int_{0}^{t}\Psi_{s}\ti {b}(X^{\e}_s(z))\ud s +\e
\Psi^{-1}_{t}\int_{0}^{t}\Psi_{s}\ud L_s,\label{e:VOCX}
\end{align}
and 
\begin{align}\label{e:VOCZ}
Z^{\e}_{t}(z)=\Psi^{-1}_t z+\Psi^{-1}_t\int_{0}^{t}\Psi_{s}\ud L_s,
\end{align}
where $\ti b(x)=b(x)-Db(0)x$, $x\in \RR^d$ and $\Psi_t = e^{Db(0) t}$. 
We denote 
\[
U_{\e} := \Psi^{-1}_{\Delta_\e} \int_{0}^{\Delta_\e}\Psi_s \ud  L_s
\quad \textrm{ and } \quad
D_{\e}(z) :=\Psi^{-1}_{\Delta_\e}\int_{0}^{\Delta_\e}\Psi_s \ti b( X^\e_s(z))\ud s,
\]
such that 
\begin{align*}
\frac{X^\e_{\Delta_\e}(z) - Z^\e_{\Delta_\e}(z)}{\e \Delta^{1/\alpha}_{\e}} = \Big(\frac{D_{\e}(z)}{\e \Delta^{1/\alpha}_{\e}}+\frac{1}{\Delta^{1/\al}_\e}U_{\e}+a_\e^0 -U\Big)-\Big(\frac{1}{\Delta^{1/\alpha}_{\e}}U_{\e}+a_\e^0 -U \Big),
\end{align*}
where $a_\e^0$ is given in Proposition~\ref{tomate} 
and $U$ is $\mathcal{S}_\alpha(\Lambda_1)$-distributed.
By Proposition~\ref{tomate} we have
\begin{equation}\label{e:termterm2}
\norm{U-\Big(\frac{1}{\Delta^{1/\alpha}_{\e}}U_{\e}+a_{\e}^0\Big)}\to 0,\quad \textrm{ as } \e\to 0.
\end{equation}
In this subsection we show the following. 
\begin{lem}\label{lem:indistri}
Assume Hypotheses~\ref{hyp: potential},
~\ref{hyp: moment condition}, 
~\ref{hyp: regularity} and ~\ref{hyp: blumental} are satisfied for $\alpha \in(0,2)$ and $\beta>0$.
Then 
\begin{equation}\label{eq:indistri}
\frac{D_{\e}(z)}{\e \Delta^{1/\alpha}_{\e}}+\frac{1}{\Delta^{1/\al}_\e}U_{\e}+a_\e^0 \ra U, \qquad \mbox{ as } \e \ra 0. 
\end{equation}
\end{lem}
\noindent This convergence is strengthened to the total variation distance in Step 3 below. 

\begin{proof}[Proof of Lemma \ref{lem:indistri}:]
By \eqref{e:termterm2} and Slutsky's lemma we have the following statement: 
\begin{equation}\label{e: convergence in d}
\mbox{If }\quad \frac{|D_{\e}(z)|}{\e \Delta^{1/\alpha}_{\e}}\stackrel{\PP}{\lra} 0\quad \mbox{ as }\e \ra 0\mbox{,}\qquad \mbox{ then }\quad 
\frac{D_{\e}(z)}{\e \Delta^{1/\alpha}_{\e}}+\frac{1}{\Delta^{1/\al}_\e}U_{\e}+a_\e^0 \stackrel{d}{\lra} U
\quad\mbox{ as } \e \ra 0.
\end{equation}
Consequently, the remainder of the proof is dedicated to the verification of 
\begin{align}\label{e: drift in P}
\frac{|D_{\e}(z)|}{\e \Delta^{1/\alpha}_{\e}}\stackrel{\PP}{\lra} 0,\quad \mbox{ as }\e \ra 0.
\end{align}
For $\eta>0$, $r_\e = \e^{1-\vt}$, $\vt \in (0, \nicefrac{1}{4})$, and $|z|\lqq r_\e$ we have 
\begin{align}
&\PP\Big(\frac{|D_{\e}(z)|}{\e \Delta^{1/\alpha}_{\e}} \gqq \eta\Big)
\lqq \PP\Big(\frac{1}{\e\Delta^{1/\al}_\e}
\int_{0}^{\Delta_\e }
\Big| \ti b(X^{\e}_s(z)) \Big|\ud s\gqq \eta\Big)\nonumber\\
&\qquad \lqq \PP\Big(\frac{1}{\e\Delta^{1/\al}_\e}
\int_{0}^{\Delta_\e }
\Big| \ti b(X^{\e}_s(z))\Big|\ud s\gqq \eta, 
\sup_{0\lqq s\lqq \Delta_\e}|X^{\e}_s(z)|\lqq 2 {r_\e}\Big) \label{e:twoterms}
+
\PP\Big(\sup_{0\lqq s\lqq \Delta_\e}|X^{\e}_s(z)|> 2 {r_\e}\Big).
\end{align}
We start with the first term of the preceding inequality.
Since $\ti b\in \cC^2$, there are positive constants $C, r$ such that
\[
|\ti b(y)| = |b(y)-Db(0)y|\lqq C|y|^2\quad \textrm{ for any } |y|\lqq r.
\]
Bearing in mind that $r_\e \ra 0$ we have  
\begin{align*}
&\PP\Big(\frac{1}{\e\Delta^{1/\al}_\e}
\int_{0}^{\Delta_\e}
\Big| \ti b(X^{\e}_s(z)) \Big|\ud s\gqq \eta, \sup_{0\lqq s\lqq \Delta_\e}|X^{\e}_s(z)|\lqq 2{r_\e}\Big)\\
&\hspace{5cm} \lqq 
\PP\Big(\frac{4C}{\e\Delta^{1/\al}_\e}\Delta_\e {r^2_\e}
\gqq \eta, \sup_{0\lqq s\lqq \Delta_\e}|X^{\e}_s(z)|\lqq 2 {r_\e}\Big)= 0,
\end{align*}
for all $\e$ small enough, since the choice $r_\e=\e^{1-\vartheta}$, $\vt \in (0, \nicefrac{1}{4})$ and $\Delta_\e = \e^{\frac{\al}{2}}$ yields 
\[
\frac{1}{\e\Delta^{1/\al}_\e}\Delta_\e {r^2_\e}\lqq \Delta^{1-1/\alpha}_\e \e^{1-2\vartheta} = \e^{\al/2}\e^{1/2-2\vartheta}\ra 0, \quad 
\mbox{ as } \e \ra 0.
\]
It remains to treat the second term on the right-hand side of \eqref{e:twoterms}. 
More precisely we show 
\begin{equation}\label{eq:siorpaes99X}
\PP\Big(\sup_{0\lqq s\lqq \Delta_\e}|X^{\e}_s(z)|> 2{r_\e}\Big) \ra 0, \quad \mbox{ as } \e \ra 0. 
\end{equation}
By Theorem~1 in \cite{SIORPAES}, we have the following almost sure estimate 
\begin{align*}
\sup_{0\lqq s\lqq \Delta_\e}|X^\e_s(z)|
&\lqq 6\sqrt{[X^\e(z)]_{\Delta_\e}} +2 \int_0^{\Delta_\e}  H^\e_{s-}(z) \ud X^{\e}_s(z),
\end{align*}
where 
\begin{align*}
H^\e_s(z) = \frac{X^\e_{s-}(z)}{\sqrt{\sup\limits_{0 \lqq u \lqq s} |X^{\e}_{u-}(z)|^2 + [X^\e(z)]_{s-}}}. 
\end{align*}
Recall that by the L\'evy-It\^o decomposition \cite{Sa}, Chapter 4,  
the driving noise process $(L_t)_{t\gqq 0}$ under Hypotheses \ref{hyp: regularity}
has the following 
representation as Poisson random integrals 
\begin{equation*}
L_t = \int_{|z|\lqq 1} z \ti N(\ud s \ud z) + \int_{|z| > 1} z N(\ud s \ud z), 
\end{equation*}
where $N$ is the Poisson random measure associated to the L\'evy measure $\nu$ on $\RR^d \setminus \{0\}$ and $\ti N$ is its compensated 
counterpart 
\[
\ti N([a, b] \times A) = N([a, b] \times A) - (b-a) \nu(A), \qquad a < b, \quad A \in \bB(\RR^d). 
\]
In particular, we have the representation of the quadratic variation of $X^\e$ given by 
\begin{align*}
&[X^\e(z)]_{t} = [L]_t = \e^2\int_0^t \int_{|u|\lqq 1} |u|^2 N(\ud s\ud u).
\end{align*}
Furthermore, we have 
\begin{align*}
&\int_0^{t}  H_{s-}^\e(z) \ud X_s^\e(z) = \int_0^t \< H^\e_{s-}(z), -b(X^\e_s(z)) \> \ud s \\
&\qquad + 
\int_0^t \int_{|u|\lqq 1}\< H^\e_{s-}(z), \e u\> \ti N(\ud s\ud u) + 
\int_0^t \int_{|u|> 1}\< H^\e_{s-}(z), \e u\> N(\ud s\ud u).
\end{align*}
Since $b(0)=0$, Hypothesis~\ref{hyp: potential} yields   
\[
\int_0^t \< H^\e_{s-}(z), -b(X^\e_s(z)) \> \ud s \lqq 0, \quad  a.s. 
\]  
Hence \begin{align*}
&\PP\Big( \sup_{0\lqq s\lqq \Delta_\e}|X^\e_s(z)| > 2{r_\e}\Big) 
\lqq \PP\Big(6\e \sqrt{\int_0^{\Delta_\e} \int_{|u|\lqq 1} |u|^2 N(\ud s\ud u)} \\
&\qquad+  2\int_0^{\Delta_\e} \int_{|u|\lqq 1}\< H^\e_{s-}(z), u\> \ti N(\ud s\ud u) + 
2\int_0^{\Delta_\e} \int_{|u|> 1}\< H^\e_{s-}(z), u\> N(\ud s\ud u)  > 2{r_\e}\Big) \\
&\lqq \PP\Big(\e^2\int_0^{\Delta_\e} \int_{|u|\lqq 1} |u|^2 N(\ud s\ud u) > \frac{r^2_\e}{9^2}\Big) 
+ \PP\Big(\int_0^{\Delta_\e} \int_{|u|\lqq 1}\< H^\e_{s-}(z), \e u\> \ti N(\ud s\ud u) > \frac{1}{3} 
{r_\e}\Big) \\
&\qquad + \PP\Big(\int_0^{\Delta_\e} \int_{|u|> 1}\< H^\e_{s-}(z), \e u\> N(\ud s\ud u) > \frac{1}{3} {r_\e}\Big)\\
&=\PP\Big(\int_0^{\Delta_\e} \int_{|u|\lqq 1} |u|^2 N(\ud s\ud u) > \frac{\e^{-2\vartheta}}{9^2}\Big)
 + \PP\Big(\int_0^{\Delta_\e} \int_{|u|\lqq 1}\< H^\e_{s-}(z), u\> \ti N(\ud s\ud u) > 
\frac{1}{3} \e^{-\vartheta}\Big) \\
&\qquad + \PP\Big(\int_0^{\Delta_\e} \int_{|u|> 1}\< H^\e_{s-}(z),  u\> N(\ud s\ud u) > \frac{1}{3} \e^{-\vartheta}\Big).
\end{align*}
We continue term by term. The first term on the right side of the preceding inequality satisfies   
\begin{align}
\PP\Big(\int_0^{\Delta_\e} \int_{|u|\lqq 1} |u|^2 N(\ud s\ud u) > \frac{\e^{-2\vartheta}}{9^2}\Big)
&\lqq 9^2\Delta_\e \e^{2\vartheta}\int_{|u|\lqq 1} |u|^2 \nu(\ud u) \ra 0, \mbox{ as } \e\ra 0.\label{e: bterm1}
\end{align}
The second term can be estimated as follows 
\begin{align}
&\PP\Big(\int_0^{\Delta_\e} \int_{|z|\lqq 1}\< H^\e_{s-}(z), z\> \ti N(\ud s\ud z) > 
\frac{1}{3} \e^{-\vartheta}\Big)
\lqq 9\e^{2\vartheta}\EE\Big[\Big(\int_0^{\Delta_\e} \int_{|u|\lqq 1}\< H^\e_{s-}(z), u\> \ti N(\ud s\ud u\Big)^2\Big]\nonumber\\
&\qquad =  9\e^{2\vartheta}\EE\Big[\int_0^{\Delta_\e} \int_{|u|\lqq 1}\< H^\e_{s-}(z), u\>^2 \nu(\ud u)\ud s\Big]
= 9\e^{2\vartheta}\Delta_\e \int_{|u|\lqq 1} |u|^2 \nu(\ud u)\ra 0, \mbox{ as } \e \ra 0.\label{e: bterm2}
\end{align}
Finally, the third term is treated as follows. For $\beta\gqq 1$ we have 
\begin{align}
\PP\Big(\int_0^{\Delta_\e} \int_{|u|> 1}\< H_{s-}^\e(z), u\> N(\ud s\ud u) > 
\frac{1}{3} \e^{-\vt}\Big) &\lqq 
\PP\Big(\int_0^{\Delta_\e} \int_{|u|> 1} |u| N(\ud s\ud u) > \frac{1}{3} \e^{-\vt}\Big) \nonumber\\
&\lqq 3 \e^{\vt} \Delta_\e \int_{|u|>1} |u| \nu(du) \ra 0,\quad \mbox{ as } \e\ra 0.\label{e: bterm3a}
\end{align}
For $\beta \in (0, 1)$ we use the subadditivity of the root for sums of nonnegative terms 
(see \cite{Bie}), Markov's inequality and Hypothesis~\ref{hyp: moment condition} and obtain 
\begin{align}
&\PP\Big(\int_0^{\Delta_\e} \int_{|u|> 1}\< H_{s-}^\e(z), u\> N(\ud s\ud u) > 
\frac{1}{3} \e^{-\vt}\Big) 
\lqq 
\PP\Big(\big(\int_0^{\Delta_\e} \int_{|u|> 1} |u| N(\ud s\ud u)\big)^\beta > \frac{1}{3^\beta} \e^{-\beta\vt}\Big) \nonumber\\
&\lqq 
\PP\Big(\int_0^{\Delta_\e} \int_{|u|> 1} |u|^\beta N(\ud s\ud u) > \frac{1}{3^\beta} \e^{-\beta\vt}\Big) \lqq 3^\beta \e^{\beta \vt} \Delta_\e \int_{|u|>1} |u|^\beta \nu(du) \ra 0, \quad\mbox{ as } \e\ra 0.\label{e: bterm3b}
\end{align}
This finishes the proof of Lemma \ref{lem:indistri}.
\end{proof}

\bigskip 
\subsubsection{\textbf{Step 3: Strengthening Step 2 to $\norm{\cdot}$ by localization for $\al> 3/2$}}\label{ss:Step3} \hfill\\

\noindent In this step we prove that $G_1\to 0$, $\e\to 0$. More precisely, we show
for $\Delta_\e=\e^{\al/2}$, $\alpha \in (3/2,2)$, $\beta>0$ the following limit
\begin{equation}\label{eq: short coupling}
\lim\limits_{\e\rightarrow 0}
\norm{X^{\e}_{\Delta_\e}(X^{\e}_{T^x_\e}(x))-
Z^{\e}_{\Delta_\e}(X^{\e}_{T^x_\e}(x))}=0.
\end{equation}
Note that the scale and shift invariance property of the total variation imply 
\begin{align}
&\norm{X^\e_{\Delta_\e}(z) - Z^\e_{\Delta_\e}(z)}\nonumber\\
&=\norm{\Big(\frac{X^\e_{\Delta_\e}(z) - e^{-Db(0)\Delta_\e}z}{\e \Delta^{1/\alpha}_{\e}}+a_\e^0\Big)- 
\Big(\frac{Z^\e_{\Delta_\e}(z) - e^{-Db(0)\Delta_\e}z}{\e \Delta^{1/\alpha}_{\e}}+a_\e^0\Big)}\nonumber\\
&\lqq \norm{\Big(\frac{X^\e_{\Delta_\e}(z) - e^{-Db(0)\Delta_\e}z}{\e \Delta^{1/\alpha}_{\e}}+a_\e^0\Big)-U}- 
\norm{\Big(\frac{Z^\e_{\Delta_\e}(z) - e^{-Db(0)\Delta_\e}z}{\e \Delta^{1/\alpha}_{\e}}+a_\e^0\Big)-U}\label{eq:uptodist}
\end{align}
By \eqref{e:termterm2} it remains to prove the following result. 
\begin{prop}[Nonlinear local short-time coupling]\label{p:localcoupling}
For $\al\in (\nicefrac{3}{2},2)$ and $\beta>0$ it follows 
\begin{equation}\label{eq:desiredlimit}
\sup_{|z| \lqq r_\e} \norm{\Big(\frac{X^{\e}_{\Delta_\e}(z)-e^{-Db(0)\Delta_\e}z}{\e \Delta^{1/\alpha}_{\e}} + a_\e^0\Big)- U} \ra 0, \quad \mbox{ as } \e \ra 0,  
\end{equation}
where $a_\e^0$ is given in Proposition~\ref{tomate}, $U\stackrel{d}= \mathcal{S}_\al(\Lambda_1)$ and $r_\e= \e^{1-\vt}$, $\vt \in (0, \nicefrac{1}{4})$, defined below \eqref{e: drift in P}.
\end{prop}
The proof is given after the subsequent localization results. 
Note that \eqref{eq:indistri} states exactly \eqref{eq:desiredlimit} in distribution in a slightly different notation. In order to strengthen the result to the total variation   
we apply the following consecutive localization procedures to bounded jumps and a bounded vector field. 
The proof relies on the Plancherel isometry and Fourier inversion.  
We stress that the following two lemmas are true in full generality, that is, 
for any $\alpha\in (0,2)$ and $\beta>0$.

\begin{lem}[Jump size localization] \label{lem: TVlocalization}
Let 
\begin{equation}\label{eq:jumplocalization}
\tT_\e := \inf\{t>0~|~|\e (L_t - L_{t-})|> 1\}, \qquad \e \in (0, 1).
\end{equation}
Then for any $z\in \RR^d$,
\[
X^{\e, \tT_\e}_{t}(z) := X^{\e}_{t\wedge \tT_\e}(z) \qquad \mbox{ and } \qquad Z^{\e, \tT_\e}_{t}(z) := Z^{\e}_{t\wedge \tT_\e}(z),
\]
we have 
\[
\left|
\norm{X^{\e}_{\Delta_\e}(z)-Z^{\e}_{\Delta_\e}(z)}-\norm{X^{\e, \tT_\e}_{\Delta_\e}(z)-Z^{\e, \tT_\e}_{\Delta_\e}(z)} \right|\lqq 2\PP(\tT_\e\lqq \Delta_\e).
\]
In addition, $\beta>0$ and Hypothesis~\ref{hyp: moment condition} imply, for $\e$ small enough, 
\begin{equation}\label{e:localizandocantitativo}
\left|
\norm{X^{\e}_{\Delta_\e}(z)-Z^{\e}_{\Delta_\e}(z)}-\norm{X^{\e, \tT_\e}_{\Delta_\e}(z)-Z^{\e, \tT_\e}_{\Delta_\e}(z)} \right|\lqq 2\Delta_\e \e^\beta.
\end{equation}
\end{lem}

\begin{proof}
It is well-known that $\tT_\e$ and $(X^{\e, \tT_\e}_{t}(z))_{t\gqq 0}$ are conditionally independent 
($(Z^{\e, \tT_\e}_{t}(z))_{t\gqq 0}$, respectively). Hence disintegration yields 
\begin{align*}
&\norm{X^{\e, \tT_\e}_{\Delta_\e \wedge \tT_\e}(z)-Z^{\e, \tT_\e}_{\Delta_\e \wedge \tT_\e}(z)}\lqq \int_{0}^{\infty}
\EE\left[\norm{X^{\e, \tT_\e}_{\Delta_\e \wedge \tT_\e}(z)-Z^{\e, \tT_\e}_{\Delta_\e \wedge \tT_\e}(z)}\big| \tT_\e=s\right]\PP(\tT_\e \in \ud s)\\
&\hspace{3cm} = \int_{0}^{\infty}
\norm{X^{\e}_{\Delta_\e \wedge \tT_\e}(z)-Z^{\e}_{\Delta_\e \wedge \tT_\e}(z)}\PP(\tT_\e \in \ud s)\\
&\hspace{3cm} =\int_{0}^{\Delta_\e}\norm{X^{\e}_{s}(z)-Z^{\e}_{s}(z)}\PP(\tT_\e \in \ud s)
    +\norm{X^\e_{\Delta_\e}(z)-Y^\e_{\Delta_\e}(z)}\int_{\Delta_\e}^{\infty}\PP(\tT_\e \in \ud s)\\
&\hspace{3cm} \lqq \PP(\tT_\e \lqq \Delta_\e) + \norm{X^\e_{\Delta_\e}(z)-Z^\e_{\Delta_\e}(z)}.
\end{align*}
On the other hand, we notice
\begin{align*}
&\norm{X^\e_{\Delta_\e}(z)-Z^\e_{\Delta_\e}(z)}\\
&\hspace{2cm} \lqq \norm{X^\e_{\Delta_\e}(z)-X^{\e, \tT_\e}_{\Delta_\e}(z)}+\norm{X^{\e, \tT_\e}_{\Delta_\e}(z)-Z^{\e, \tT_\e}_{\Delta_\e}(z)}+\norm{Z^{\e, \tT_\e}_{\Delta_\e}(z)-Z^{\e}_{\Delta_\e}(z)}\\
&\hspace{2cm} \lqq 2\PP(\tT_\e \lqq \Delta_\e) +\norm{X^{\e, \tT_\e}_{\Delta_\e}(z)-Z^{\e, \tT_\e}_{\Delta_\e}(z)}.
\end{align*}
Consequently, it follows 
\begin{equation}\label{e:e1}
\Big| \norm{X^{\e, \tT_\e}_{\Delta_\e \wedge \tT_\e}(z)-Z^{\e, \tT_\e}_{\Delta_\e \wedge \tT_\e}(z)} - \norm{X^\e_{\Delta_\e}(z)-Z^\e_{\Delta_\e}(z)} \Big| \lqq 2 \PP(\tT_\e \lqq \Delta_\e).
\end{equation}
Finally we calculate 
 \begin{align}\label{e:e2}
\PP(\tT_\e \lqq \Delta_\e) = 1- \PP(\tT_\e > \Delta_\e) = 1- e^{-\Delta_\e \nu(\frac{1}{\e} B_1^c(0))}.
\end{align}
Since $\beta>0$, Hypothesis~\ref{hyp: moment condition} 
implies $
\lim \limits_{r\to \infty} r^\beta \nu(r B_1^c(0)) = 0$, 
which yields
\begin{equation}\label{e:e3}
\limsup_{\e \ra 0} \Delta_\e \nu\Big(\frac{1}{\e} B_1^c(0)\Big) \lqq \limsup_{\e \ra 0} \e^\beta \Delta_\e = 0. 
\end{equation}
Combining \eqref{e:e1}-\eqref{e:e3} we obtain \eqref{e:localizandocantitativo}.
\end{proof}
\noindent Since $\tT_\e > \Delta_\e$ with high probability, 
we can assume without loss of generality 
the presence of only bounded jumps even 
in the total variation distance. 

\begin{lem}[Spatial localization]\label{lem:spatloc}\hfill\\
Let $h\in \cC^\infty_b(\RR^d,[0,1])$ be given by
\begin{equation}\label{eq:mollifier}
h(\zeta)=
\begin{cases}
1 & \mathrm{ for }\;\;|\zeta|\lqq 1,  \\
\in (0,1) &  \mathrm{ for }\;\;1<|\zeta|< 2,\\
0  &  \mathrm{ for }\;\; |\zeta|\gqq 2.
\end{cases}
\end{equation}
Consider the following localized solutions
\begin{align*}
\ud \hat X^{\e}_t(z)&=-b(\hat X^{\e}_t(z))h(\hat X^{\e}_t(z))\ud t+\e \ud L_t, \qquad  \hat X^{\e}_0(z)=z,\\
\ud \hat Z^{\e}_t(z)&=-Db(0)\hat Z^{\e}_t(z)h(\hat Z^{\e}_t(z))\ud t+\e \ud L_t, \qquad  \hat Z^{\e}_0(z)=z
\end{align*}
of $(X^\e_t)_{t \gqq 0}$ and $(Z^\e_t)_{t\gqq 0}$ defined in \eqref{e:VOCX} and \eqref{e:VOCZ}, respectively. 
Then for $|z|<1/2$ we have for all $t\gqq 0$
\begin{align}\label{eq:localspatial}
\Big|\norm{ X^{\e}_t(z)-Z^{\e}_t(z)}
-\norm{ \hat X^{\e}_t(z)-\hat Z^{\e}_t(z)}\Big |
\lqq \PP(\hat \tau^{\e}(z)<t) +\PP(\hat \sigma^\e(z)<t), 
\end{align}
where 
\begin{align*}
\hat\tau^\e(z)=\inf\{s\gqq 0: |X^{\e}_s(z)|>1 \} \quad \textrm{ and} \quad
\hat\sigma^\e(z)=\inf\{s\gqq 0: 
|Z^{\e}_s(z)|>1\}.
\end{align*}
In particular, we have
\begin{align}\label{eq:spatialcoupling}
\lim_{\e\to 0}\Big|\norm{X^{\e}_{\Delta_\e}(z)-Z^{\e}_{\Delta_\e}(z)}-
\norm{\hat X^{\e}_{\Delta_\e}(z)-\hat Z^{\e}_{\Delta_\e}(z)}\Big|
=0.
\end{align}
\end{lem}

\begin{proof}
By the triangle inequality and the coupling representation of the total variation distance, we have
\begin{align*}
\norm{ X^{\e}_t(z)-Z^{\e}_t(z)}& \lqq 
\norm{ X^{\e}_t(z)-\hat X^{\e}_t(z)}+\norm{ \hat X^{\e}_t(z)-\hat Z^{\e}_t(z)}+
\norm{ \hat Z^{\e}_t(z)-Z^{\e}_t(z)}\\
&\lqq \PP(X^{\e}_t(z)\neq \hat X^{\e}_t(z))+\norm{ \hat X^{\e}_t(z)-\hat Z^{\e}_t(z)}+
\PP(Z^{\e}_t(z)\neq \hat Z^{\e}_t(z))\\
& \lqq 
\PP(\hat \tau^{\e}(z)<t)+
\norm{ \hat X^{\e}_t(z)-\hat Z^{\e}_t(z)}+
\PP(\hat \sigma^{\e}(z)<t).
\end{align*}
Exchanging the roles of $X^{\e}_t(z)$ and $Z^{\e}_t(z)$ with $\hat X^{\e}_t(z)$ and $\hat Z^{\e}_t(z)$ yields \eqref{eq:localspatial}.
Since $\Delta_\e\to 0$, we have 
\[ 
\PP\Big(\hat \tau^{\e}(z)<\Delta_\e\Big)\lqq 
\PP\Big(\sup_{s\in [0,\Delta_\e]}|X^{\e}_s(z)|>1/2\Big)
\]
 which tends to zero due to \eqref{eq:siorpaes99X}.
Analogously, the same result holds true for the linear process
$(Z^{\e}_s(z))_{t\gqq 0}$ and 
$\sigma^{\e}(z)$.
This implies the desired result
 \eqref{eq:spatialcoupling}.
\end{proof}

\begin{rem}\label{rem:localization}~
\begin{enumerate}
\item 
We stress the following 
intentional abuse of notation.
Lemma \ref{lem: TVlocalization} 
yields that it is enough to prove \eqref{eq:desiredlimit} for $X^\e$ being replaced by $X^{\e, \tT_\e}$. In other words, we may assume that $X^{\e}$ 
has bounded jumps beforehand and consequently all polynomial moments finite. 
\item In addition, Lemma \ref{lem:spatloc} allows us to consider bounded vector fields in the spirit of Section 4 in \cite{Dong}. That is to say, it is enough to prove \eqref{eq:desiredlimit} for $X^\e$ being replaced by $\fX^\e_t(z)=\hat X^{\e,\tT^\e}_t(z)$.
\item We emphasize that due to 
Hypothesis~\ref{hyp: regularity}, in particular, \eqref{eq: gradientbound} the process
 $(\fX^\e_t(z))_{t\gqq 0}$ is a strong Feller process with $\mathcal{C}^1_b$ density $f_\e$. Since $|f_\e|_\infty<\infty$ , we have
 $f_\e\in L^2(\RR^d)$. 
 See Theorem~1.1 and Theorem~1.3 in \cite{SongZhang} for details. 
\end{enumerate}
\end{rem}

\begin{proof}[Proof of Proposition \ref{p:localcoupling}:]
For simplicity we keep the same notation 
except for the driving noise which we denote by $\ti L$. 
We set
\[
\ti L_t = \int_0^t \int_{|u|\lqq 1} u \ti N(\ud s \ud u).  
\]
\noindent Note that since $\Psi_t = e^{Db(0)t}$ we have
\begin{align}\label{e:cambiodenotacion}
\frac{\fX^{\e}_{\Delta_\e}(z)-\Psi^{-1}_{\Delta_\e}z}{\e \Delta^{1/\alpha}_{\e}} + a_\e^0 
= \frac{\fD_\e(z)}{\e \Delta_\e^{1/\al}} + \frac{1}{\Delta_\e^{1/\al}} \mathfrak{U}_\e + a_\e^0,
\end{align}
where 
\begin{align*}
\mathfrak{U}_{\e} &:= \Psi^{-1}_{\Delta_\e}\int_{0}^{\Delta_\e}\Psi_s\ud  \ti L_s, 
\qquad \fD_{\e}(z) :=\Psi^{-1}_{\Delta_\e} \int_{0}^{\Delta_\e}\Psi_s\ti b( \fX^\e_s(z))\ud s
\end{align*}
and  $\ti b(x)=b(x)h(x)-Db(0)x$, $x\in \RR^d$,
where $h$ is given in \eqref{eq:mollifier}.
Note that the limit \eqref{e: drift in P} is shown for $X^\e$. It is easily seen - going through the proof of Lemma \ref{lem:indistri} line by line  - that with the help of \eqref{e:cambiodenotacion} the limit \eqref{eq:indistri} remains valid  for $X^\e$ being replaced by $\fX^\e$, i.e., 
\begin{equation}\label{e: debil}
\frac{\fX^{\e}_{\Delta_\e}(z)-\Psi^{-1}_{\Delta_\e}z}{\e \Delta_\e^\frac{1}{\al}} + a_\e^0 \stackrel{d}{\lra} \mathcal{S}_\alpha(\Lambda_1), \qquad \mbox{ as } \e \ra 0. 
\end{equation}
Recall
\[
\fX^{\e}_t(z)=\Psi^{-1}_t z-\Psi^{-1}_t\int_{0}^{t}\Psi_s\ti {b}(\fX^{\e}_s(z))\ud s+\e
\Psi^{-1}_t\int_{0}^{t}\Psi_s\ud \ti L_s,
\] 
and set $\xX^{\e}_t(z)= \fX^{\e}_t(z)-\Psi^{-1}_tz$ which satisfies 
\begin{equation}\label{eq: Xdescontada}
\xX^{\e}_t(z)=-\Psi^{-1}_t\int_{0}^{t}\Psi_s\ti {b}(\xX^{\e}_s(z)+
\Psi^{-1}_s z)\ud s+\e
\Psi^{-1}_t\int_{0}^{t}\Psi_s\ud \ti L_s.
\end{equation}
In the sequel, we strengthen the convergence of \eqref{e: debil} 
to the convergence in the total variation distance. 
Since $\ti L$ has absolutely continuous marginals and 
$\fX^\e_{\Delta_\e}$ is a continuous push-forward of $\ti L$, 
it retains the absolute continuity 
property. In addition, it is not hard to see that Lemma \ref{lem oreymasuda} yields a $\cC^\infty$-density for $\mathcal{S}_\al(\Lambda_1)$. Hence it is enough to prove 
\[
\int_{\RR^d} |f_\e(u) - f_0(u)|\ud u \ra 0, \quad \mbox{ as } \e \ra 0, 
\]
where $f_\e$ is the density of $\nicefrac{\xX^{\e}_{\Delta_\e}(z)}{(\e \Delta_\e^\frac{1}{\al})} +a_\e$ 
and $f_0$ is the density of $\mathcal{S}_\al(\Lambda_1)$. 
By Scheff\'e's lemma for densities it is sufficient show that $f_\e \ra f_0$, as $\e \ra 0$, Lebesgue almost everywhere in $\RR^d$. 
For this sake, it is sufficient to prove that 
\[
\int_{\RR^d} |f_\e(u) - f_0(u)|^2 \ud u \ra 0,\quad \mbox{ as } \e \ra 0.
\]
Since $f_\e,f_0\in L^2(\RR^d)$,
by Plancherel's identity we have a positive constant $C_\pi$ such that 
\[
\int_{\RR^d} |f_\e(u) - f_0(u)|^2 \ud u = C_\pi \int_{\RR^d} |\hat f_\e(\theta) - \hat f_0(\theta)|^2 \ud \theta.
\]
Since the weak convergence \eqref{e: debil} implies that $\hat f_\e \ra \hat f_0$ uniformly on compacts, we have for any $K>0$  
\begin{align*}
\limsup_{\e \ra 0} \int_{\RR^d} |\hat f_\e(\theta) - \hat f_0(\theta)|^2 \ud \theta 
&\lqq \limsup_{\e \ra 0} \int_{|\theta|>K} |\hat f_\e(\theta) - \hat f_0(\theta)|^2 \ud \theta \\
&\lqq 2 \limsup_{\e \ra 0}  \int_{|\theta|>K} |\hat f_\e(\theta)|^2 \ud \theta+  2 \int_{|\theta|>K} |\hat f_0(\theta)|^2 \ud \theta.
\end{align*}
The exponential decay of $\hat f_0$ yields $\hat f_0 \in L^2(\RR^d)$. 
Sending $K$ to infinity we deduce that 
\begin{align}\label{e: UI} 
\limsup_{\e \ra 0} \int_{\RR^d} |\hat f_\e(\theta) - \hat f_0(\theta)|^2 \ud \theta 
&\lqq 2 \limsup_{K\ra \infty} \limsup_{\e \ra 0}  \int_{|\theta|>K} |\hat f_\e(\theta)|^2 \ud \theta.
\end{align}
In order to conclude, it remains to show that the right-hand side of the preceding inequality is~$0$. 
Recall the differential version of \eqref{eq: Xdescontada}  
\[
\ud \xX^{\e}_t(z)=-Db(0) \xX^{\e}_t(z) \ud t 
-\ti {b}(\xX^{\e}_t(z)+\Psi^{-1}_t z)\ud t+\e \ud \ti L_t 
\]
with initial datum $\xX^{\e}_0(z)=0$. In the sequel, we calculate
$\phi_t(\theta) := \EE\big[e^{i \<\theta, \xX^\e_{t}(z)\>}\big]$. 
It\^o's formula yields 
\begin{align*}
\exp\Big(i \<\theta, \xX^\e_{t}(z)\>\Big)&=1+\int_0^t \exp\Big(i \<\theta,\xX^\e_{s}(z)\>\Big) i \< \theta,  -Db(0)\xX^\e_{s}(z)-\ti b(\xX^\e_{s}(z)+\Psi_s^{-1}z)\> \ud s \\
&\quad + \int_0^t \int_{|z|\lqq 1}\bigg(\exp\Big(i \<\theta,  \xX^\e_{s-}(z)+ \e u\>\Big) -\exp(\Big(i \<\theta, \xX^\e_{s-}(z)\>\Big)\bigg) \ti N(\ud s \ud u)  \\
&\quad + \int_0^t \exp\Big(i \<\theta,  \xX^\e_{s-}(z)\>\Big)\int_{|z|\lqq 1}
\Big(\exp\big(i \<\theta, \e u \>\big)-1-i \< \theta, \e u\>\Big) \nu(\ud u) 
\ud s.
\end{align*}
Since the process $\fX^\e$ has finite first moment, 
taking expectation and using Fubini's  theorem we obtain
\begin{align*}
&\phi_t(\theta) 
= \EE\Big[\exp\Big(i \<\theta, \xX^\e_{t}(z)\>\Big)\Big]\\
&=1 +\int_0^t \EE\Big[\exp\Big(i \<\theta,\xX^\e_{s}(z)\>\Big) i \< \theta,  -Db(0)\xX^\e_{s}(z)-\ti b(\xX^\e_{s}(z)+\Psi_s^{-1}z)\> \Big]\ud s \\
&\quad + \int_0^t \EE\Big[\exp\Big(i \<\theta,  \xX^\e_{s-}(z)\>\Big)\int_{|z|\lqq 1}
\Big(\exp\big(i \<\theta, \e u \>\big)-1-i \< \theta, \e u\>\Big) \nu(\ud u) 
\Big]\ud s.
\end{align*}
Note that
\[
\psi(\e \theta)=\int_{\RR^d}
\Big(\exp\big(i \<\theta, \e u \>\big)-1-i \< \theta, \e u\>\Big) \nu(\ud u),
\]
where $\psi$ is the characteristic exponent of the L\'evy measure $\nu$. 
We set $\theta_\e := \nicefrac{\theta}{(\e \Delta_\e^\frac{1}{\al})}$ for $|\theta|\gqq K$. 
For the real and the imaginary part of $\phi_t(\theta)$ we have the equalities
\begin{align*}
&\EE\Big[\cos \Big(\<\theta_{\e}, \xX^\e_{t}(z)\>\Big)\Big]=1 -\int_0^t \EE\Big[\sin \Big(\<\theta_{\e}, \xX^\e_{s}(z)\>\Big) 
\< \theta_{\e},  -Db(0)\xX^\e_{s}(z)-\ti b(\xX^\e_{s}(z)+\Psi_s^{-1}z)\> \Big]\ud s \\
&\hspace{2.8cm} +\mathsf{Re}\psi(\e \theta_\e) 
\int_0^t \EE\Big[\cos \Big(\<\theta_{\e}, \xX^\e_{s}(z)\>\Big)\Big]\ud s 
- \mathsf{Im}\psi(\e \theta_\e) 
\int_0^t \EE\Big[\sin \Big(\<\theta_{\e}, \xX^\e_{s}(z)\>\Big)\Big]\ud s
\end{align*}
and 
\begin{align*}
&\EE\Big[\sin \Big(\<\theta_{\e}, \xX^\e_{t}(z)\>\Big)\Big]
=1 +\int_0^t \EE\Big[\cos \Big(\<\theta_{\e}, \xX^\e_{s}(z)\>\Big) 
\< \theta_{\e},  -Db(0)\xX^\e_{s}(z)-\ti b(\xX^\e_{s}(z)+\Psi_s^{-1} z)\> \Big]\ud s \\
&\hspace{2.8cm} +\mathsf{Re}\psi(\e \theta_\e) 
\int_0^t \EE\Big[\sin \Big(\<\theta_{\e}, \xX^\e_{s}(z)\>\Big)\Big]\ud s 
+ \mathsf{Im} \psi(\e \theta_\e) 
\int_0^t \EE\Big[\cos \Big(\<\theta_{\e}, \xX^\e_{s}(z)\>\Big)\Big]\ud s.
\end{align*}
The chain rule for the respective differential forms reads as follows
\begin{align*}
&\frac{\ud }{\ud t}\left(\EE\Big[\cos \Big(\<\theta_{\e}, \xX^\e_{t}(z)\>\Big)\Big]\right)^2
=2\EE\Big[\cos \Big(\<\theta_{\e}, \xX^\e_{t}(z)\>\Big)\Big]\frac{\ud}{\ud t}\EE\Big[\cos \Big(\<\theta_{\e}, \xX^\e_{t}(z)\>\Big)\Big]
\\
&=-2\EE\Big[\cos \Big(\<\theta_{\e}, \xX^\e_{t}(z)\>\Big)\Big]\EE\Big[\sin \Big(\<\theta_{\e}, \xX^\e_{t}(z)\>\Big) 
\< \theta_{\e},  -Db(0)\xX^\e_{t}(z)-\ti b(\xX^\e_{t}(z)+\Psi_t^{-1}z)\> \Big] \\
&\qquad  +2\mathsf{Re}\psi(\e \theta_\e)\left(\EE\Big[\cos \Big(\<\theta_{\e}, \xX^\e_{t}(z)\>\Big)\Big]\right)^2 \\
&\qquad - 2\mathsf{Im}\psi(\e \theta_\e) 
\EE\Big[\cos \Big(\<\theta_{\e}, \xX^\e_{t}(z)\>\Big)\Big] \EE\Big[\sin \Big(\<\theta_{\e}, \xX^\e_{t}(z)\>\Big)\Big],
\end{align*}
with $\left(\EE\Big[\cos \Big(\<\theta_{\e}, \xX^\e_{0}(z)\>\Big)\Big]\right)^2 = 1$ and 
\begin{align*}
&\frac{\ud }{\ud t}\left(\EE\Big[\sin \Big(\<\theta_{\e}, \xX^\e_{t}(z)\>\Big)\Big]\right)^2=2\EE\Big[\cos \Big(\<\theta_{\e}, \xX^\e_{t}(z)\>\Big)\Big]\frac{\ud}{\ud t}\EE\Big[\cos \Big(\<\theta_{\e}, \xX^\e_{t}(z)\>\Big)\Big]
\\
&=2\EE\Big[\sin \Big(\<\theta_{\e}, \xX^\e_{t}(z)\>\Big)\Big]\EE\Big[\cos \Big(\<\theta_{\e}, \xX^\e_{t}(z)\>\Big) 
\< \theta_{\e},  -Db(0)\xX^\e_{t}(z)-\ti b(\xX^\e_{t}(z)+\Psi_t^{-1}z)\> \Big] \\
&\qquad +2\mathsf{Re}\psi(\e \theta_\e)\left(\EE\Big[\sin \Big(\<\theta_{\e}, \xX^\e_{t}(z)\>\Big)\Big]\right)^2 \\
&\qquad +2\mathsf{Im}\psi(\e \theta_\e) 
\EE\Big[\sin \Big(\<\theta_{\e}, \xX^\e_{t}(z)\>\Big)\Big] \EE\Big[\cos \Big(\<\theta_{\e}, \xX^\e_{t}(z)\>\Big)\Big],
\end{align*}
with $\left(\EE\Big[\sin \Big(\<\theta_{\e}, \xX^\e_{0}(z)\>\Big)\Big]\right)^2 = 0$. 
We sum up the preceding equations and obtain 
\begin{align*}
&\frac{\ud}{\ud t} |\phi_t(\theta_\e)|^2 
= 2 \mathsf{Re} \psi(\e \theta_\e) |\phi_t(\theta_\e)|^2 \\
& -2\EE\Big[\cos \Big(\<\theta_{\e}, \xX^\e_{t}(z)\>\Big)\Big]\EE\Big[\sin \Big(\<\theta_{\e}, \xX^\e_{t}(z)\>\Big) 
\< \theta_{\e},  -Db(0)\xX^\e_{t}(z)-\ti b(\xX^\e_{t}(z)+\Psi_t^{-1}z)\> \Big] \\
&+2\EE\Big[\sin \Big(\<\theta_{\e}, \xX^\e_{t}(z)\>\Big)\Big]\EE\Big[\cos \Big(\<\theta_{\e}, \xX^\e_{t}(z)\>\Big) 
\< \theta_{\e},  -Db(0)\xX^\e_{t}(z)-\ti b(\xX^\e_{t}(z)+\Psi_t^{-1}z)\> \Big].
\end{align*}
We start with the first term on the right-hand side. 
By $|\theta|\gqq K$ and Lemma \ref{lem oreymasuda}
we have for small $\e$ (where the smallness of $\e$ only depends of $K$ and $C_\sphericalangle$) the estimate
\begin{align*}
 \mathsf{Re} \psi(\e\theta_\e) 
 &= \int_{|u|\lqq 1} \Big(\cos\Big(\Big\< \frac{\theta}{\Delta_\e^{1/\al}}, u\Big\>\Big) - 1\Big)\nu(\ud u) \lqq \int_{\substack{|\< \frac{\theta}{\Delta_\e^{1/\al}}, u\>|\lqq 1\\
|u|\lqq 1
} } \Big(\cos\Big(\Big\< \frac{\theta}{\Delta_\e^{1/\al}}, u\Big\>\Big) - 1\Big)\nu(\ud u)\\
 &\lqq - \frac{2}{\pi^2}\int_{\substack{|\< \frac{\theta}{\Delta_\e^{1/\al}}, u\>|\lqq 1\\
|u|\lqq 1
} } \Big|\Big\< \frac{\theta}{\Delta_\e^{1/\al}}, u\Big\>\Big|^2\nu(\ud u) \lqq -\frac{2}{\pi^2} c_\sphericalangle \Big|\frac{\theta}{\Delta_\e^{1/\al}}\Big|^{\al}
 = -C \frac{|\theta|^{\alpha}}{\Delta_\e}.
\end{align*}
We continue with the second term. Recall that $\alpha \in (3/2,2)$. 
By the Cauchy-Schwarz inequality  and the classical Young inequality for $p = p^* = 2$  we have 
\begin{align*}
&\Big|-2\EE\Big[\cos \Big(\<\theta_{\e}, \xX^\e_{t}(z)\>\Big)\Big]\EE\Big[\sin \Big(\<\theta_{\e}, \xX^\e_{t}(z)\>\Big) 
\< \theta_{\e},  -Db(0)\xX^\e_{t}(z)-\ti b(\xX^\e_{t}(z)+\Psi_t^{-1}z)\> \Big]\Big| \\
&\lqq 2\big|\EE\Big[\cos \Big(\<\theta_{\e}, \xX^\e_{t}(z)\>\Big)\Big]\big|~
|\e \theta_{\e}|^{3/4} \frac{|\theta_\e|^{1/4}}{\e^{3/4}}
~\EE\Big[ | -Db(0)\xX^\e_{t}(z)-\ti b(\xX^\e_{t}(z)+\Psi_t^{-1}z)| \Big]\\
&\lqq |\e \theta_{\e}|^{3/2}~\EE\Big[\cos \Big(\<\theta_{\e}, \xX^\e_{t}(z)\>\Big)\Big]^2  + \frac{|\theta_\e|^{1/2}}{\e^{3/2}} ~\EE\Big[|-Db(0)\xX^\e_{t}(z)-\ti b(\xX^\e_{t}(z)+\Psi_t^{-1}z)| \Big]^2.
\end{align*}
For the third term we infer analogously 
\begin{align*}
&\Big|2\EE\Big[\sin \Big(\<\theta_{\e}, \xX^\e_{t}(z)\>\Big)\Big]\EE\Big[\cos \Big(\<\theta_{\e}, \xX^\e_{t}(z)\>\Big) 
\< \theta_{\e},  -Db(0)\xX^\e_{t}(z)-\ti b(\xX^\e_{t}(z)+\Psi_t^{-1}z)\> \Big]\Big|\\
&\lqq |\e \theta_{\e}|^{3/2}~\EE\Big[\sin \Big(\<\theta_{\e}, \xX^\e_{t}(z)\>\Big)\Big]^2  + 
\frac{|\theta_\e|^{1/2}}{\e^{3/2}} ~\EE\Big[|-Db(0)\xX^\e_{t}(z)-\ti b(\xX^\e_{t}(z)+\Psi_t^{-1}z)| \Big]^2.
\end{align*}
Since $\al \in (\nicefrac{3}{2}, 2)$ we obtain for sufficiently small $\e$ and $|\theta|\gqq K$ 
\begin{align*}
\mathsf{Re} \psi(\e \theta_\e) + |\e \theta_\e|^{3/2}  \lqq - 2 C |\e \theta_\e|^{\alpha} + |\e \theta_\e|^{3/2} 
\lqq - C \frac{|\theta|^\al}{\Delta_\e}.
\end{align*}
Therefore for $|\theta|\gqq K$ and $\e$ small enough 
we have the following differential inequality 
\begin{align*}
\frac{\ud}{\ud t} |\phi_t(\theta_\e)|^2 
&\lqq \Big(\mathsf{Re} \psi(\e \theta_\e)  + |\e \theta_{\e}|^{3/2}\Big) |\phi_t(\theta_\e)|^2 + 2\frac{|\theta_\e|^{1/2}}{\e^{3/2}} ~ \EE\Big[|-Db(0)\xX^\e_{t}(z)-\ti b(\xX^\e_{t}(z)+\Psi_t^{-1}z)| \Big]^2\\
&\lqq - C \frac{|\theta|^\al}{\Delta_\e} |\phi_t(\theta_\e)|^2  + 2\frac{|\theta_\e|^{1/2}}{\e^{3/2}} ~ \EE\Big[|-Db(0)\xX^\e_{t}(z)-\ti b(\xX^\e_{t}(z)+\Psi_t^{-1}z)| \Big]^2.
\end{align*}
In the sequel, we dominate the term $\EE\Big[|-Db(0)\xX^\e_{t}(z)-\ti b(\xX^\e_{t}(z)+\Psi_t^{-1}z)| \Big]^2$. 
Recall the definition of $\ti b(x)=b(x)h(x)-Db(0)x$, $x\in \RR^d$,
where $h$ is given in \eqref{eq:mollifier}.
Note that
\[
\begin{split}
&-Db(0)\xX^\e_{t}(z)-\ti b(\xX^\e_{t}(z)+\Psi_t^{-1}z)\\
&\quad =-Db(0)\xX^\e_{t}(z) - b(\xX^\e_{t}(z)+\Psi_t^{-1}z)h(\xX^\e_{t}(z)+\Psi_t^{-1}z)+Db(0)(\xX^\e_{t}(z)+\Psi_t^{-1}z)\\
&\quad =-b(\xX^\e_{t}(z)+Db(0)\Psi_t^{-1}z)h(\xX^\e_{t}(z)+\Psi_t^{-1}z)+Db(0)\Psi_t^{-1}z.
\end{split}
\]
Since  $\xX^{\e}_t(z)=\fX^{\e}_t(z)-\Psi_t^{-1}z$, \eqref{eq:siorpaes99X}
implies
\begin{align*}
\lim_{\e\to 0}\PP\Big(\sup_{s\in [0, \Delta_\e]} |\xX^\e_s(z)+Db(0)\Psi_s^{-1}z|>  2 r_\e\Big)=\lim_{\e\to 0}\PP\Big(\sup_{s\in [0, \Delta_\e]} |\xX^\e_s(z)|>  2 r_\e\Big)=0
\end{align*}
for $|z|\lqq r_\e$ for sufficiently small $\e$. 
Taylor's theorem combined with
the jump size and spatial localizations yields
for $\e$ sufficiently small 
\begin{align*}
\sup_{s\in [0, \Delta_\e]}\EE\Big[\Big|-Db(0)\xX^\e_{s}(z)-\ti b(\xX^\e_{s}(z)+\Psi_s^{-1}z)h(\xX^\e_{s}(z)+\Psi_s^{-1}z)\Big| \Big]^2  \lqq C_2 r^2_\e,
\end{align*}
where $C_2>0$ only depends $\max_{|u|\lqq 2}|b(u)|$, $\max_{|u|\lqq 2}|Db(u)|$ and $\max_{|u|\lqq 2}|D^2b(u)|$. 
Hence the variation of constants formula yields 
\begin{align}
\nonumber
|\phi_{\Delta_\e}(\theta_\e)|^2
&\lqq e^{-C \frac{|\theta|^\al}{\Delta_\e} \Delta_\e} 1 + C_2 \e^{2(1-\vartheta)}  \frac{|\theta_\e|^{1/2}}{\e^{3/2}} 
e^{-C \frac{|\theta|^\al}{\Delta_\e} \Delta_\e} 
\int_0^{\Delta_\e}  e^{C \frac{|\theta|^\al}{\Delta_\e} s} ds \nonumber\\
&= e^{-C |\theta|^\al}  + C_3\frac{\Delta_\e^{1-\frac{1}{2\al}} }{\e^{-2 \vt} |\theta|^{\al-\frac{1}{2}}}  
\Big(1- e^{-C \frac{|\theta|^\al}{\Delta_\e} \Delta_\e}\Big)\lqq e^{-C |\theta|^\al}  + C_3\frac{\Delta_\e^{1-\frac{1}{2\al}} }{\e^{-2 \vt} |\theta|^{\al-\frac{1}{2}}}\label{ine: fourieroder}. 
\end{align}
Note that the shift $a_\e^0$ does not change the modulus of $\hat f_\e(\theta)$ and 
hence the integrability in $\theta$. 
The parameter value $\al \in (\nicefrac{3}{2}, 2)$ implies that 
\begin{align*}
\int_{|\theta|>K} |\hat f_\e(\theta)|^2 \ud \theta 
\lqq \int_{|\theta|>K } 
e^{-C |\theta|^\al} \ud \theta  + 2 C_3\frac{\Delta_\e^{1-\frac{1}{2\al}} }{\e^{-2 \vt}} \int_{|\theta| > K} \frac{1}{|\theta|^{\al-\frac{1}{2}}} \ud \theta < \infty.  
\end{align*}
Since $\Delta_\e = \e^{\frac{\al}{2}}$ we have the desired limit \eqref{e: UI} for any $|z|\lqq {r_\e}$ 
\begin{align*}
\lim_{K\ra \infty} \limsup_{\e \ra 0} \int_{|\theta|>K} |\hat f_\e(\theta)|^2 \ud \theta
&\lqq \lim_{K\ra \infty} \int_{|\theta|>K } 
e^{-C |\theta|^\al} \ud \theta = 0. 
\end{align*}
In order to see the uniformity we refer to the continuity of the map
\[
z \mapsto  \norm{\frac{X^{\e}_{\Delta_\e}(z)-\Psi^{-1}_{\Delta_\e}z}{\e \Delta^{1/\alpha}_{\e}} + a_\e^0- U}.
\]
That is, the supremum is taken at some value $z_\e$ such that 
\begin{align*}
\sup_{|z|\lqq {r_\e}} \norm{\frac{X^{\e}_{\Delta_\e}(z)-\Psi^{-1}_{\Delta_\e}z}{\e \Delta^{1/\alpha}_{\e}} + a_\e^0- U} 
= \norm{\frac{X^{\e}_{\Delta_\e}(z_\e)-\Psi^{-1}_{\Delta_\e}z_\e}{\e \Delta^{1/\alpha}_{\e}} + a_\e^0- U}. 
\end{align*}
In the previous calculation the only property of $z$ we use is that $|z|\lqq r_\e$. 
Hence all previous results remain valid for $z$ being replaced by $z_\e$. 
This finishes the proof of Proposition \ref{p:localcoupling}.
\end{proof}

\begin{proof}[Proof of Proposition~\ref{prop: stc}:] 
The proof consists of the domination of the error terms $G_1$ and $G_2$ in \eqref{e:G1G2}. 
The result of Subsubsection \ref{sss: Step1} is the convergence $G_2\to 0$ as $\e \ra 0$. 
The term $G_1$ is estimated by inequality \eqref{eq: pivot disintegration} 
whose right-hand side is dominated by the terms given in \eqref{e:secondorder}, Proposition~\ref{tomate} and Proposition \ref{p:localcoupling}, all of which tend to $0$ as $\e \to 0$. 
This finishes the proof of Proposition~\ref{prop: stc}. 
\end{proof}

\bigskip

\subsection{\textbf{{Inhomogeneous O-U} approximation of the limiting distribution (Prop.~\ref{prop: equil})}}\label{sec: local}\hfill\\

\begin{proof}[Proof of Proposition~\ref{prop: equil}: ]
\noindent Let $x_0\in \RR^d$ and $t>0$. The triangle inequality yields
\begin{align}\label{main estimate}
\norm{\mu^{\e}-\mu^{\e}_*}\lqq \norm{\mu^{\e}-X^{\e}_{t}(x_0)}+
\norm{X^{\e}_{t}(x_0)-Y^{\e, x_0}(t; 0, x_0)}+\norm{Y^{\e, x_0}(t; 0, x_0)-\mu^{\e}_*}.
\end{align}
Here, we estimate the first-term of the right-hand side of inequality \eqref{main estimate}. By disintegration and the invariance property of $\mu^\e$ it follows
\[
\norm{\mu^{\e}-X^{\e}_{t}(x_0)}\lqq \int_{\RR^d}\norm{X^{\e}_t(u)-X^{\e}_{t}(x_0)}\mu^{\e}(\ud u).
\]
Let $s_\e\gg t^{x_0}_{\e}$ for sufficiently small $\e$.
The triangle inequality for the total variation distance implies 
\begin{align*}
&\int_{\RR^d}\norm{X^{\e}_{s_\e}(u)-X^{\e}_{s_\e}(x_0)}\mu^{\e}(\ud u)
\lqq \int_{\RR^d}\norm{X^{\e}_{s_\e}(u)-Y^{\e,u}(s_\e; 0, u)}\mu^{\e}(\ud u)\\
&\qquad+\int_{\RR^d}\norm{Y^{\e,u}(s_\e; 0, u)-Y^{\e,x_0}(s_\e; 0, x_0)}\mu^{\e}(\ud u)+\norm{Y^{\e,x_0}(s_\e; 0, x_0)-X^{\e}_{s_\e}(x_0)}.
\end{align*}
Since the total variation distance is bounded by one, we have for any $K>0$
\begin{align*}
\int_{\RR^d}\norm{X^{\e}_{s_\e}(u)-Y^{\e,u}(s_\e; 0, u)}\mu^{\e}(\ud u)
&\lqq 
\int_{|u|\lqq K}\norm{X^{\e}_{s_\e}(u)-Y^{\e,u}(s_\e; 0, u)}\mu^{\e}(\ud u) +\mu^{\e}(|u|\gqq K)
\end{align*}
and
\begin{align*}
&\int_{\RR^d}\norm{Y^{\e,u}(s_\e; 0, u)-Y^{\e,x_0}(s_\e; 0, x_0)}\mu^{\e}(\ud u)\\
&\qquad \lqq 
\int_{|u|\lqq K}\norm{Y^{\e,u}(s_\e; 0, u)-Y^{\e,x_0}(s_\e; 0, x_0)}\mu^{\e}(\ud u)
+\mu^{\e}(|u|\gqq K).
\end{align*}
Combining the preceding inequalities with inequality \eqref{main estimate} we obtain
\begin{align}
\norm{\mu^{\e}-\mu^{\e}_*}
&\lqq I_1+2I_2+2I_3+I_4+I_5,\label{e:pivpivpiv}
\end{align}
where 
\begin{align*}
I_1 &:= \norm{Y^{\e,x_0}(s_\e; 0, x_0)-\mu^{\e}_*},\qquad 
I_2 := \norm{Y^{\e,x_0}(s_\e; 0, x_0)-X^{\e}_{s_\e}(x_0)},\\
I_3 &:= \mu^{\e}(|u|\gqq K), \qquad I_4 := \int_{|u|\lqq K}\norm{X^{\e}_{s_\e}(u)-Y^{\e,u}(s_\e; 0, u)}\mu^{\e}(\ud u) \qquad \mbox{ and } \\
I_5 &:= \int_{|u|\lqq K}\norm{Y^{\e,u}(s_\e; 0, u)-Y^{\e,x_0}(s_\e; 0, x_0)}\mu^{\e}(\ud u),
\end{align*}
for any $x_0\in \RR^d$ and $s_{\e}\gg t^{x_0}_{\e}$ for sufficiently small $\e>0$. 
The remainder of the proof consists of showing that each of the terms $I_i \ra 0$ as $\e \ra 0$, $i=1, \dots, 5$. 
\\

\noindent
\textbf{Estimates for $I_1$ in \eqref{e:pivpivpiv}.}
Let $x_0\not=0$. By Proposition~\ref{prop: windows} we have 
\begin{equation}\label{eq: I1}
\lim\limits_{\e\to 0}\norm{Y^{\e,x_0}(s_\e; 0, x_0)-\mu^{\e}_*}=0\quad \textrm{ for any } s_{\e}\gg t^{x_0}_{\e}, \textrm{ as } \e\to 0.
\end{equation}
By \eqref{ine: fine} and Lemma~\ref{lem: continuidad} in Appendix~\ref{Appendix C} we have 
that for any $K>0$ 
\begin{align*}
\lim\limits_{\e\to 0}\sup_{|x_0|\lqq K}\norm{Y^{\e,x_0}(s_\e; 0, x_0)-\mu^{\e}_*}=0\quad \textrm{ for any } s_{\e}\gg \ln(1/\e) \textrm{ as } \e\to 0.\\
\end{align*}

\noindent
\textbf{Estimates for $I_2$ in \eqref{e:pivpivpiv}.}
We estimate the second term as follows
\begin{align}
\norm{X^{\e}_{s_\e}(x_0)-Y^{\e,x_0}(s_\e; 0, x_0)}\lqq &
\norm{X^{\e}_{\Delta_\e}(X^{\e}_{s_\e-\Delta_\e}(x_0))-Y^{\e, x_0}({\Delta_\e};
s_{\e}-\Delta_\e, X^{\e}_{s_\e-\Delta_\e}(x_0))}\nonumber\\
&\hspace{-4.5cm}+
\norm{Y^{\e, x_0}({\Delta_\e};s_{\e}-\Delta_\e, X^{\e}_{s_\e-\Delta_\e}(x_0))-Y^{\e, x_0}(\Delta_\e;s_\e-\Delta_\e, Y^{\e, x_0}(s_\e-\Delta_\e;0, x_0))}.\label{eq: estI2}
\end{align}
By Proposition~B.1 in Appendix \ref{ap: multiscale} our estimates in the previous sections remain valid up to times of order $\e^{-\vartheta}$ for some $\vartheta>0$. In the sequel, we set $s_\e:= \ln^2(\e)$. 

\noindent We start with the second term on the right-hand side of \eqref{eq: estI2} and lighten the notation. By Proposition~\ref{prop: secondorder} it is not hard to see that Lemma~\ref{lem: dependence} remains valid for 
\[
z=X^{\e}_{s_\e-\Delta_\e}(x_0) \qquad \mbox{ and } \qquad \ti z=Y^{\e, x_0}(s_\e-\Delta_\e;0, x_0). 
\]
For the convenience of the reader, we restate it here.
\begin{lem}
Let $\gamma_\e={\Delta_\e}^{-\nicefrac{1}{\alpha}}$, where $\Delta_\e=\e^{\nicefrac{\alpha}{2}}$.
Then
\begin{equation}\label{lim: siropaes}
\lim_{\e \ra 0}\PP_\e(\gamma_\e |z - \ti z|> \eta \e)=0 \quad \textrm{ for any } \eta>0,
\end{equation}
where $\PP_\e(\ud u, \ud \tilde{u})$ denotes the joint distribution   
$\PP\left(X^{\e}_{s_\e-\Delta_\e}(x_0)\in \ud u, Y^{\e, x_0}(s_\e-\Delta_\e;0, x_0)\in \ud\ti u\right)$. 
\end{lem}
\noindent 
By \eqref{e:P5I1}, \eqref{e:P5I11}, Proposition \ref{tomate} and the application the preceding statement to \eqref{cerveza} we deduce
\[
\lim\limits_{\e \to 0}\norm{Y^{\e, x_0}({\Delta_\e};s_{\e}-\Delta_\e, X^{\e}_{s_\e-\Delta_\e}(x_0))-Y^{\e, x_0}(\Delta_\e;s_\e-\Delta_\e, Y^{\e, x_0}(s_\e-\Delta_\e,0, x_0))}=0.
\]
In the sequel, we continue with the first term on the right-hand side of \eqref{eq: estI2}.
By Corollary~\ref{cor: demomento} in Appendix \ref{Appendix D} we have for any $\eta>0$, $\vartheta\in (0,1)$ and $K>0$ 
\[
\lim\limits_{\e\to 0} 
\sup_{|x_0|\lqq K}\mathbb{P}(|X^{\e}_{s_\e-\Delta_\e}(x_0)|\gqq \eta r_\e)=0.
\] 
Since $s_\e-\Delta_\e\gg t^{x_0}_\e$, it is straightforward to see that the
limit
\eqref{eq: short coupling} remains valid for $T^x_\e$ being replaced by $s_\e - \Delta_\e$. 
Consequently, we have
\begin{align*}
\lim\limits_{\e\to 0}\norm{X^{\e}_{\Delta_\e}(X^{\e}_{s_\e-\Delta_\e}(x_0))-Y^{\e, x_0}({\Delta_\e};s_{\e}-\Delta_\e, X^{\e}_{s_\e-\Delta_\e}(x_0))}=0.\\
\end{align*}

\noindent
\textbf{Estimates for $I_3$ in \eqref{e:pivpivpiv}.} 
By Corollary~\ref{lem: demomento} in Appendix \ref{Appendix D} we have for all $\beta' \lqq \beta\wedge 1$
a positive constant $C$ such that 
\[
\EE[|X^{\e}_t(x)|^{\beta'}]\lqq C \e^{\beta'} + |\varphi^x_t|^{\beta'} \lqq C \e^{\beta'} + e^{-\delta {\beta'} t} |x|^{\beta'} 
\quad \textrm{ for } t\gqq 0, ~x\in \RR^d.
\]
Let $n\in \mathbb{N}$. Then
\[
\EE[|X^{\e}_t(x)|^{\beta'} \wedge n]\lqq e^{-\delta{\beta'} t}|x|^{\beta'}\wedge n+ C\e^{\beta'} \wedge n \quad \textrm{ for } t\gqq 0, ~x\in \RR^d.
\]
Since $\mu^\e$ is stationary, we estimate for all $t\gqq 0$
\begin{align*}
\int_{\RR^d} (|u|^{\beta'}\wedge n) \mu^{\e}(\ud u)=
\int_{\RR^d} \EE[|X^{\e}_t(u)|^{\beta'}\wedge n] \mu^{\e}(\ud u)
\lqq \int_{\RR^d} (e^{-\delta {\beta'} t}|u|^{\beta'} \wedge n)\mu^{\e}(\ud u)+
C \e^{\beta'} \wedge n.
\end{align*}
By the dominated convergence theorem we infer 
\[
\lim\limits_{t\to \infty}\int_{\RR^d} (e^{-\delta {\beta'} t}|u|^{\beta'} \wedge n)\mu^{\e}(\ud u)=0\quad \textrm{ for all } n\in \NN, ~\e\in (0,1].
\]
Therefore,  we have for all $n\in \mathbb{N}$ such $n>C \e$
\[
\int_{\RR^d} (|u|^{\beta'} \wedge n) \mu^{\e}(\ud u)\lqq C\e^{\beta'} \wedge n\lqq C \e^{\beta'}.
\]
By the monotone convergence theorem we obtain
\begin{equation}\label{eq:invariant moment est}
\int_{\RR^d} |u|^{\beta'} \mu^{\e}(\ud u)\lqq C \e^{\beta'}.
\end{equation}
The Markov inequality and \eqref{eq:invariant moment est} imply 
\begin{align*}
\mu^{\e}(|u|\gqq K)\lqq \frac{\int_{\RR^d}|u|^{\beta'}\mu^{\e}(
\ud u)}{K^{\beta'}}\lqq \frac{C \e^{\beta'}}{K^{\beta'}} .\\
\end{align*}

\noindent
\textbf{Estimates for $I_5$ in \eqref{e:pivpivpiv}.}
We start with the triangle inequality
\begin{align*}
&\int_{|u|\lqq K}\norm{Y^{\e,u}(s_\e; 0, u)-Y^{\e,x_0}(s_\e; 0, x_0)}\mu^{\e}(\ud u)\\
&\qquad \lqq \sup_{|u|\lqq K} \norm{Y^{\e,u}(s_\e; 0, u)-\mu^{\e}_*}+
\norm{\mu^{\e}_*-Y^{\e,x_0}(s_\e; 0, x_0)}.
\end{align*}
The second term of the preceding inequality 
is equal to $I_1$ and tends to $0$ as $\e \ra 0$.
By \eqref{ine: fine} and since in Lemma \ref{lem:convergenceindist}, Item (1), in Appendix~\ref{Appendix C} it is shown that $\mu^*_\e$ is the law of $\e Z_\infty$ we have
\[
\norm{Y^{\e,u}(s_\e; 0, u)-\mu^{\e}_*} 
\lqq \norm{Y^{\e,u}(s_\e; 0, u)-Z_\infty}+\norm{\nicefrac{\varphi^u_{s_\e}}{\e}+Z_\infty-Z_\infty}\quad \textrm{ for any } u\in \RR^d.
\]
We start with the first term. 
By Lemma~\ref{lem: cvtlineal} in Appendix~\ref{Appendix C} we have
\[
\lim\limits_{\e\to 0}\sup_{|u|\lqq K}\norm{Y^{\e,u}(s_\e; 0, u)-Z_\infty}=0.
\]
We treat the second term.
Let $\eta>0$. By the shift-continuity of $L^1$ distance yields that there exists $\rho:=\rho(\eta)>0$ such that 
\[
\norm{u+Z_\infty-Z_\infty}\lqq \eta\quad \textrm{ whenever } |u|\lqq \rho.
\]
Note that for $|u|\lqq K$ we have
\[
\big|\frac{\varphi^u_{s_\e}}{\e}\big| \lqq \frac{e^{-\delta s_\e}|u|}{\e}\lqq 
\frac{e^{-\delta s_\e}K}{\e}<\rho \quad \textrm{ for sufficiently small } \e.
\]
Therefore 
\[
\limsup_{\e \to 0}\int_{|u|\lqq K}\norm{\frac{\varphi^u_{s_\e}}{\e}+Z_\infty-Z_\infty}\mu^{\e}(\ud u)\lqq \eta,
\]
and consequently 
\begin{align*}
\lim_{\e \to 0}\int_{|u|\lqq K}\norm{\frac{\varphi^u_{s_\e}}{\e}+Z_\infty-Z_\infty}\mu^{\e}(\ud u)=0.\\
\end{align*}

\noindent
\textbf{Estimates for $I_4$ in \eqref{e:pivpivpiv}.}
Note that
\begin{align}
&\norm{X^{\e}_{s_\e}(x)-Y^{\e, x}(s_\e;0, x)}\lqq 
\norm{X^{\e}_{\Delta_\e}(X^{\e}_{s_\e-\Delta_\e}(x))-Y^{\e, x}({\Delta_\e};s_{\e}-\Delta_\e, X^{\e}_{s_\e-\Delta_\e}(x))}\nonumber\\
&\qquad\quad +
\norm{Y^{\e, x}({\Delta_\e};s_{\e}-\Delta_\e, X^{\e}_{s_\e-\Delta_\e}(x))-Y^{\e,x}(\Delta_\e;s_\e-\Delta_\e, Y^{\e, x}(s_\e-\Delta_\e; 0, x))}.\label{ine: principalI4}
\end{align}
We start with the first term. 
Recall that $r_\e=\e^{1-\vt}$ for $\vt\in (0,\nicefrac{1}{4})$.
For \mbox{$\PP^{x}_\e(\ud z)=\PP(X^{\e}_{s_\e-\Delta_\e}(x)\in \ud z)$}
disintegration yields
\begin{align*}
&\int_{|u|\lqq K}\norm{X^{\e}_{\Delta_\e}(X^{\e}_{s_\e-\Delta_\e}(u))-Y^{\e, u}({\Delta_\e};s_{\e}-\Delta_\e, X^{\e}_{s_\e-\Delta_\e}(u))}\mu^{\e}(\ud u)\\
&\lqq  \int_{|u|\lqq K}\int_{|z|\lqq 2r_\e}\norm{X^{\e}_{\Delta_\e}(z)-Y^{\e, u}({\Delta_\e};s_{\e}-\Delta_\e, z)}\PP^{u}_\e(\ud z)\mu^{\e}(\ud u)+\sup_{|u|\lqq K}\PP(|X^{\e}_{s_\e-\Delta_\e}(u))|> 2r_\e)\\
& \lqq \sup_{|u|\lqq K}\sup_{|z|\lqq 2r_\e}\norm{X^{\e}_{\Delta_\e}(z)-Y^{\e, u}({\Delta_\e};s_{\e}-\Delta_\e, z)}
+\sup_{|u|\lqq K}\PP(|X^{\e}_{s_\e-\Delta_\e}(u)|> 2r_\e)\\
&=\norm{X^{\e}_{\Delta_\e}(z_\e)-Y^{\e, u_\e}({\Delta_\e};s_{\e}-\Delta_\e, z_\e)}
+\sup_{|u|\lqq K}\PP(|X^{\e}_{s_\e-\Delta_\e}(u)|> 2r_\e)
\end{align*}
for some  $|u_\e|\lqq K$ and $|z_\e|\lqq 2r_\e$.
The right-hand side of the preceding inequality tends to zero, as $\e\to 0$. This is due to~Proposition~\ref{prop: stc} and Corollary~\ref{cor: demomento} in Appendix \ref{Appendix D}.

\noindent We continue with the second term on the right-hand side of \eqref{ine: principalI4}.
Let 
\[
\PP^u_\e(\ud z, \ud \ti z)=\PP(X^{\e}_{s_\e-\Delta_\e}(u) \in \ud z, Y^{\e, u}(s_\e-\Delta_\e;0, u)\in \ud \ti z).\]
Using the shift continuity \eqref{e: fresa} we fix $\rho$ and choose $\eta>0$ accordingly. 
Again, by disintegration we have 
\begin{align}
&\int_{|u|\lqq K}\norm{Y^{\e, u}({\Delta_\e};s_{\e}-\Delta_\e, X^{\e}_{s_\e-\Delta_\e}(u))-Y^{\e, u}(\Delta_\e;s_\e-\Delta_\e, Y^{\e, u}(s_\e-\Delta_\e;0, u))}\mu^{\e}(\ud u)\nonumber\\
&\lqq 
\int\limits_{|u|\lqq K}\int_{\substack{|z-\ti z|\lqq \eta \e \Delta_\e^{1/\al},\nonumber\\
|z|\lqq 2r_\e,|\ti z|\lqq 2r_\e}}
\norm{Y^{\e, u}({\Delta_\e};s_{\e}-\Delta_\e, z)-Y^{\e, u}(\Delta_\e;s_\e-\Delta_\e, \ti z)}\PP^u_\e(\ud z, \ud \ti z)\mu^{\e}(\ud u)\nonumber\\
&+\sup_{|u|\lqq K}\PP(|X^{\e}_{s_\e-\Delta_\e}(u)-Y^{\e, u}(s_\e-\Delta_\e,u,0)|>\eta \e\Delta_\e^{1/\al})\nonumber\\
&+\sup_{|u|\lqq K} \PP(|X^{\e}_{s_\e-\Delta_\e}(u)|>2r_\e) +
\sup_{|u|\lqq K}\PP(|Y^{\e, u}(s_\e-\Delta_\e; 0, u)|>2r_\e).\label{e:I41}
\end{align}
where the first term on the right-hand side is estimated by 
\begin{equation}\label{e:I42}
\sup_{|u|\lqq K}
\sup_{\substack{|z-\ti z|\lqq \eta \e \Delta_\e^{1/\al},\\
|z|\lqq 2r_\e,|\ti z|\lqq 2r_\e}}
\norm{Y^{\e, u}({\Delta_\e};s_{\e}-\Delta_\e, z)-Y^{\e, u}(\Delta_\e;s_\e-\Delta_\e, \ti z)}.
\end{equation}
We prove that the right-hand sides of \eqref{e:I41} and \eqref{e:I42} tend to zero, as $\e\to 0$. 
Due to limit \eqref{lim: siropaes} it follows
\[
\lim_{\e\ra 0}\sup_{|u|\lqq K}\PP(|X^{\e}_{s_\e-\Delta_\e}(u)-Y^{\e, u}(s_\e-\Delta_\e,u,0)|>\eta \e\Delta_\e^{1/\al}) =0.
\]
By Corollary~\ref{cor: demomento} in Appendix \ref{Appendix D} and a straightforward adaptation 
for the linearization $Y^{\e, u}$, we have
\[
\lim\limits_{\e\to 0}\sup_{|u|\lqq K} \PP(|X^{\e}_{s_\e-\Delta_\e}(u)|>2r_\e)=
\lim\limits_{\e\to 0}\sup_{|u|\lqq K}\PP(|Y^{\e, u}(s_\e-\Delta_\e; 0, u)|>2r_\e)=0.
\]
We continue with the term \eqref{e:I42} 
\[
\sup_{|u|\lqq K}
\sup_{\substack{|z-\ti z|\lqq \eta \e \Delta_\e^{1/\al},\\
|z|\lqq 2r_\e,|\ti z|\lqq 2r_\e}}
\norm{Y^{\e, u}({\Delta_\e};s_{\e}-\Delta_\e, z)-Y^{\e, u}(\Delta_\e;s_\e-\Delta_\e, \ti z)}.\]
By \eqref{eq: representation0} we have 
\begin{align*}
&Y^{\e,u}(\Delta_\e;s_\e-\Delta_\e,z)\\
&=(\Phi^\e_{\Delta_\e}(u))^{-1}z+
(\Phi^\e_{\Delta_\e}(u))^{-1}
\int_{0}^{\Delta_\e}
\Phi^\e_s(u)\left(Db(\varphi^u_{s_\e-\Delta_\e+s})\varphi^u_{s_\e-\Delta_\e+s}-b(\varphi^{u}_{s_\e-\Delta_\e+s})\right)\ud s+
\e U_\e.
\end{align*}
By the shift and scale invariance of the total variation distance we obtain
\begin{align*}
&\norm{Y^{\e, u}({\Delta_\e};s_{\e}-\Delta_\e, z)-Y^{\e, u}(\Delta_\e;s_\e-\Delta_\e, \ti z)}\\[2mm]
&= \norm{(\Phi^\e_{\Delta_\e}(u))^{-1}z + \e U_\e) - (\Phi^\e_{\Delta_\e}(u))^{-1}\ti z + \e U_\e)}= \norm{\Big((\Phi^\e_{\Delta_\e}(u))^{-1}(z -\ti z)+ \e U_\e\Big) -  \e U_\e}\\
&= \norm{\Big(\frac{(\Phi^\e_{\Delta_\e}(u))^{-1}(z- \ti z)}{\e \Delta_\e^{\nicefrac{1}{\al}}}+\frac{U_\e}{\Delta_\e^{\nicefrac{1}{\al}}}+ a^u_\e\Big) - \Big(\frac{U_\e}{\Delta_\e^{\nicefrac{1}{\al}}} +a^u_\e\Big)
}.
\end{align*}
Hence 
\begin{align*}
&\norm{Y^{\e, u}({\Delta_\e};s_{\e}-\Delta_\e, z)-Y^{\e, u}(\Delta_\e;s_\e-\Delta_\e, \ti z)}\\
&\qquad \lqq \norm{\Big(\frac{(\Phi^\e_{\Delta_\e}(u))^{-1}(z- \ti z)}{\e \Delta_\e^{\nicefrac{1}{\al}}}+\frac{U_\e}{\Delta_\e^{\nicefrac{1}{\al}}}+ a^u_\e\Big) - \Big(\frac{(\Phi^\e_{\Delta_\e}(u))^{-1}(z- \ti z)}{\e \Delta_\e^{\nicefrac{1}{\al}}} +U\Big)
}\\
&\qquad \qquad +
\norm{\Big(\frac{(\Phi^\e_{\Delta_\e}(u))^{-1}(z- \ti z)}{\e \Delta_\e^{\nicefrac{1}{\al}}} +U\Big) - U
}
+~\norm{U  - \Big(\frac{U_\e}{\Delta_\e^{\nicefrac{1}{\al}}} +a^u_\e\Big)
}\\
&\qquad = 2 \norm{U  - \Big(\frac{U_\e}{\Delta_\e^{\nicefrac{1}{\al}}} +a^u_\e\Big)} 
+ \norm{\Big(\frac{(\Phi^\e_{\Delta_\e}(u))^{-1}(z- \ti z)}{\e \Delta_\e^{\nicefrac{1}{\al}}} +U\Big) - U
},
\end{align*}
where $U\stackrel{d}= \mathcal{S}_\alpha(\Lambda_1)$. 
Proposition~\ref{tomate} yields 
\[
\lim_{\e \ra 0} \sup_{|u|\lqq K} \norm{U  - \Big(\frac{U_\e}{\Delta_\e^{\nicefrac{1}{\al}}} +a^u_\e\Big)} =0. 
\]
It remains to show for $\eta>0$ 
\begin{equation}\label{e: stepa}
\sup_{|u|\lqq K}
\sup_{\substack{|z-\ti z|\lqq \eta \e \Delta_\e^{1/\al},\\
|z|\lqq 2r_\e,|\ti z|\lqq 2r_\e}} \norm{\Big(\frac{(\Phi^\e_{\Delta_\e}(u))^{-1}(z- \ti z)}{\e \Delta_\e^{\nicefrac{1}{\al}}} +U\Big) - U
}. 
\end{equation}
Recall that \eqref{e: fresa} implies that 
\[
\Big|\frac{(\Phi^\e_{\Delta_\e}(u))^{-1}(z- \ti z)}{\e \Delta_\e^{\nicefrac{1}{\al}}}\Big| \lqq \sqrt{d} \eta  
\]
yields that \eqref{e: stepa} is bounded from above by $\rho$. 
Sending first $\e\to 0$ and then~$\rho\to 0 $  yields the limit of \eqref{e: stepa} equals $0$. 

\noindent
By \eqref{cotacontractiva} we have for $|z- \ti z|\lqq \eta \e \Delta_\e^\frac{1}{\al}$ that
\[
\Big|\frac{(\Phi^\e_{\Delta_\e}(u))^{-1}(z- \ti z)}{\e \Delta_\e^{\nicefrac{1}{\al}}}\Big| 
\lqq \sqrt{d} \eta.
\]
Hence $I_4 \ra 0$ as $\e \ra 0$. This completes the proof of Proposition \ref{prop: equil}. 
\end{proof}

\bigskip

\appendix
\section{\textbf{The deterministic dynamics}} \label{A}\hfill

This section gathers all results concerning the deterministic fine dynamics 
of the solution $\varphi_t^x$ of \eqref{dde1.1} under Hypothesis \ref{hyp: potential}. 
The following lemma is of interest 
since it shows that the time scale $t_\e^x$ 
yields an estimate on the deterministic dynamics with 
of order exactly $\e$. 

\begin{lem}\label{lem:ordenepsilon}
Let $\Delta_\e>0$ such that 
$\lim\limits_{\e \to 0}\Delta_\e=0$.
Let $\rho\in \mathbb{R}$ we define
$T^x_\e=t^x_\e-\Delta_\e+\rho\cdot w^x_\e$,
where $t^x_\e $ and $w^x_\e$ are given in Theorem~\ref{thm: main result}. Then
there exists a positive constant $C(|x|,\rho)$ that depends continuously on $|x|$ such that
$|\varphi^{x}_{T^x_\e}|\lqq C(|x|,\rho)\e$.
\end{lem}

\begin{proof}
By Lemma~\ref{asymp} we have
\begin{equation}\label{eq: otravez}
\lim_{t \to \infty} 
\left| \frac{e^{\lambda_x t}}{t^{\ell_x-1}} \varphi^x_{\tau_x+t} - v(t,x) \right|=0,
\end{equation}
where $v(t,x)=\sum_{k=1}^m e^{i\theta^k_x t}v^k_x$.
A straightforward calculation shows that
\begin{equation}\label{eq: denuevo}
\lim\limits_{\e \to 0} 
\frac{(T^x_\e)^{\ell_x-1} e^{-\lambda_x T^x_\e}}{\e}=\lambda_x^{1-\ell_x}e^{-\rho}.
\end{equation}
Then the triangle inequality yields
\[
\begin{split}
|\varphi^{x}_{T^x_\e}|&\lqq |\varphi^{x}_{T^x_\e}- (T^x_\e)^{\ell_x-1} e^{-\lambda_x T^x_\e} v(T^x_\e,x)|+(T^x_\e)^{\ell_x-1} e^{-\lambda_x T^x_\e}|v(T^x_\e,x)|\\[2mm]
&=(T^x_\e)^{\ell_x-1} e^{-\lambda_x T^x_\e}\left|\frac{e^{\lambda_x T^x_\e}\varphi^{x}_{T^x_\e}}{(T^x_\e)^{\ell_x-1} }-v(T^x_\e,x)\right|+
(T^x_\e)^{\ell_x-1} e^{-\lambda_x T^x_\e}|v(T^x_\e,x)|
\lqq C(|x|,\rho)\e,
\end{split}
\]
where the last inequality follows from limit \eqref{eq: otravez} and limit \eqref{eq: denuevo}.
\end{proof}
The following strong version of the Gr\"onwall-Bellman lemma frequently used and given for completeness. 
\begin{lem}[Gr\"onwall-Bellman inequality]
\label{lem: gronwall}\hfill\\
Let $T> 0$ be fixed. Let $g:[0,T]\rightarrow \mathbb{R}$ be a  $\mathcal{C}^1$-function and $h:[0,T]\rightarrow \mathbb{R}$ be continuous.
If
\[
\frac{\ud}{\ud t}g(t)\lqq -ag(t)+h(t)\quad \textrm{ for any } t\in [0,T],
\] 
where $a\in \mathbb{R}$, and the derivative at $0$ and $T$ are understanding as the right and left derivatives, respectively.  Then 
\[
g(t)\lqq e^{-at}g(0)+e^{-at}\int_{0}^{t}{e^{as}h(s)}\ud s \quad \textrm{ for any }
t\in[0,T].
\]
Moreover, if $a\not=0$ we have
\[
|g(t)|\lqq e^{-at}|g(0)|+\frac{(1-e^{-at})}{a}\max_{s\in[0,t]}|h(s)|
\quad \textrm{ for any }
t\in[0,T].
\]
\end{lem}
For the proof, see for instance Theorem~1.3.3 page 15 of \cite{PACH}. Due to the variation of constants formula, the proof of linear cutoff relies essentially on precise norm estimates of the homogeneous and inhomogeneous linear solution flow, which are gathered in the following lemma. 
\begin{lem}\label{lem: cotainfsup}
Let $(\varphi^x_t)_{t\gqq 0}$ be the solution of \eqref{dde1}. We
consider for any fixed $T\gqq 0$ the solution $\Phi=(\Phi_t(x))_{t\gqq 0}$  of the matrix differential equation 
\[
\frac{\ud}{\ud t} \Phi_t = \Phi_t Db(\varphi^x_{T+t})\quad \textrm{ with }  \Phi_0 = I_d,
\]
the solution $\Psi=(\Psi_t)_{t\gqq 0}$ of the matrix differential equation 
\[
\frac{\ud}{\ud t} \Psi_t = \Psi_t Db(0)\quad \textrm{ with }  \Psi_0 = I_d 
\]
and the standard matrix $2$-norm $|\cdot|$. Then
the following statements are valid for any $0\lqq s\lqq t$.
\begin{itemize}
\item[i)] It follows
\begin{equation}\label{cotacontractiva}
|\Phi^*_s(x)(\Phi^{-1}_t(x))^*|\lqq \sqrt{d}e^{-\delta(t-s)}
\quad \textrm{ and }\quad |\Psi^*_s(\Psi^{-1}_t)^*|\lqq \sqrt{d}e^{-\delta(t-s)}.
\end{equation}
\item[ii)] For $C(|x|)=\max\limits_{|u|\lqq |x|}|Db(u)|$ we have
\begin{equation*}
|\Phi^*_t(x)(\Phi^{-1}_s(x))^*|\lqq \sqrt{d}e^{-C(|x|)(t-s)}\quad \textrm{ and }\quad |\Psi^*_t(\Psi^{-1}_s)^*|\lqq \sqrt{d}e^{-|Db(0)|(t-s)}.
\end{equation*}
\item[iii)]
Let $c_1=\nicefrac{1}{\sqrt{d}}$, $c_2=C(|x|)$, $c_3=\sqrt{d}$ and $c_4=\delta$, where $C(|x|)$ is the constant obtained in item ii).
Then for all $z\in \RR^d$
\begin{equation*}
c_1e^{-c_2 (t-s)}|z| \lqq |\Phi^*_s(x)(\Phi^{-1}_t(x))^*z|\lqq c_3e^{-c_4 (t-s)}|z|.
\end{equation*}
\item[iv)] There exist positive constant $\ti c_1$, $\ti c_2$, $\ti c_3$ and $\ti c_4$ such that for all $z\in \RR^d$
\begin{equation*}
\ti c_1e^{-\ti c_2 (t-s)}|z| \lqq |\Psi^*_s(\Psi^{-1}_t)^*z|\lqq \ti c_3e^{-\ti c_4 (t-s)}|z|.
\end{equation*}
\item[v)] For $C(|x|)$ given in item iii) we have
\begin{equation}\label{eq: v)}
|\Phi^{-1}_t(x)\Phi_s(x) -\Psi^{-1}_t\Psi_s |^2\lqq 
\frac{C^2(|x|) d^3}{4\delta^2}
|\varphi^x_{T}|^2
e^{-\delta t}(1-e^{-4\delta (t-s)}).
\end{equation}
 In particular, 
\begin{align}\label{eq: v1}
|\Phi^{-1}_t(x)\Phi_s(x) -\Psi^{-1}_t\Psi_s |^2\lqq 
\frac{C^2(|x|) d^3}{4\delta^2}
|\varphi^x_{T}|^2
e^{-\delta t}.
\end{align}
\end{itemize}
\end{lem}

\begin{proof}
Let  $t\gqq s\gqq 0$ be fixed.\\

\noindent\textit{Proof of item \textnormal{i)}.}
Define 
$\Pi^s_t(x):=\Phi^{-1}_t(x) \Phi_s(x)$.
Note that
\[
\frac{\ud}{\ud t}\Pi^s_t(x)=\frac{\ud}{\ud t}\Phi^{-1}_t(x) \Phi_s(x)=
-Db(\varphi^x_{T+t})\Phi^{-1}_t(x) \Phi_s(x)=-Db(\varphi^x_{T+t})\Pi^s_t(x).
\]
We denote by $\Pi^s_t(x)=((\Pi^s_t(x))_{i,j})_{i,j\in\{1,\ldots,d\}}$. Observe that
\[
\begin{split}
\frac{\ud}{\ud t}|\Pi^s_t(x)|^2&=-2\sum\limits_{i,j=1}^{d}(\Pi^s_t(x))_{i,j}\sum\limits_{k=1}^{d}(Db(\varphi^x_{T+t}))_{i,k}(\Pi^s_t(x))_{k,j}\\
&=-2\sum\limits_{j=1}^{d}\sum\limits_{i,k=1}^{d}
(\Pi^s_t(x))_{i,j}(Db(\varphi^x_{T+t}))_{i,k}(\Pi^s_t(x))_{k,j}
\lqq -2\delta |\Pi^s_t(x)|^2,
\end{split}
\]
where the last inequality follows from Hypothesis~\ref{hyp: potential}.
Since $\Pi^s_s(x)=I_d$, Lemma~\ref{lem: gronwall} yields
$|\Phi^{-1}_t(x) \Phi_s(x)|^2\lqq de^{-2\delta(t-s)}$.\\

\noindent
\textit{Proof of item \textnormal{ii)}.}
Let $\tilde{\Pi}^s_t(x):=\Phi^{-1}_s(x)\Phi_t(x)$.
Note that
\[
\frac{\ud}{\ud t}\tilde{\Pi}^s_t(x)=\Phi^{-1}_s(x)\frac{\ud}{\ud t}\Phi_t(x)=
\Phi^{-1}_s(x) \Phi_t(x) Db(\varphi^x_{T+t}) =\tilde{\Pi}^s_t(x) Db(\varphi^x_{T+t}).
\]
Observe that
$
|\<\tilde{z},Db(\varphi^x_{T+t})\tilde{z}\>|\lqq |Db(\varphi^x_{T+t})||\tilde{z}|^2
$  for  
$\tilde{z}\in \RR^d$.
By Hypothesis~\ref{hyp: potential} we obtain that $|\varphi^x_t|\lqq |x|e^{-\delta t}$ combined with $b\in \mathcal{C}^2$ implies $|Db(\varphi^x_t)|\lqq \max\limits_{|u|\lqq |x|}|Db(u)|$.
Let $C(|x|):=\max\limits_{|u|\lqq |x|}|Db(u)|$.
Here, we denote by
$\tilde{\Pi}^s_t(x)=((\tilde{\Pi}^s_t(x))_{i,j})_{i,j\in \{1,\ldots,d\}}$.
Then we have
\[
\begin{split}
\frac{\ud}{\ud t}|\tilde{\Pi}^s_t(x)|^2&=2\sum\limits_{i,j=1}^{d}(\tilde{\Pi}^s_t(x))_{i,j}\sum\limits_{k=1}^{d}
(\tilde{\Pi}^s_t(x))_{i,k}(Db(\varphi^x_{T+t}))_{k,j}\\
&=2\sum\limits_{i=1}^{d}\sum\limits_{k,j=1}^{d}
(\tilde{\Pi}^s_t(x))_{i,j}(Db(\varphi^x_{T+t}))_{k,j}(\tilde{\Pi}^s_t(x))_{i,k}
\lqq 2C|\tilde{\Pi}^s_t(x)|^2.
\end{split}
\]
Since $\tilde{\Pi}^s_s=I_d$, Lemma~\ref{lem: gronwall} yields
$|\Phi^{-1}_s(x)\Phi_t(x)|^2\lqq de^{2C(|x|)(t-s)}$.\\

\noindent
\textit{Proof of item \textnormal{iii)}.}
Let $z \in \mathbb{R}^d$ be fixed.
On the one hand, item i) yields
\[
|\Phi^*_s(x)(\Phi^{-1}_t(x))^*z|\lqq |\Phi^*_s(x)(\Phi^{-1}_t(x))^*||z|\lqq \sqrt{d}e^{-\delta(t-s)}|z|.
\]
On the other hand,
 we have
\begin{align*}
|z|=|(\Phi_t(x))^*(\Phi^{-1}_s(x))^*\Phi^*_s(x)(\Phi^{-1}_t(x))^*z|& \lqq 
|(\Phi_t(x))^*(\Phi^{-1}_s(x))^*|
|\Phi^*_s(x)(\Phi^{-1}_t(x))^*z|\\[2mm]
&\lqq 
\sqrt{d}e^{C(|x|)(t-s)}|\Phi^*_s(x)(\Phi^{-1}_t(x))^*z|,
\end{align*}
where the last inequality follows from item \textnormal{ii)}.
Consequently, 
\begin{align*}
|\Phi^*_s(x)(\Phi^{-1}_t(x))^*z|\gqq \frac{1}{\sqrt{d}}e^{-C(|x|)(t-s)}|z|.\\
\end{align*}

\noindent
\textit{Proof of item \textnormal{iv)}.} It follows analogously from item i) and ii). We omit the details.\\

\noindent
\textit{Proof of item \textnormal{v)}.}
Let $\Delta^{s}_t(x):=\Phi^{-1}_t(x)\Phi_s(x) -\Psi^{-1}_t\Psi_s $. Then
\[
\begin{split}
\frac{\ud}{\ud t} \Delta^{s}_t(x) &=\frac{\ud}{\ud t}\Phi^{-1}_t(x)\Phi_s(x)
-\frac{\ud}{\ud t}\Psi^{-1}_t\Psi_s \\[2mm]
&=- Db(\varphi^x_{T+t}) \Phi_t^{-1}(x)\Phi_s(x)+ 
Db(0) \Psi_t^{-1}\Psi_s\\[2mm]
&=- Db(\varphi^x_{T+t}) \Delta^{s}_t(x)+ 
(Db(0)-Db(\varphi^x_{T+t})) \Psi_t^{-1}\Psi_s.
\end{split}
\]
Here we denote by $\Delta^s_t(x)=((\Delta^s_t(x))_{i,j})_{i,j\in \{1,\ldots,d\}}$. Note that
\[
\begin{split}
\frac{\ud}{\ud t} |\Delta^{s}_t(x)|^2 &=2
\sum\limits_{i,j=1}^{d}(\Delta^s_t(x))_{i,j}\frac{\ud}{\ud t}(\Delta^s_t(x))_{i,j}\\
&\hspace{-1cm}=2\sum\limits_{i,j=1}^{d}(\Delta^s_t(x))_{i,j}\left(
\sum\limits_{k=1}^d 
- (Db(\varphi^x_{T+t}))_{i,k}(\Delta^{s}_t(x))_{k,j}+ 
(Db(0)-Db(\varphi^x_{T+t}))_{i,k} (\Psi_t^{-1}\Psi_s)_{k,j}
\right)\\
&\hspace{-1cm}=-2\sum\limits_{j=1}^{d}
\sum\limits_{i,k=1}^d (\Delta^s_t(x))_{i,j}
(Db(\varphi^x_{T+t}))_{i,k}(\Delta^{s}_t(x))_{k,j}\\
&+
2\sum\limits_{i,j,k=1}^{d}(\Delta^s_t)_{i,j}(Db(0)-Db(\varphi^x_{T+t}))_{i,k} (\Psi_t^{-1}\Psi_s)_{k,j}.
\end{split}
\]
By Hypothesis~\ref{hyp: potential} we obtain
\begin{equation}\label{qt0}
\begin{split}
\frac{\ud}{\ud t} |\Delta^{s}_t(x)|^2 
&\lqq -2\delta 
|\Delta^{s}_t(x)|^2+
2\sum\limits_{i,j,k=1}^{d}(\Delta^s_t(x))_{i,j}(Db(0)-Db(\varphi^x_{T+t}))_{i,k} (\Psi_t^{-1}\Psi_s)_{k,j}.
\end{split}
\end{equation}
The Young inequality yields
\begin{equation}\label{qt1}
\begin{split}
2\sum\limits_{i,j,k=1}^{d}|(\Delta^s_t(x))_{i,j}(Db(0)-Db(\varphi^x_{T+t}))_{i,k} (\Psi_t^{-1}\Psi_s)_{k,j}|&  \\
&\hspace{-5cm}\lqq \sum\limits_{i,j,k=1}^{d}
\left(\frac{\delta}{d}|(\Delta^s_t(x))_{i,j}|^2+ 
\frac{d}{\delta}|(Db(0)-Db(\varphi^x_{T+t}))_{i,k} (\Psi_t^{-1}\Psi_s)_{k,j}|^2\right) \\
&\hspace{-5cm}=\delta|\Delta^s_t(x)|^2+ 
\frac{d}{\delta}\sum\limits_{i,j,k=1}^{d}|(Db(0)-Db(\varphi^x_{T+t}))_{i,k} (\Psi_t^{-1}\Psi_s)_{k,j}|^2.
\end{split}
\end{equation}
Since $b\in \mathcal{C}^2$,  there exists a positive constant $C:=C(|x|)$ such that
\[
|Db(y)-Db(0)|\lqq C|y|\quad \textrm{ for any } y \textrm{ with } |y|\lqq |x|.
\]
By Hypothesis~\ref{hyp: potential} we observe that
$|\varphi^x_t| \lqq |x| e^{-\delta t}\lqq |x|$.
From Lemma~\ref{lem: cotainfsup}.ii) we have $|(\Psi_t^{-1}\Psi_s)|^2\lqq de^{-2\delta(t-s)} $. Then
\begin{equation}\label{qt2}
\begin{split}
\sum\limits_{i,j,k=1}^{d}|(Db(0)-Db(\varphi^x_{T+t}))_{i,k} (\Psi_t^{-1}\Psi_s)_{k,j}|^2
&\lqq C^2
\sum\limits_{i,j,k=1}^{d} |\varphi^x_{T+t}|^2 |(\Psi_t^{-1}\Psi_s)_{k,j}|^2\\
&\lqq 
C^2|\varphi^x_{T+t}|^2 d^2e^{-4\delta (t-s)}.
\end{split}
\end{equation}
Combining \eqref{qt0}, \eqref{qt1} and \eqref{qt2} we infer
\begin{equation*}
\begin{split}
\frac{\ud}{\ud t} |\Delta^{s}_t(x)|^2 
&\lqq -\delta 
|\Delta^{s}_t(x)|^2+
\frac{C^2 d^3}{\delta}|\varphi^x_{T+t}|^2 e^{-4\delta (t-s)}.
\end{split}
\end{equation*}
Since $\Delta^s_s(x)=0$, the preceding differential inequality with the help of Lemma~\ref{lem: gronwall} imply
\begin{equation*}
|\Delta^{s}_t(x)|^2\lqq 
\frac{C^2 d^3}{\delta}e^{-\delta(t-s)}\int_{s}^{t} |\varphi^x_{T+u}|^2 e^{-3\delta (u-s)}\ud u.
\end{equation*}
Observe that $|\varphi^x_{T+t}|=|\varphi^{\varphi^x_T}_t|\lqq e^{-\delta t}|\varphi^x_T|$. Then
\begin{equation*}
|\Delta^{s}_t(x)|^2\lqq 
\frac{C^2 d^3}{\delta}
|\varphi^x_{T}|^2
e^{-\delta(t-s)}\int_{s}^{t} e^{-\delta u} e^{-3\delta (u-s)}\ud u.
\end{equation*}
The integral version of the Gr\"onwall-Bellman lemma given in \cite{Mikami}, Lemma 1, yields    
\begin{equation*}
|\Delta^{s}_t(x)|^2\lqq 
\frac{C^2 d^3}{4\delta^2}
|\varphi^x_{T}|^2
e^{-\delta t}(1-e^{-4\delta (t-s)}).
\end{equation*}
\end{proof}

\bigskip

\section{\textbf{Freidlin-Wentzell first order approximation}}\label{ap: multiscale}

The result of this section yields a precise quantification of the inhomogeneous linearization error 
of $X^{\e, x}$ by $Y^{\e}_\cdot(x)$ given in \eqref{eq: Yxte} under the Hypothesis~\ref{hyp: potential} and \ref{hyp: moment condition} for any moment $\beta>0$. 
\begin{lem}[Quantitative first order expansion]\label{prop: secondorder}\hfill\\
Assume Hypothesis~\ref{hyp: potential} and \ref{hyp: moment condition} for some $\beta>0$.
For $(t_\e)_{\e>0}$ with
$t_\e \to \infty$ as $\e\to 0$ let the following limit hold true
\[
\lim_{\e \to 0} t_\e {\e}^{\frac{1}{1+ 2(\beta \wedge 1)}} = 0.
\]
Then for any 
$\alpha\in (0,2)$, $K>0$ and  $\Delta_\e = \e^{\frac{\al}{2}}$ 
there exist  positive constants 
$\e_0=\e_0(K,\alpha,\beta,\delta)$
 and 
$C=C(K,\alpha,\beta,\delta)$
such that for all $\e\in (0,\e_0]$ 
\[
\sup_{|x|\lqq K}
{\mathbb{P}(|X^{\e,x}_{t_\e}-Y^{\e}_{t_\e}(x)|\gqq \Delta^{1/\alpha}_\e \e)}\lqq C(K){\e^{\frac{\beta \wedge 1}{1+ 2 (\beta \wedge 1)}}}.
\]
\end{lem}
\begin{proof}
Let $t\gqq 0$.
Recall that $
Y^{\e}_t(x)=\varphi^x_t+\e Y^x_t$,
by \eqref{eq: Yxte},
where 
\[
\e Y^x_t = - \int_0^t Db(\varphi^x_s) \e Y^x_s \ud s  + \e \ud L_t.  
\]
That is $W^{\e}_t = \e Y^x_t$ satisfies 
\[
W^\e_t = - \int_0^t Db(\varphi^x_s) W^\e_s \ud s  + \e \ud L_t.  
\]
Hence 
\begin{align*}
&X^{\e, x}_t -Y^{\e}_t(x) = X^{\e, x}_t - \varphi^x_t - W^\e_t \\
&= \int_0^t \Big(-b(X^{\e, x}_s) - \big(-b(Y^{\e}_s(x))\big)\Big) \ud s 
+ \int_0^t \Big(-b(Y^{\e}_s(x)) - \big(-b(\varphi_s^x) - Db(\varphi^x_s) W^\e_s\big)\Big) \ud s.
\end{align*}
The chain rule yields 
\begin{align*}
&|X^{\e, x}_t -Y^{\e}_t(x)|^2 
= -2 \int_0^t \< b(X^{\e, x}_s) -b(Y^{\e}_s(x)), X^{\e, x}_s -Y^{\e}_s(x)\> \ud s \\
&\hspace{3cm} + 2 \int_0^t \< -b(Y^{\e}_s(x)) - \big(-b(\varphi_s^x) - Db(\varphi^x_s) \e Y_s\big), X^{\e, x}_s -Y^{\e}_s(x)\> \ud s. 
\end{align*}
By the mean value theorem, the Cauchy-Schwarz and the Young inequality we have
\begin{align*}
|-b(Y^{\e}_s(x)) - \big(-b(\varphi^x_s) - Db(\varphi^x_s) \e Y^x_s \big)|&=|-b(\varphi^x_s + \e Y^x_s) - \big(-b(\varphi^x_s) - Db(\varphi^x_s) \e Y^x_t\big)| \\
&\lqq \int_0^1\int_0^1 \|D^2 b(\varphi^x_s+ \theta_1\theta_2 \e Y^x_s) \|\ud \theta_1 \ud \theta_2 ~ |W^\e_s|^2.
\end{align*}
Together with Hypothesis \ref{hyp: potential} we obtain 
\begin{align*}
&|X^{\e, x}_t -Y^{\e}_t(x)|^2 \\
&\lqq  - \delta \int_0^t |X^{\e, x}_s -Y^{\e}_s(x)|^2 \ud s +  \frac{1}{\delta} \int_0^t \left(\int_0^1\int_0^1 \|D^2 b(\varphi^x_s+ \theta_1\theta_2 \e Y^x_s) \|\ud \theta_1 \ud \theta_2 ~ |W^\e_s|^2\right)^2 \ud s.
\end{align*}
Then the integral version of the Gr\"onwall-Bellman lemma given in \cite{Mikami}, Lemma 1, yields 
\begin{align*}
|X^{\e, x}_t -Y^{\e}_t(x)|^2 
&\lqq \frac{\e^{2-\theta}}{\delta} \int_0^t \left(\int_0^1\int_0^1 \|D^2 b(\varphi^x_s+ \theta_1\theta_2 \e Y^x_s) \|\ud \theta_1 \ud \theta_2 \right)^2 \e^{\theta}|Y^x_s|^2\ud s.
\end{align*}
Let $M>0$ and $\theta\in (0,1)$ and introduce
\[
A^{\e}_M:=\Big\{\e^{\theta}\sup_{0\lqq s\lqq t}|Y^1_s(x)|^2\lqq M\Big\}.
\]
For $\e\in (0,1]$ we have
\[
A^{\e}_M\subset \Big\{\sup_{0\lqq s\lqq t}|Y^1_s(x)|^2\lqq M\Big\}.
\]
Then on the event $A^{\e}_M$ it follows
\begin{align*}
|X^{\e, x}_t -Y^{\e}_t(x)|^2 
&\lqq \e^{2-\theta}C_M t.
\end{align*}
Observe that
\[
\mathbb{P}(|X^{\e,x}_{t_\e}-Y^\e_{t_\e}(x)|\gqq \Delta^{1/\alpha}_\e \e)\lqq 
\mathbb{P}(|X^{\e,x}_{t_\e}-Y^\e_{t_\e}(x)|\gqq \Delta^{1/\alpha}_\e \e,A^{\e}_M)
+\mathbb{P}((A^{\e}_M)^c).
\]
Then
\[
\mathbb{P}\Big(|X^{\e,x}_{t_\e}-Y^\e_{t_\e}(x)|\gqq \Delta^{1/\alpha}_\e \e,A^{\e}_M\Big)\lqq 
\mathbb{P}\Big(C_M\e^{2-\theta} t_\e \gqq \Delta^{1/\alpha}_\e \e,A^{\e}_M\Big)=
\mathbb{P}\Big(C_M t_\e \gqq \frac{\Delta^{1/\alpha}_\e}{\e^{1-\theta}},A^{\e}_M\Big).
\]
Choosing $\Delta_\e=\e^{\nicefrac{\alpha}{2}}$ we obtain
\[
\mathbb{P}\Big(|X^{\e,x}_{t_\e}-Y^\e_{t_\e}(x)|\gqq \Delta^{1/\alpha}_\e \e,A^{\e}_M\Big)\lqq 
\mathbb{P}\Big(C_M t_\e \gqq \frac{1}{\sqrt{\e^{1-\theta}}},A^{\e}_M\Big)=0\quad \textrm{ for } \e\ll 1.
\]
In the sequel, we estimate the term 
\[
\mathbb{P}\Big(\sup_{0\lqq s\lqq t_\e}|Y^1_s(x)|^2> \frac{M}{\e^\theta}\Big). 
\]
By Theorem~1 in \cite{SIORPAES} we have
\begin{align*}
\sup_{0\lqq s\lqq t_\e}|Y^1_s(x)|
&\lqq 6\sqrt{[Y^1_\cdot(x)]_{t_\e}} + 2\int_0^{t_\e}  H_{s-} \ud Y^1_s(x)\quad \textrm{ a.s.},
\end{align*}
where 
\begin{align*}
H_{s-} = \frac{Y^1_s(x)}{\sqrt{\sup_{0 \lqq s \lqq t_\e} |Y^1_s(x)|^2 + [Y^1_\cdot(x)]_{s-}}}. 
\end{align*}
In particular, we have 
\begin{align*}
&[Y^1_\cdot(x)]_{t} = [L]_{t} = \int_0^t \int_{|z|\lqq 1} |z|^2 N(\ud s\ud z)\quad \textrm{such that }\\
&\int_0^{t}  H_{s-} \ud Y^1_s(x) = \int_0^t \< H_{s-}, - Db(\varphi^x_s) Y^1_s(x) \> \ud s + 
\int_0^t \int_{|z|\lqq 1}\< H_{s-}, z\> \ti N(\ud s\ud z) \\
&\hspace{3cm}+ 
\int_0^t \int_{|z|> 1}\< H_{s-}, z\> N(\ud s\ud z).
\end{align*}
We apply Hypothesis~\ref{hyp: potential} and obtain a.s. 
\[
\int_0^t \< H_{s-}, - Db(\varphi^x_s) Y^1_s(x) \> \ud s \lqq 0.  
\]
Hence 
\begin{align*}
&\PP\Big( \sup_{0\lqq s\lqq t_\e}|Y^1_s(x)| > \frac{M}{\e^\theta}\Big)\\ 
&\lqq \PP\Big(6 \Big(\int_0^{t_\e} \int\limits_{|z|\lqq 1} |z|^2 N(\ud s\ud z)\Big)^{\nicefrac{1}{2}}  +  \int_0^{t_\e} \int\limits_{|z|\lqq 1}\< H_{s-}, z\> \ti N(\ud s\ud z) + 
\int_0^{t_\e} \int\limits_{|z|> 1}\< H_{s-}, z\> N(\ud s\ud z)  > \frac{M}{\e^\theta}\Big) \\
&\lqq \PP\Big(\int_0^{t_\e} \int_{|z|\lqq 1} |z|^2 N(\ud s\ud z) > \frac{1}{18^2} \frac{M^2}{\e^{2\theta}}\Big)  + \PP\Big(\int_0^{t_\e} \int_{|z|\lqq 1}\< H_{s-}, z\> \ti N(\ud s\ud z) > \frac{1}{3} \frac{M}{\e^\theta}\Big) \\
&\qquad + \PP\Big(\int_0^{t_\e} \int_{|z|> 1}\< H_{s-}, z\> N(\ud s\ud z) > \frac{1}{3} \frac{M}{\e^\theta}\Big).
\end{align*}
We continue term by term. First we obtain  
\begin{align}
\PP\Big(\int_0^{t_\e} \int_{|z|\lqq 1} |z|^2 N(\ud s\ud z) > (\frac{1}{18})^2 \frac{M^2}{\e^{2\theta}}\Big) 
&\lqq t_\e \e^{2\theta} \frac{(18)^2 }{M^2} \int_{|z|\lqq 1} |z|^2 \nu(\ud z) = C~t_\e \e^{2\theta}.\label{e: term1}
\end{align}
By the 
Markov inequality we bound the second term and obtain
\begin{align}
\PP\Big(\int_0^{t_\e} \int_{|z|\lqq 1}\< H_{s-}, z\> \ti N(\ud s\ud z) > \frac{1}{3} \frac{M}{\e^\theta}\Big) 
&\lqq \e^{2\theta} \Big(\frac{3}{M}\Big)^2 \EE\Big[\Big(\int_0^{t_\e} \int_{|z|\lqq 1}\< H_{s-}, z\> \ti N(\ud s\ud z\Big)^2\Big]\nonumber\\
&= \e^{2\theta} \Big(\frac{3}{M}\Big)^2 \EE\Big[\int_0^{t_\e} \int_{|z|\lqq 1}\< H_{s-}, z\>^2 \nu(\ud z)\ud s\Big]
\nonumber\\
&= \e^{2\theta} t_\e \Big(\frac{3}{M}\Big)^2 \int_{|z|\lqq 1} |z|^2 \nu(\ud z).\label{e: term2}
\end{align}
Finally, 
\begin{align}
\PP\Big(\int_0^{t_\e} \int_{|z|> 1}\< H_{s-}, z\> N(\ud s\ud z) > \frac{1}{3} \frac{M}{\e^\theta}\Big) 
&\lqq \PP\Big(\int_0^{t_\e} \int_{|z|> 1} |z| N(\ud s\ud z) > \frac{1}{3} \frac{M}{\e^\theta}\Big) \nonumber\\
&\lqq \e^{(\beta\wedge 1)\theta} \Big(\frac{3}{M}\Big)^{(\beta\wedge 1)}
\EE\Big[\Big(\int_0^{t_\e} \int_{|z|> 1} |z| N(\ud s\ud z)\Big)^{(\beta\wedge 1)}\Big] \nonumber\\
&\lqq \e^{(\beta\wedge 1)\theta} \Big(\frac{3}{M}\Big)^{(\beta\wedge 1)}
\EE\Big[\int_0^{t_\e} \int_{|z|> 1} |z|^{(\beta\wedge 1)} N(\ud s\ud z)\Big] \label{e: Loubert Bie}\\
&= \e^{(\beta\wedge 1)\theta}  t_\e \Big(\frac{3}{M})^{(\beta\wedge 1\Big)} \int_{|z|> 1} |z|^{(\beta\wedge 1)} \nu(\ud z), \label{e: term3}
\end{align}
where we have used the subadditivity of the power $\beta \wedge 1$ in the sense of 
Subsection 1.1.2, see formula (1.6) in \cite{Bie}. 
Optimizing over $\theta$ we obtain $\theta = \frac{1}{1 + 2(\beta \wedge 1)}$. 
\end{proof}

\bigskip 

\section{\textbf{The linear inhomogeneous dynamics $Y_\cdot^{\e}(x)$}}\label{Appendix C}

This section gathers properties of the inhomogeneous first order expansion $Y_\cdot^{\e}(x)$ of $X^{\e, x}$ mainly with the help of Fourier techniques. 

\bigskip 

\subsection{\textbf{$\beta$-H\"older continuity of the characteristic exponent of a L\'evy process}}\label{Appendix C.1}\hfill\\

\noindent It is classical that $\beta\gqq 1$ in Hypothesis \ref{hyp: moment condition} implies that 
the characteristic function is continuously differentiable, and hence locally Lipschitz 
continuous. This remains valid for the characteristic exponent $\psi$.
In the sequel, we provide an elementary proof for the respective fractional case $\beta\in (0,1)$.
\begin{prop}[Local H\"older continuity of the characteristic exponent]\label{lem:Holdercontinuity}\hfill\\
Let $L=(L_t)_{t\gqq 0}$ be a L\'evy process on $\mathbb{R}^d$. 
Denote by $\psi$ its characteristic exponent and by $\nu$ its L\'evy measure. Assume that
\[
\int_{|z|\gqq 1} |z|^\beta\nu(\ud z)<\infty\quad \textrm{ for some } \beta>0.
\]
Then we have the following.
\begin{enumerate}
\item If $\beta\gqq 1$, $\psi$ is $\cC^1$. In particular, it is Lipschitz continuous.
\item If $\beta \in (0,1)$, $\psi$ is locally H\"older continuous with H\"older index $\beta$.
\end{enumerate}
\end{prop}
\begin{proof}
The proof of item (1) is given in
Theorem~15.32 of \cite{Klenke}. 
We continue with the proof of~ (2).
Assume that $\beta\in (0,1)$. We prove that $\psi$ is locally H\"older continuous.
Recall that
\[
\psi(z)=\int_{\RR^d}\big(e^{i\<u,z\>}-1-i\<u,z\>\ind_{\{|u|\lqq 1\}}(u)\big)\nu(\ud u), \quad z\in \RR^d.
\]
For any $z\in \RR^d$, let 
\[
f_1(z)=\int_{|u|\lqq 1}\big(e^{i\<u,z\>}-1-i\<u,z\>\big)\nu(\ud u)
\quad \textrm{ and } \quad f_2(z)=\int_{|u|>1}\big(e^{i\<u,z\>}-1\big)\nu(\ud u). 
\]
First, we analyze $f_2$. Let $z_1,z_2\in \RR^d$.
Notice that
\[
\begin{split}
&|f_2(z_1)-f_2(z_2)| \lqq \int_{|u|>1}|e^{i\<u,z_1\>}-e^{i\<u,z_2\>}|\nu(\ud u)\\
&\quad  =\int_{|u|>1}|e^{i\<u,z_1-z_2\>}-1|\nu(\ud u)=
\sqrt{2}\int_{|u|>1}\sqrt{1-\cos(\<u,z_1-z_2\>)}\nu(\ud u)\\
&\quad = 2\int_{|u|>1}\left|\sin\left(\frac{\<u,z_1-z_2\>}{2}\right)\right|\nu(\ud u)\\
& \quad =2\int\limits_{|u|>1,|\<u,z_1-z_2\>|>1}\left|\sin\left(\frac{\<u,z_1-z_2\>}{2}\right)\right|\nu(\ud u)
+2\int\limits_{|u|>1,|\<u,z_1-z_2\>|\lqq 1}\left|\sin\left(\frac{\<u,z_1-z_2\>}{2}\right)\right|\nu(\ud u)\\
& \quad \lqq 2\int\limits_{|u|>1,|\<u,z_1-z_2\>|>1} |\<u,z_1-z_2\>|^{\beta}\nu(\ud u)
+2\int\limits_{|u|>1,|\<u,z_1-z_2\>|\lqq 1}\left|\sin\left(\frac{\<u,z_1-z_2\>}{2}\right)\right|\nu(\ud u)\\
& \quad \lqq 2|z_1-z_2|^{\beta}\int\limits_{|u|>1} |u|^{\beta}\nu(\ud u)
+2\int\limits_{|u|>1,|\<u,z_1-z_2\>|\lqq 1}\left|\sin\left(\frac{\<u,z_1-z_2\>}{2}\right)\right|\nu(\ud u)\\
& \quad =C_\beta|z_1-z_2|^{\beta}
+2\int\limits_{|u|>1,|\<u,z_1-z_2\>|\lqq 1}\left|\sin\left(\frac{\<u,z_1-z_2\>}{2}\right)\right|\nu(\ud u),\\
\end{split}
\]
where $C_\beta=2\int_{|u|>1} |u|^{\beta}\nu(\ud u)<\infty$.
Let $\tilde{C}_\beta=\sup\left\{\frac{|\sin(\nicefrac{\theta}{2})|}{|\theta|^{\beta}}: |\theta|\lqq 1\right\}$.
Since $\beta\in (0,1)$ we have $\tilde{C}_\beta<\infty$. Indeed, notice that
 $\lim\limits_{\theta\to 0}\frac{|\sin(\nicefrac{\theta}{2})|}{|\theta|^{\beta}}=0$ then $C<\infty$. 
 Furthermore, 
 \[
 \begin{split}
 \int\limits_{|u|>1,|\<u,z_1-z_2\>|\lqq 1}\left|\sin\left(\frac{\<u,z_1-z_2\>}{2}\right)\right|\nu(\ud u) 
 &\lqq 
 \int\limits_{|u|>1,|\<u,z_1-z_2\>|\lqq 1}\tilde{C}_\beta |\<u,z_1-z_2\>|^\beta \nu(\ud u)\\
& \lqq \tilde{C}_\beta |z_1-z_2|^\beta \int_{|u|>1}|u|^\beta \nu(\ud u).
 \end{split}
 \]
Hence, $|f_2(z_2)-f_2(z_1)|\lqq C(\beta)|z_2-z_1|^\beta$ for any $z_1,z_2\in \mathbb{R}^d$.
In the sequel, we  analyze $f_1$.  We calculate for $z_1,z_2\in \RR^d$
\[
\begin{split}
|f_1(z_1)-f_1(z_2)|&\lqq \int_{|u|\lqq 1}\left| e^{i\<u,z_1\>}-e^{i\<u,z_2\>}-i\<u,z_1-z_2\>\right|\nu(\ud u)\\
&\lqq  C|z_1-z_2|\int_{|u|\lqq 1}|u|^2 \nu(\ud u)
=C_1|z_1-z_2|,
\end{split}
\]
where we have used the mean value theorem for the integrand 
\[
\begin{split}
\left| e^{i\<u,z_1\>}-e^{i\<u,z_2\>}-i\<u,z_1-z_2\>\right|&=
\left| \int_{0}^{1} e^{i\<u,z_1+\theta(z_2-z_1)\>}\<u,z_1-z_2\>\ud \theta-\<u,z_1-z_2\>\right|\\
&=\int_{0}^{1} |e^{i\<u,z_1+\theta(z_2-z_1)\>}-1 |\ud \theta |\<u,z_1-z_2\>|
\end{split}
\]
with $C_1=C\int_{|u|\lqq 1}|u|^2 \nu(\ud u)<\infty$.
If $|z_1|\lqq \frac{1}{2}$ and $|z_2|\lqq  \frac{1}{2}$, then
$|z_1-z_2|\lqq |z_1-z_2|^\beta$. This concludes the proof of (2).
\end{proof}
\begin{rem}
\hfill
\begin{enumerate}
\item Note that the above calculations for $f_1$ give an elementary proof 
of the fact that any pure jump L\'evy process with uniformly bounded jumps-sizes
has a globally Lipschitz continuous characteristic exponent $\psi$.
\item The calculations for $f_2$ yield that any compound Poisson process with $\beta$-integrability
\[\int_{|z|>1} |z|^{\beta}\nu(\ud z)<\infty \quad\textrm{ for some } \beta>0\]
 has a locally H\"older continuous characteristic exponent $\psi$ with H\"older index $\beta$. This extends the well-known result that the existence of integer moments translates to the respective order of differentiability of the characteristic function to the case of fractional moments.\\
\end{enumerate}
\end{rem}

\bigskip 

\subsection{\textbf{Continuous dependence of the total variation in the nonlinearity}}\hfill

\begin{lem}[Continuous dependence on the initial value]\label{lem: continuidad}\hfill\\
Let $t>0$, $x\in \RR^d$ and denote by $(Y^{x}_t)_{t\gqq 0}$ the unique strong solution of 
\eqref{eq: Yxt} as well as by $g^{x}_t$ the respective density of $Y^{x}_t$. 

Then $x\mapsto g^{x}_t(u)$ is continuous for any fixed $t>0$ and $u\in \RR^d$.
In addition, the map $x\mapsto \norm{Y^{x}_t-U}$ is continuous for any fixed $t>0$ and 
$U$ any random vector on $\RR^d$. 
\end{lem}
\begin{proof}
Let $x,x'\in \RR^d$. The Fourier inversion formula yields
\begin{align}\label{fourier}
g^{x}_t(u)-g^{x'}_t(u)&=C_\pi \int_{\RR^d} e^{i\<u,\theta\>}\left(\hat{f}^{x}_t(\theta)-\hat{f}^{x'}_t(\theta)\right)\ud \theta\nonumber\\
&=C_\pi \int_{|\theta|\lqq K} e^{i\<u,\theta\>}\left(\hat{f}^{x}_t(\theta)-\hat{f}^{x'}_t(\theta)\right)\ud \theta+
C_\pi \int_{|\theta|>K} e^{i\<u,\theta\>}\left(\hat{f}^{x}_t(\theta)-\hat{f}^{x'}_t(\theta)\right)\ud \theta
\end{align}
for any $u\in \RR^d$. We start with the first term of the right-hand side of the preceding inequality. Recall 
\[
\hat{f}^{x}_t(\theta)=\exp\left(\int_{0}^{t}\psi(\Phi^*_s(x)(\Phi^*_t)^{-1}(x)\theta)\ud s\right)
\quad \mbox{and}\quad \hat{f}^{x'}_t(\theta)=\exp\left(\int_{0}^{t}\psi(\Phi^*_s(x')(\Phi^*_t)^{-1}(x')\theta)\ud s\right).
\]
For any $|\theta|\lqq K$ we have
\[
\sup_{0\lqq s\lqq t}|\psi(\Phi^{*}_s(x')(\Phi^{*}_t)^{-1}(x')\theta)-\psi(\Phi^{*}_s(x)(\Phi^{*}_t)^{-1}(x)\theta)|\to 0,\quad\textrm{ as } x'\to x.
\]
Indeed, by Proposition~\ref{lem:Holdercontinuity}, Item (2), in Appendix~\ref{Appendix C} there exists a positive constant $C_K$ such that
for all $|\theta|\leq K$, $x,x'\in \RR^d$ we have 
\begin{align*}
|\psi(\Phi^{*}_s(x')(\Phi^{*}_t)^{-1}(x')\theta)&-\psi(\Phi^{*}_s(x)(\Phi^{*}_t)^{-1}(x)\theta)|\\[2mm]
&
\lqq C_K|\theta|^{1\wedge \beta}|\Phi^{*}_s(x')(\Phi^{*}_t)^{-1}(x')-\Phi^{*}_s(x)(\Phi^{*}_t)^{-1}(x)|^{1\wedge \beta}\\[2mm]
&\lqq C_K|\theta|^{1\wedge \beta}(|\Phi^{*}_s(x')(\Phi^{*}_t)^{-1}(x')|^{1\wedge \beta}+|\Phi^{*}_s(x)(\Phi^{*}_t)^{-1}(x)|^{1\wedge \beta})\\[2mm]
&\lqq 2C_k|\theta|^{1\wedge \beta}(\sqrt{d}e^{-\delta(t-s)})^{1\wedge \beta}.
\end{align*}
Then the dominated convergence theorem in the exponent yields for any $|\theta|\lqq K$
\[
\hat{f}^{x'}_t(\theta)\to \hat{f}^{x}_t(\theta),\quad \textrm{ as } x'\to x.
\]
Again, by dominated convergence we have 
\[
\int_{|\theta|\lqq K} e^{i\<u,\theta\>}\left(\hat{f}^{x}_t(\theta)-\hat{f}^{x'}_t(\theta)\right)\ud \theta\to 0,\quad \textrm{ as } x'\to x.
\]
We continue with the second term of the right-hand side.
Let $|\theta|>K$ and we assume that  $|x'|\lqq r$ and $|x|\lqq r$ where $r=2|x|$.
We analyze
\begin{align*}
|\hat{f}^{x}_t(\theta)| &=\exp\left( \int_{0}^{t} \int_{\RR^d}(\cos(\<\Phi^*_s(x)(\Phi^*_t)^{-1}(x)\theta,w\>)-1)\nu(\ud w)\ud s\right)\\
&\lqq \exp\left( \int_{0}^{t} \int_{|\<\Phi^*_s(x)(\Phi^*_t)^{-1}(x)\theta,w\>|\lqq \pi}(\cos(\<\Phi^*_s(x)(\Phi^*_t)^{-1}(x)\theta,w\>)-1)\nu(\ud w)\ud s\right)\\
&\lqq \exp\left(-\frac{2}{\pi^2} \int_{0}^{t} \int_{|\<\Phi^*_s(x)(\Phi^*_t)^{-1}(x)\theta,w\>|\lqq \pi}
|\<\Phi^*_s(x)(\Phi^*_t)^{-1}(x)\theta,w\>|^2 \nu(\ud w)\ud s\right).
\end{align*}
By Lemma~\ref{lem: cotainfsup} we have
\[
|\Phi^*_s(x)(\Phi^*_t)^{-1}(x)\theta|\gqq c_1e^{-c_2(r)(t-s)}|\theta|\quad \textrm{ for any }
s\in [0,t], |x|\lqq r, \theta\in \RR^d.
\]
Note that 
\[
c_1e^{-c_2(r)(t-s)}|\theta|\gqq c_1 Ke^{-c_2(r)t}.
\]
Since $t>0$ is fixed. 
The choice  $K>\frac{e^{c_2(r)t}C_\sphericalangle}{c_1}$ yields $c_1e^{-c_2(r)(t-s)}|\theta|>C_\sphericalangle$ where $C_\sphericalangle$ is the constant that appears in Lemma~\ref{lem oreymasuda}.
Then we have for $|\theta|>K$
\begin{align*}
|\hat{f}^{x}_t(\theta)|&\lqq \exp\left(-\frac{2}{\pi^2}c_\sphericalangle\int_{0}^{t} |\Phi^*_s(x)(\Phi^*_t)^{-1}(x)\theta|^{\alpha} \ud s\right)
\lqq \exp\left(-\frac{2}{\pi^2}c_\sphericalangle|\theta|^\alpha \int_{0}^{t} c^{\alpha}_1 e^{-c_2(r)s\alpha} \ud s\right).
\end{align*}
Then
\begin{align*}
\int_{|\theta|>K}|\hat{f}^{x'}_t(\theta)-\hat{f}^{x}_t(\theta)|\ud \theta
&\lqq \int_{|\theta|>K}|\hat{f}^{x'}_t(\theta)|\ud \theta+\int_{|\theta|>K}|\hat{f}^{x}_t(\theta)|\ud \theta\\
&\lqq 2\int_{|\theta|>K}\exp\left(-\frac{2}{\pi^2}c_\sphericalangle|\theta|^\alpha \int_{0}^{t} c^{\alpha}_1 e^{-c_2(r)s\alpha} \ud s\right)\ud \theta<\infty.
\end{align*}
Sending $x'\to x$ and subsequently $K \to \infty$ we obtain
\[
\lim\limits_{K\to \infty}\limsup\limits_{x'\to x}\int_{|\theta|>K}|\hat{f}^{x'}_t(\theta)-\hat{f}^{x}_t(\theta)|\ud \theta=0.
\]
By \eqref{fourier} we obtain
\[
\lim\limits_{x'\to x}\int_{\RR^d}|\hat{f}^{x'}_t(\theta)-\hat{f}^{x}_t(\theta)|\ud \theta=0.
\]
The preceding limit yields that $x\in \RR^d\mapsto g^{x}_t(u)\in [0,\infty)$ is continuous for any $t>0$ and $u\in \mathbb{R}^d$ fixed. 
This proves the first part of the statement.

\noindent We show the second part of the statement. The Scheff\'e lemma applied to the densities $g^x_t$, $g^{x'}_t$   
 implies for any $t>0$
\[
\norm{Y^{x'}_t-Y^{x}_t}\to 0,\quad \textrm{ as } x'\to x.
\]
The triangle inequality yields for any random vector $U$ on $\RR^d$ 
\[
\left|\norm{Y^{x'}_t-U}-\Big\lVert
Y^{x}_t-U
\Big\rVert_{\mathrm{TV}}\right| \lqq \norm{Y^{x'}_t-Y^{x}_t}
\]
for any $t>0$ and $x,x'\in \RR^d$. 
Combining both preceding expressions finishes the proof.
\end{proof}

\bigskip 

\subsection{\textbf{Ergodicity of the inhomogeneous O-U process $Y^\e_t(x)$}}
\subsubsection{\textbf{Existence of the limiting distribution $\mu^{\e}_*$ and its convergence in law}}\hfill
\begin{lem}\label{lem:convergenceindist} \hfill
\begin{enumerate}
 \item For any $\e\in (0, 1)$ and $x\in \RR^d$ we have that 
$Y^{\e}_t(x)$ converges in distribution to $\mu^{\e}_*$ as $t\ra \infty$, where 
$\e Z_\infty$ has the law of $\mu^{\e}_*$.
 \item For any $\e\in (0,1]$, $K>0$, $t\gqq 0$ and $(x_{\e, t})_{\e, t}$ 
with $|x_{\e, t}|\lqq K$ 
we have that $Y^{\e}_{t}(x_{\e, t})$ converges in distribution to $\mu^{\e}_*$ as $t \ra \infty$, where 
$\e Z_\infty$ has the law of $\mu^{\e}_*$.
\end{enumerate}
\end{lem}
\noindent
\textbf{Proof.} We start with the proof of (1). 
 Let $x\in \RR^d$ and $\e\in (0,1]$ be fixed.
Recall that $Y^{\e}_t(x)=\varphi^x_t+\e Y^x_t$ for any $t\gqq 0$
by \eqref{eq: Yxte}, where $Y^x=(Y^x_t)_{t\gqq 0}$
is the solution 
of the stochastic differential equation
\begin{align*}
\ud Y^x_t =-Db(\varphi^x_t)Y^x_t \ud t + \ud L_t \quad \textrm{ with } Y_0=0
\end{align*}
and $(\varphi^x_t)_{t\gqq 0}$ is the solution of \eqref{dde1}.
By the variation of constants formula, it is not hard to see that 
\begin{align*}
Y^x_t = \Phi^{-1}_t(x) \int_0^t \Phi_s(x) \ud L_s\quad 
\textrm{ for } t\gqq 0,
\end{align*}
where $\Phi(x):=(\Phi_t(x))_{t\gqq 0}$ is the solution of the matrix differential equation 
\[
\frac{\ud}{\ud t} \Phi_t(x) = \Phi_t(x)\, Db(\varphi_t^x)\quad \textrm{ with }  \Phi_0 = I_d.   
\]
In addition, examining the Wronskian at $0$ we have that $\det \Phi_t(x) \neq 0$ for all $t\gqq 0$. 
The inverse matrix $\Phi_t^{-1}(x)$ exists for any $t\gqq 0$ and 
$\Phi^{-1}(x):=(\Phi^{-1}_t(x))_{t\gqq 0}$ is the solution of the matrix differential equation
\[
\frac{\ud}{\ud t} \Phi_t^{-1}(x) =  - Db(\varphi_t^x)\, \Phi_t^{-1}(x)\quad \textrm{ with }  \Phi_0 = I_d.   
\]
Recall that $Z=(Z_t)_{t\gqq 0}$ is the solution of
\begin{equation*}\label{def:Z}
\ud Z_t = - Db(0) Z_t \ud t +\ud L_t\quad \textrm{ with } Z_0= 0.
\end{equation*}
Since this equation is also linear, 
its solution $Z$ is also given explicitly by the 
variation-of-constants formula 
\begin{align*}
Z_t = \Psi^{-1}_t \int_0^t \Psi_s \ud L_s\quad 
\textrm{ for any  } t\gqq 0,
\end{align*}
where $\Psi:=(\Psi_t)_{t\gqq 0}$ is the solution of the matrix differential equation 
\[
\frac{\ud}{\ud t} \Psi_t = \Psi_t \,Db(0)\quad\textrm{ with }  \Psi_0 = I_d.   
\]
Note that the inverse matrix $\Psi_t^{-1}$ exists for any $t\gqq 0$ and $\Psi^{-1}:=(\Psi^{-1}_t)_{t\gqq 0}$ satisfies the matrix differential equation
\[
\frac{\ud}{\ud t} \Psi_t^{-1} =  - Db(0)\, \Psi_t^{-1}\quad \textrm{ with }  \Psi_0 = I_d.   
\]
Since we are interested in convergence in distribution, we analyze the characteristic function of $Y^x_t$ and $Z_t$ for $t>0$.
By Theorem~3.1 in \cite{SaYa} we know that
\begin{equation}\label{eq:cfou}
\mathbb{E}\left[e^{i\<z,Z_t\>}\right]=\exp\left(\int_{0}^{t}\psi(\Psi^*_s(\Psi^{-1}_t)^*z)\ud s\right)\quad  \textrm{ for } z\in \RR^d,
\end{equation}
where $\psi:\RR^d\to \CC$ is the characteristic exponent of the L\'evy process $L$.
For the inhomogeneous process $(Y^x_t)_{t\gqq 0}$, 
a standard discretization procedure combined with \eqref{eq:cfou} yields
\[
\mathbb{E}\left[e^{i\<z,Y^x_t\>}\right]=\exp\left(\int_{0}^{t}\psi(\Phi^*_s(x)(\Phi^{-1}_t(x))^*z)\ud s\right )\quad \textrm{ for } z\in \RR^d.
\]
By Lemma~\ref{lem: cotainfsup} part iii) we note
that there exist uniform positive 
constants $c_3$ and $c_4$ such that
\begin{equation*}
|\Phi^*_s(x)(\Phi^{-1}_t(x))^*z|\lqq c_3e^{-c_4 (t-s)}|z|
\qquad \textrm{ and } \qquad
|\Psi^*_s(\Psi^{-1}_t)^*z|\lqq c_3e^{-c_4 (t-s)}|z| 
\end{equation*}
for any $t\gqq 0$, $s\in [0,t]$ and $z\in \RR^d$.
For $t>0$ and $z\in \RR^d$, we define the error term by
\[
\Theta^z_t(x):=\int_{0}^{t}\psi(\Phi^*_s(x)(\Phi^{-1}_t(x))^*z)\ud s-\int_{0}^{t}\psi(\Psi^*_s(\Psi^{-1}_t)^*z)\ud s.
\]
Since we are assuming that the L\'evy process $L=(L_t)_{t\gqq 0}$ has $\beta$-moment for some $\beta>0$ (see Hypothesis~\ref{hyp: moment condition}), the characteristic exponent $\psi$ is  differentiable for $\beta\gqq 1$ and locally H\"older continuous with index $\beta$ for $\beta\in (0,1)$, a proof is given in  Proposition~\ref{lem:Holdercontinuity} in Appendix~\ref{Appendix C}.
Let $z\in \RR^d$ with $|z|\lqq \frac{1}{2c_3}$. Then there exists a positive constant $C_1:=C(c_3,\beta)$  such that
\begin{equation}\label{comparacion}
\begin{split}
|\Theta^z_t(x)| 
&\lqq \int_{0}^{t}|\psi(\Phi^*_s(x)(\Phi^{-1}_t(x))^*z)-\psi(\Psi^*_s(\Psi^{-1}_t)^*z)|\ud s\\
&\lqq C_1\int_{0}^{t}|\Phi^*_s(x)(\Phi^{-1}_t(x))^*-\Psi^*_s(\Psi^{-1}_t)^*|^{\beta\wedge 1}\ud s \quad \textrm{ for } t\gqq 0.
\end{split}
\end{equation}
By Lemma~\ref{lem: cotainfsup}, part v), there is a positive constant $C = C(|x|)$ such that
\begin{align*}
|\Phi^*_s(x)(\Phi^{-1}_t(x))^* -(\Psi_s)^*(\Psi^{-1}_t)^* |^2 &\lqq 
\frac{C^2(|x|) d^3}{4\delta^2}
|\varphi^x_{0}|^2
e^{-\delta t}(1-e^{-4\delta (t-s)})\lqq C_2(|x|)e^{-\delta t}
\end{align*}
for any $s\in [0,t]$, where $C_2(|x|)$ is a constant that depends continuously on $|x|$.
Using then the preceding inequality in \eqref{comparacion} we obtain
\begin{equation}\label{e: charexpcontrol}
|\Theta^z_t(x)|\lqq C_3(|x|)e^{-\frac{\delta}{2} (\beta\wedge 1) t}t
\quad \textrm{ for } t\gqq 0,
\end{equation}
where $C_3(|x|)$ is a constant that depends continuously on $|x|$.
Sending $t\to \infty$, we obtain $\Theta^z_t(x)\to 0$ for any 
$|z|\lqq \frac{1}{2c_3}$. 
\noindent
In the sequel, we prove
$Y_t^x  \stackrel{d}{\lra} Z_\infty$. 
By Theorem~4.1  in \cite{SaYa} we know that 
$Z_t  \stackrel{d}{\lra} Z_\infty$, that is,
\begin{equation}\label{limitee}
\lim\limits_{t\to\infty}
\mathbb{E}\left[e^{i\<z,Z_t\>}\right]=
\exp\Big(\int_{0}^{\infty}\psi(e^{-Db(0)s}z)\ud s\Big)=:\chi(z)
\quad \textrm{ for } z\in \RR^d.
\end{equation}
Recall that for each $t>0$, $Z_t$ is infinitely divisible (see for instance Theorem~9.1 in \cite{SA}), then $\mathbb{E}[e^{i\<z,Z_t\>}]\not=0$ for any $z\in \RR^d$ (see Lemma~7.5 in \cite{SA}). 
Hence, \eqref{e: charexpcontrol} implies
\begin{equation}\label{e: charfuncfrac}
\lim\limits_{t\to \infty}\frac{\mathbb{E}\left[e^{i\<z,Y^x_t\>}\right]}{\mathbb{E}\left[e^{i\<z,Z_t\>}\right]}=\lim\limits_{t\to \infty}\exp\left(\Theta^z_t(x)\right)=1 \quad \textrm{ for } |z|\lqq \frac{1}{2c_3}.
\end{equation}
By \eqref{limitee} we infer
$\lim\limits_{t\to \infty}\mathbb{E}\left[e^{i\<z,Y^x_t\>}\right]=\chi(z)$ for  $|z|\lqq \frac{1}{2c_3}$.
Since $\chi$ is a characteristic function, it is uniquely determined by its values in an open neighborhood of the origin. As a result we obtain
$\lim\limits_{t\to \infty}\mathbb{E}\left[e^{i\<z,Y^x_t\>}\right]=\chi(z)$  for any $z\in \RR^d$.
By the L\'evy continuity theorem we obtain $Y^x_t  \stackrel{d}{\lra} Z_\infty$. 
Recall that $Y^{\e}_t(x)=\varphi^x_t+\e Y^x_t$, $t\gqq 0$.
Since $\varphi^x_t \to 0$, as $t\to \infty$, the Slutsky lemma yields
$Y^{\e}_t(x)  \stackrel{d}{\lra} \e Z_\infty$
as $t \to \infty$. This finishes the proof of (1). 

We finish with the proof of (2). By \eqref{e: charexpcontrol} we have 
that the convergence \eqref{e: charfuncfrac} only depends of $x$ via $|x|\lqq K$ 
and consequently is valid for all $|x_{t, \e}|\lqq K$ as stated in (2). 

\bigskip 

\subsubsection{\textbf{Convergence of $Y^\e_t(x)$ to $\mu^{\e}_*$ in the total variation distance}}\label{ss: ERGINHOMOUTV}\hfill\\
\begin{lem} \label{lem: cvtlineal}
For any $K>0$ and $\e>0$ we have 
\begin{equation}\label{e: cvtlineal}
\lim\limits_{t\to \infty}\sup_{|x|\lqq K} \norm{Y^{\e}_t(x) - \mu^{\e}_*}= 0, 
\end{equation}
where $\e Z_\infty$ has the law of $\mu^{\e}_*$.
In particular, 
\begin{equation}\label{eq: tvd22}
\lim\limits_{t\to \infty}\sup_{|x|\lqq K} \norm{Y^{x}_t - Z_\infty}=0.
\end{equation}
\end{lem}

\begin{proof}The idea is to show that convergence in distribution
 (Lemma~\ref{lem:convergenceindist} (2)) combined with 
 the Orey-Masuda cone condition (Lemma~\ref{lem oreymasuda}) 
 implies $\lim\limits_{t\to \infty}\sup_{|x|\lqq K}\norm{Y^{\e}_t(x) - \mu^{\e}_*}=0$.
It is enough to prove that
$Z_\infty$ has a continuous density and 
$\lim\limits_{t\to \infty}\sup_{|x|\lqq K}\norm{Y_t^x - Z_\infty}= 0$.
The latter implies $\lim\limits_{t\to \infty}\sup_{|x|\lqq K}\norm{Y^{\e}_t(x) - \mu^{\e}_*}=0$. Let $|x|\lqq K$. Indeed, 
\begin{align}\label{ine2}
\norm{Y^{\e}_t(x)-\mu^{\e}_*}
& \lqq \norm{(\varphi^x_t+\e Y_t^x)-(\varphi^x_t+\e Z_\infty)}+
\norm{(\varphi^x_t+\e Z_\infty)-(\e Z_\infty)} \nonumber \\ 
&=\norm{Y_t^x-Z_\infty}+
\norm{(\nicefrac{\varphi^x_t}{\e}+ Z_\infty)-Z_\infty}\nonumber\\
&\lqq \sup_{|x|\lqq K}\norm{Y_t^x-Z_\infty}+
\sup_{|x|\lqq K} \norm{(\nicefrac{\varphi^x_t}{\e}+ Z_\infty)-Z_\infty}.
\end{align}
We start with the second term in \eqref{ine2}. 
Since $Z_\infty$ has a continuous density (see for instance, Case~3 in Section~4 in \cite{BP}),
the Scheff\'e lemma yields 
\[
\sup_{|x|\lqq K} \norm{(\nicefrac{\varphi^x_t}{\e}+ Z_\infty)-Z_\infty} 
= \norm{(\nicefrac{\varphi^{\ti x}_t}{\e}+ Z_\infty)-Z_\infty}
\]
for some $|\ti x| \lqq K$. By Hypothesis~\ref{hyp: potential} 
we have 
$
|\varphi^x_t| \lqq e^{-\delta t} |x| \lqq e^{-\delta t} K  
$
whenever $|x|\lqq K$. Again by the Scheff\'e lemma we deduce 
\[
\norm{(\nicefrac{\varphi^{\ti x}_t}{\e}+ Z_\infty)-Z_\infty} \ra 0, \qquad t\ra \infty. 
\]
The latter, together with inequality \eqref{ine2} and  $\sup_{|x|\lqq K}\norm{Y^x_t - Z_\infty} \ra 0$ for $t\ra \infty$ implies $\sup_{|x|\lqq K} \norm{Y^{\e}_t(x) - \e Z_\infty}\ra 0$ for $t\ra \infty$.
In the sequel, we dominate  $\sup_{|x|\lqq K}\norm{Y_t^x-Z_\infty}$. 
By Lemma~\ref{lem: continuidad} we have 
\[
\sup_{|x|\lqq K}\norm{Y_t^x-Z_\infty} =  \norm{Y_t^{\ti x}-Z_\infty} 
\]
for some $|\ti x|\lqq K$. For convenience of notation we drop the tilde and write $x$. 
The proof is divided in 2 steps. 

\bigskip
\noindent
\textbf{Step 1.}
We start with the proof that for any $0<t\lqq \infty$, $Y_t^x$ has a continuous density.
From Theorem~28.1  in \cite{SA}, it is sufficient to show that 
\begin{equation}\label{eq: characteristic function in L1}
\int_{\RR^d} \big|\EE\big[e^{i \< z, Y_t^x\>}\big]\big| \ud z < \infty \qquad \mbox{ for all  } t >0.   
\end{equation}
Fix $t>0$. Since
\[
\mathbb{E}\left[e^{i\<z,Y_t^x\>}\right]=
\exp\Big(\int_{0}^{t}\psi(\Phi^*_s(x)(\Phi^{-1}_t(x))^*z)\ud s\Big)\quad \textrm{ for } z\in \RR^d,
\]
we have
\begin{align}\label{bound}
\big|\EE\big[e^{i \< z, Y_t^x\>}\big]\big| 
&\lqq  
\exp\left(\int_0^t \int_{\RR^d}\big(\cos(\<(\Phi^*_s(x)(\Phi^{-1}_t(x))^*z, \theta\>) -1\big) \nu(\ud \theta) \ud s\right)\nonumber\\ 
&\lqq \exp\left(\int_0^t \int_{|\<(\Phi^*_s(x)(\Psi^{-1}_t(x))^*z, \theta\>)| \lqq \pi}\big(\cos(\<(\Phi^*_s(x)(\Phi^{-1}_t(x))^*z, \theta\>) -1\big) \nu(\ud \theta) \ud s\right)\nonumber\\ 
&\lqq \exp\left(-2\int_0^t \int_{|\<(\nicefrac{\Phi^*_s(x)(\Phi^{-1}_t(x))^*z}{\pi}, \theta\>)| \lqq 1}
|\<(\nicefrac{\Phi^*_s(x)(\Phi^{-1}_t(x))^*z}{\pi}, \theta\>)|^2
\nu(d\theta) \ud s\right),
\end{align}
where the last inequality follows from the well-known inequality $1- \cos(x) \gqq 2 \big(\frac{x}{\pi}\big)^2$ for $|x|\lqq \pi$.
By Lemma~\ref{lem: cotainfsup} we know that there exist positive 
constants
$c_1,c_2(|x|)$ such that
\begin{equation}\label{eq: exponential estimates}
c_1e^{-c_2(|x|) (t-s)}|z| \lqq |\Phi^*_s(x)(\Phi^{-1}_t(x))^*z|
\end{equation}
for any $t\gqq 0$, $s\in [0,t]$ and $z\in \RR^d$.
Due to the boundedness of the characteristic function it is enough to prove that
\[
\int_{|z|> R}\big|\EE\big[e^{i \< z, Y_t^x\>}\big]\big| \ud z<\infty\quad \textrm{ for some } R>0.
\]
Let  $R>C_\sphericalangle \frac{ \pi e^{c_2(|x|) t}}{c_1}$ and $|z|\gqq R$,
where $C_\sphericalangle$ is given in Lemma~\ref{lem oreymasuda}.
 Then
\[
\frac{|\Phi^*_s(x)(\Phi^{-1}_t(x))^*z|}{\pi}\gqq \frac{c_1e^{-c_2(|x|) (t-s)}}{\pi}|z|\gqq C_\sphericalangle e^{c_2(|x|) s}\gqq C_\sphericalangle \quad \textrm{ for any } s\in [0,t].
\]
 By  \eqref{bound} and Lemma~\ref{lem oreymasuda} we obtain 
\begin{align*}
\big|\EE\big[e^{i \< z, Y_t^x\>}\big]\big| 
\lqq \exp\left(-\frac{2c_\sphericalangle c^{\alpha}_1|z|^{\alpha}}{\pi^{\alpha}}\int_0^t e^{-c_2(|x|) \alpha(t-s)} \ud s\right)
\lqq \exp\left(-\frac{2c_\sphericalangle c^{\alpha}_1|z|^{\alpha}}{\pi^{\alpha}c_2(|x|) \alpha }(1-e^{-c_2(|x|) \alpha t})\right)
\end{align*}
for any $|z|\gqq R$, which implies the existence of $\mathcal{C}^{\infty}_{\textrm{b}}$ density (see for instance Theorem~28.1 in \cite{SA}).
For $Y_\infty$, we just notice that $Y_\infty=Z_\infty$ in distribution and 
$Z_\infty$ has a $\mathcal{C}^{\infty}_{\textrm{b}}$ density (see Case~3 Section~4 in \cite{BP}).\\

\noindent
\textbf{Step 2.} Convergence in total variation.
We prove $\norm{Y_t^x - Z_\infty} \ra 0$, as $t\ra \infty$.
For any $R>0$ fixed we split 
\begin{align*}
\int_{\RR^d} \big|\EE\big[e^{i \< z, Y^x_t\>}\big] - \EE\big[e^{i \< z, Z_\infty \>}\big]\big|\ud z
&= \bigg(\int_{|z|\lqq R} + \int_{|z|> R}\bigg) \big|\EE\big[e^{i \< z, Y^x_t\>}\big] - \EE\big[e^{i \< z, Z_\infty \>}\big]\big|\ud z. 
\end{align*}
By Lemma~\ref{lem:convergenceindist} and the uniform convergence of the characteristic functions on compact sets we have that 
\[
\lim_{t \ra \infty} 
\int_{|z|\lqq R} \big|\EE\big[e^{i \< z, Y^x_t\>}\big] - \EE\big[e^{i \< z, Z_\infty \>}\big]\big| \ud z = 0.
\]
Note that
\begin{align*}
\int_{|z|> R} \big|\EE\big[e^{i \< z, Y^x_t\>}\big] - \EE\big[e^{i \< z, Z_\infty \>}\big]\big| \ud z 
\lqq 
\int_{|z|> R} \big| \EE\big[e^{i \< z, Y^x_t\>}\big]\big | \ud z 
+\int_{|z|> R} \big| \EE\big[e^{i \< z, Z_\infty \>}\big]\big| \ud z.
\end{align*}
It is easy to see that the Orey-Masuda condition implies condition (H) in \cite{BP}. In the proof of Proposition~5.3 there it is shown that under condition (H) we have
\[
\lim\limits_{R\to \infty}\limsup\limits_{t\to \infty}\int_{|z|> R} \big| \EE\big[e^{i \< z, Z_\infty \>}\big]\big| \ud z=0 .
\]
Therefore, the limit  
\begin{equation}\label{falta}
\lim\limits_{R\to \infty}\limsup\limits_{t\to \infty}\int_{|z|> R} \big| \EE\big[e^{i \< z, Y^x_t \>}\big]\big| \ud z=0
\end{equation}
yields the desired result. Indeed, 
\begin{align*}
&\limsup_{t\ra \infty}  \int_{\RR^d} \big|\EE\big[e^{i \< z, Y^x_t\>}\big] - \EE\big[e^{i \< z, Z_\infty \>}\big]\big|\ud z  \\
&\hspace{2cm} \lqq 
\limsup\limits_{t\to \infty}\int_{|z|> R} \big| \EE\big[e^{i \< z, Y^x_t \>}\big]\big| \ud z+\limsup\limits_{t\to \infty}\int_{|z|> R} \big| \EE\big[e^{i \< z, Z_\infty \>}\big]\big| \ud z,
\end{align*}
where the left-hand side does not depend on $R$. Sending $R\to \infty$ we obtain  
\[
\lim_{t\ra \infty}  \int_{\RR^d} \big|\EE\big[e^{i \< z, Y^x_t\>}\big] - \EE\big[e^{i \< z, Z_\infty \>}\big]\big|\ud z=0.
\]
In the sequel, we prove inequality \eqref{falta}.
By inequality 
\eqref{bound} we have
\begin{align*}
\big|\EE\big[e^{i \< z, Y^x_t\>}\big]\big| 
\lqq \exp\left(-2\int_0^t \int_{|\<(\nicefrac{\Phi^*_s(\Phi^{-1}_t)^*z}{\pi}, \theta\>)| \lqq 1}
|\<(\nicefrac{\Phi^*_s(x)(\Phi^{-1}_t(x))^*z}{\pi}, \theta\>)|^2
\nu(\ud \theta) \ud s\right).
\end{align*}
By Lemma~\ref{lem: cotainfsup} we know that there exist positive 
constants
$c_1,c_2(|x|)$ such that
\begin{equation*}
c_1e^{-c_2(|x|) (t-s)}|z| \lqq |\Phi^*_s(x)(\Phi^{-1}_t(x))^*z|\quad \textrm{ for  } t\gqq 0, s\in [0,t], z\in \RR^d.
\end{equation*}
Let $R>\frac{\pi C_\sphericalangle}{C_1}$,
 $t>\frac{1}{c_2(|x|)}\ln(\frac{C_1 R}{\pi})=:t_0(R)$, $s\in [0,t]$ and $|z|\gqq R$, 
 where $C_\sphericalangle$ is given in Lemma~\ref{lem oreymasuda}.
Then we obtain
\[
\frac{|\Phi^*_s(x)(\Phi^{-1}_t(x))^*z|}{\pi}\gqq \frac{c_1e^{-c_2(|x|) (t-s)}R}{\pi}\gqq C_\sphericalangle\quad \textrm{ whenever }\quad s\in [t-t_0(R),t].
\]
Observe that
\begin{align*}
\big|\EE\big[e^{i \< z, Y_t^x\>}\big]\big| 
&\lqq \exp\left(-2\int_0^{t-t_0(R)} \int_{|\<(\nicefrac{\Phi^*_s(x)(\Phi^{-1}_t(x))^*z}{\pi}, \theta\>)| \lqq 1}
|\<(\nicefrac{\Phi^*_s(x)(\Phi^{-1}_t(x))^*z}{\pi}, \theta\>)|^2
\nu(\ud \theta) \ud s\right)\\
&\hspace{1cm}\cdot \exp\left(-2\int_{t-t_0(R)}^t \int_{|\<(\nicefrac{\Phi^*_s(x)(\Phi^{-1}_t(x))^*z}{\pi}, \theta\>)| \lqq 1}
|\<(\nicefrac{\Phi^*_s(x)(\Phi^{-1}_t(x))^*z}{\pi}, \theta\>)|^2
\nu(\ud \theta) \ud s\right)\\
&\lqq \exp\left(-2\int_{t-t_0(R)}^t \int_{|\<(\nicefrac{\Phi^*_s(x)(\Phi^{-1}_t(x))^*z}{\pi}, \theta\>)| \lqq 1}
|\<(\nicefrac{\Phi^*_s(x)(\Phi^{-1}_t(x))^*z}{\pi}, \theta\>)|^2
\nu(\ud \theta) \ud s\right).
\end{align*}
The Orey-Masuda cone condition (Lemma~\ref{lem oreymasuda}) and equality \eqref{eq: exponential estimates} yield
\begin{align*}
\big|\EE\big[e^{i \< z, Y_t^x\>}\big]\big| 
\lqq \exp\left(-\frac{2c_\sphericalangle c_1^{\alpha} |z|^{\alpha}}{\pi^{\alpha}}\int_0^t 
e^{-c_2(|x|) \alpha(t-s)} \ud s\right)
\lqq \exp\left(-
\tilde{c}(|x|) |z|^{\alpha}
(1-e^{-c_2(|x|) \alpha t_0(R)})\right)
\end{align*}
for any $|z|\gqq R$ and $t>t_0(R)$, where $\tilde{c}(|x|):=\frac{2c_\sphericalangle c^{\alpha}_1}{\pi^{\alpha}c_2(|x|) \alpha}>0$.
Therefore for $|x|\lqq K$ there are positive 
constants $\ti c(K)$ and $c_2(K)$ such that 
\begin{align*}
\limsup\limits_{t\to \infty}
\int_{|z|\gqq R}\big|\EE\big[e^{i \< z, Y_t^x\>}\big]\big| \ud z\lqq 
\int_{|z|\gqq R}\exp\Big(-
\tilde{c}(K) |z|^{\alpha}
(1-e^{-c_2(K) \alpha t_0(R)})\Big) \ud z 
\end{align*}
for any $R> \frac{\pi C_\sphericalangle}{C_1}$. 
Sending $R\to \infty$, the dominated convergence theorem implies \eqref{falta}.
\end{proof}

\bigskip 

\subsection{\textbf{Geometric profile characterization for rotationally invariant $Z_\infty$}}\hfill

\begin{lem}\label{lem:monotoniaTV}
Let $f\in \mathcal{C}^1(\RR^d,(0,\infty))$ be a smooth density
 such that
$f(z) = g(|z|)$
for some function $g\in \mathcal{C}^1((0, \infty), (0, \infty))$  with  $g'(s) < 0$ for all $s > 0$  and  $g'\in L^1(\RR^d)$. 
Then the map 
\[
(0, \infty) \ni r \mapsto \int_{\RR^d} |f(z+r e_1) - f(z)| \ud z \in (0, \infty)
\]
is strictly increasing. In particular, it is injective.           
\end{lem}

\begin{proof} 
First we rewrite 
\begin{align*}
\int_{\RR^d} |f(z+r e_1) - f(z)| \ud z 
&= 2 - 2 \int_{\RR^d} (f(z+r e_1) \wedge f(z))\ud z.
\end{align*} 
By the definition of the minimum we have 
 \[
  \frac{\ud }{\ud r}\Big(f(z+r e_1) \wedge f(z)\Big) 
  = 
\begin{cases} 
0 & \mbox{ for all }z: \quad f(z)<  f(z+r e_1),\\
 \frac{\ud}{\ud r}f(z+r e_1) & \mbox{ for all }z: \quad f(z)> f(z+r e_1).
\end{cases}
\]
In the sequel we determine the shape of 
$
 \{z\in \RR^d~|~f(z)> f(z+r e_1)\}.$
Since 
\[
f(z) < f(\ti z)\quad \mbox{ if and only if }\quad |z| > |\ti z|, 
\]
the continuity of $f$ yields 
\begin{align*}
f(z+r e_1) \wedge f(z) 
&= \begin{cases} f(z) & \mbox{ for all }z: \quad |z +r e_1| < |z|, \\ 
f(z+r e_1) &\mbox{ for all }z: \quad |z +r e_1| > |z|. \end{cases}
\end{align*}
That is, we obtain geometrically the shifted half space  
\begin{align*}
\{z\in \RR^d~|~|z +r e_1|^2 > |z|^2\} = \{z\in \RR^d~|~z_1 > - \nicefrac{r}{2}\}. 
\end{align*}
Consequently, it follows  
 \[
  \frac{\ud}{\ud r}\Big(f(z+r e_1) \wedge f(z)\Big) =  
  \begin{cases} 
  0 &\mbox{ for all }z: \quad z_1 < - \frac{r}{2},\\ 
  \frac{\ud}{\ud r} f(z+r e_1) & \mbox{ for all }z: \quad z_1 > - \frac{r}{2}. 
  \end{cases}
 \]
We continue with the computation of $\frac{\ud}{\ud r} f(z+r e_1)$.  
For all $z\neq 0$  we have
 \begin{align*}
\frac{\ud}{\ud r} f(z+r e_1) 
= \frac{\ud}{\ud r} g(|z+r e_1|) 
= g'(|z+r e_1|) \frac{\<z+r e_1,e_1\>}{|z
+r e_1|}
= g'(r|\frac{z}{r}+ e_1|) \frac{(1 + \<\frac{z}{r},  e_1\>) }{|e_1 +\frac{z}{r}|},
\end{align*}
such that  the Leibniz integral rule 
 and the implicit function theorem imply
\begin{align*}
\frac{\ud }{\ud r} \int_{\RR^d}  f(z+r e_1) \wedge f(z) \ud z 
&= \int_{\RR^d} \frac{\ud }{\ud r}  f(z+r e_1) \wedge f(z) \ud z=\int_{\frac{z_1}{r} > -\frac{1}{2}} g'(r| \frac{z}{r}+ e_1|) \frac{(1 + \frac{z_1}{r}) }{|e_1 +\frac{z}{r}|} \ud z\\
&=\int_{v_1 > -\frac{1}{2}} \underbrace{g'(r |v+e_1|)}_{<0} \underbrace{\frac{(1 + v_1) }{|e_1 +v|}}_{>0} r^d \ud v< 0.
\end{align*}
Consequently, we obtain the desired result
\begin{align*}
\frac{\ud}{\ud r} \int_{\RR^d} |f(z+r e_1) - f(z)| dz 
&=- 2 \frac{\ud}{\ud r}\int_{\RR^d} f(z+r e_1) \wedge f(z)dz\\
&=- 2 \int_{v_1 > -\frac{1}{2}} \underbrace{g'(r |v+e_1|)}_{<0} \underbrace{\frac{(1 + v_1) }{|e_1 +v|}}_{>0} r^d \ud v >0.
\end{align*}
\end{proof} 

\bigskip 
\section{\textbf{Exponential ergodicity of coercive L\'evy SDEs in $L^\beta, \beta>0$}}\label{Appendix D}\hfill

\noindent  
In this section we fix the following standing assumptions.
Let $b$ be a vector field satisfying Hypothesis~\ref{hyp: potential} and  $A$  a $d$-squared matrix  with real entries. 
Consider  a L\'evy process  $L=(L_t)_{t\gqq 0}$  with values in $\RR^d$ with  strongly locally layered stable
L\'evy measure  $\nu$
with given parameters 
$(\nu_0,\nu_\infty,\Lambda,q,c_0,\alpha)$ satisfying Hypotheses \ref{hyp: moment condition}, \ref{hyp: regularity} and \ref{hyp: blumental}
and the strong solution $X=(X_t)_{t\gqq 0}$  of the SDE
\begin{align}\label{ddeduf}
\left\{
\begin{array}{r@{\;=\;}l}
\ud X_t & - b(X_t) \ud t + A\ud L_t \quad \textrm{ for  } t\gqq 0,\\
X_0 & x\in \RR^d.
\end{array}
\right.
\end{align} 

\begin{defn}[H\"ormander condition, nonlinear Kalman rank condition]\label{def: controllability} \hfill\\
Under the standing assumptions we denote by $B_0=I_{d}$ be the identity matrix on $\RR^d$ and define for $n\in \NN$ the
$(d\times d)$-matrix-valued function $B_{n}(x)$ recursively by
\[
B_{n}(x):=-b(x)\cdot DB_{n-1}(x)+Db(x)B_{n-1}(x),\qquad x\in \RR^d,
\]
where $b\cdot DF:=\sum_{k=1}^{d}b_k\frac{\partial}{\partial_k}F$, and $F$ is a 
$(d\times d)$-matrix-valued function.
We say that the SDE \eqref{ddeduf} satisfies a H\"ormander condition if its coefficients $b=(b_1,\ldots, b_d)^*$ and the matrix $A$ satisfy the following:
For each $x\in \RR^d$ there exists some $n=n(x)\in \mathbb{N}\cup \{0\}$  such that
\begin{equation}\label{hormander}
\mathrm{Rank}[B_0 A, B_1(x)A,\ldots, B_{n}(x)A]=d.
\end{equation}
\end{defn}

\begin{lem}[Orey-Masuda type condition]
Under the standing assumptions,
the limit \eqref{layered0} implies
\[
\lim\limits_{h\to 0} h^{\alpha-2}\int_{|z|\lqq h} |z|^2\nu(\ud z):=\kappa_1>0.
\]
The latter is Condition $\mathrm{(1.2)}$ in \cite{SongZhang}.
\end{lem}
\begin{proof}
Observe that
\[
\int_{|z|\lqq h} |z|^2\nu(\ud z)= \int_{\SSS^{d-1}}\int_{0}^{h} r^2 q(r,\theta)\ud r\Lambda(\ud \theta).
\]
Let $\eta>0$ be fixed.
By limit \eqref{layered0}  we deduce
that there exists $r_0:=r_0(\eta)\in (0,1)$ such that for any $0<r<r_0$ we have
\[
r^{1-\alpha}(c_0(\theta)-\eta)<r^{2}q(r,\theta)<r^{1-\alpha}(c_0(\theta)+\eta) \quad \textrm{ for any } \theta\in \SSS^{d-1}.
\]
Let $h\in(0,r_0)$. Then
\[
\begin{split}
\frac{1}{2-\alpha}\int_{\SSS^{d-1}}(c_0(\theta)-\eta)\Lambda(\ud \theta)\lqq 
h^{\alpha-2}\int_{\SSS^{d-1}}\int_{0}^{h} r^2 q(r,\theta)\ud r\Lambda(\ud \theta)\lqq 
\frac{1}{2-\alpha}\int_{\SSS^{d-1}}(c_0(\theta)+\eta)\Lambda(\ud \theta).
\end{split}
\]
Sending $h\to 0$ followed by sending $\eta \to 0$, we obtain 
\[
\lim\limits_{h\to 0}h^{\alpha-2}\int_{\SSS^{d-1}}\int_{0}^{h} r^2 q(r,\theta)\ud r\Lambda(\ud \theta)= 
\frac{1}{2-\alpha}\int_{\SSS^{d-1}} c_0(\theta) \Lambda(\ud \theta)>0,
\]
where the last inequality follows from the fact that
$c_0:\SSS^{d-1}\rightarrow (0,\infty)$  and $c_0\in L^1(\Lambda)$.
\end{proof}
In the sequel we extend Theorem~4.1 in
\cite{Peng} to $L^\beta$ for arbitrary $\beta>0$.
\begin{thm}[Exponential ergodicity]\label{Ext PengZhang}\hfill\\
Under the standing assumptions
and the H\"ormander condition \eqref{hormander} there exists a unique invariant distribution $\mu$ for \eqref{ddeduf} satisfying
exponential ergodicity in the total variation distance.
\end{thm}

\begin{proof}
For $\beta\gqq 2$, it is the statement of Theorem~4.1 in \cite{Peng}.
Let $\beta\in (0,2)$.
We apply Theorem~2.1 in \cite{Peng}. Therefore, we verify Conditions LC, 
H1 and H2 in \cite{Peng}, p. 2-3.

\noindent \textbf{Condition LC in \cite{Peng}.} 
We stress that the fulfilment of Condition LC  only requires Hypothesis~\ref{hyp: potential} and Hypothesis~\ref{hyp: moment condition}.
We define $|\cdot|_c := \sqrt{|x|^{2
} + c^2}$ for $c>0$ satisfying  for all $x\in \RR^d$
\[
c \lqq |x|_c \lqq |x| + c, \quad \nabla |x|_c := \frac{x}{|x|_c} 
\quad \textrm{ and }\quad 0 \lqq \frac{|x|}{|x|_c} < 1.
\]
In addition, we have
\begin{align*} 
D^2 |x|_c 
&=\left(
\begin{array}{ccccc} 
\frac{|x|_c^2 - x_1^2}{|x|_c^3} & - \frac{x_1 x_2}{|x|_c^3} & - \frac{x_1 x_3}{|x|_c^3}  & \dots & - \frac{x_1 x_d}{|x|_c^3} \\
- \frac{x_1 x_2}{|x|_c^3}       &   \frac{|x|_c^2 - x_2^2}{|x|_c^3} & - \frac{x_2 x_3}{|x|_c^3}  &  & \\
- \frac{x_1 x_3}{|x|_c^3}       &  - \frac{x_2 x_3}{|x|_c^3} &  \frac{|x|_c^2 - x_3^2}{|x|_c^3} &   & \vdots\\
\vdots       &   &        & \ddots & \\
- \frac{x_1 x_d}{|x|_c^3} & & \dots & & \frac{|x|_c^2 - x_d^2}{|x|_c^3}
\end{array}
\right).
\end{align*}
Consequently,
\begin{align*}
\|D^2 |x|_c\|_1 &= \sum_{i,j} |(D^2 |x|_c)_{ij}| 
=\sum_{i} |(D^2 |x|_c)_{i,i}|+\sum_{i\not=j} |(D^2 |x|_c)_{ij}|\\
&=\frac{d|x|_c^2-\sum_{i}x^2_i+\sum_{i\not=j}|x_ix_j|}{|x|^3_c}
\lqq \frac{1}{|x|_c} \left(d+\frac{|x|^2}{|x|_c^2}\right) \lqq  \frac{d+1}{c}.
\end{align*}
Let
$0<\gamma\lqq \beta\wedge 1$. 
We calculate the gradient and the Hessian of $|x|^{\gamma}_c$ as follows:
\[
\nabla |x|_c^{\gamma}=\gamma |x|^{\gamma-1}_c \nabla |x|_c=
\gamma |x|^{\gamma-1}_c 
\frac{x}{|x|_c}=\gamma |x|^{\gamma-2}_c 
x
\]
and
\begin{align*}
&\partial_{ii} |x|_c^{\gamma}=
\partial_i(
\gamma |x|^{\gamma-2}_c x_i)=\gamma|x|^{\gamma-2}_c+  
\gamma
(\gamma-2)|x|^{\gamma-4}_c x^2_i,\\
&
\partial_{ij} |x|_c^{\gamma}=
\partial_i(
\gamma |x|^{\gamma-2}_c x_j)=
\gamma x_j \partial_i(
|x|^{\gamma-2}_c )=
\gamma 
(\gamma-2)|x|^{\gamma-4}_c x_ix_j\quad \textrm{ for } i\neq j.
\end{align*}
Hence,
\begin{align*}
\sum_{i,j}|\partial_{ij} |x|_c^{\gamma}|&=
\gamma 
(2-\gamma)|x|^{\gamma-4}_c \sum_{i,j} |x_ix_j|+d\gamma|x|^{\gamma-2}_c
\gamma 
(2-\gamma)|x|^{\gamma-4}_c \|x\|^2+d\gamma|x|^{\gamma-2}_c,
\end{align*}
where $\|\cdot\|$ denotes the 1-norm. Since $\|x\|\lqq \sqrt{d} |x|$, we obtain
\begin{align*}
\sum_{i,j}|\partial_{ij} |x|_c^{\gamma}|
&\lqq 
(\gamma 
(2-\gamma)d+d\gamma)|x|^{\gamma-2}_c\lqq 
(\gamma 
(2-\gamma)d+d\gamma)c^{\gamma-2}.
\end{align*}
With the help of the preceding calculations
It\^o's formula yields 
\begin{align*}
|X_t|_c^{\gamma} 
&= |x|_c^{\gamma} -\gamma \int_0^t \< 
|X_s|^{\gamma-2}_c
X_s, b(X_s)\> \ud s + \int_0^t \int_{|z|< 1} \big( |X_{s-}+ A z|_c^{\gamma} - |X_{s-}|_c^{\gamma} \big) \ti N(\ud s \ud z) \\
&\qquad + \int_0^t \int_{|z|\gqq 1} \big(|X_{s-}+ A z|_c^{\gamma} - |X_{s-}|_c^{\gamma}\big) N(\ud s \ud z) \\
&\qquad + \int_0^t \int_{|z|< 1} \big( |X_{s-}+ A z|_c^{\gamma} - |X_{s-}|_c^{\gamma} - \< \gamma |X_s|^{\gamma-2}_c
X_s, A z\> \big) \nu(\ud z) \ud s,
\end{align*}
where $N$ is a Poisson random measure with compensator $\ud t \otimes \nu(\ud z)$.  Moreover, we have the L\'evy-I\^o decomposition such that  $\PP$-a.s. for all $t\gqq 0$
\[
L_t = \int_0^t \int_{|z|\lqq 1} z \ti N(\ud s \ud z) +  \int_0^t \int_{|z|> 1} z N(\ud s \ud z),
\]
where $\ti N$ is the compensated version of $N$. 
Taking expectations we obtain
\begin{align*}
\EE\big[|X_t|_c^{\gamma}\big] 
&= |x|_c^{\gamma} -\gamma \int_0^t \EE\big[\< |X_s|^{\gamma-2}_c
X_s, b(X_s)\>\big] \ud s \\
&\qquad + \EE\Big[\int_0^t  \int_{|z|< 1} \big( |X_{s-}+ A z|_c^{\gamma} - |X_{s-}|_c^{\gamma} \big) \ti N(\ud s \ud z)\Big] \\
&\qquad + \EE\Big[\int_0^t \int_{|z|\gqq 1} \big(|X_{s-}+ A z|_c^{\gamma} - |X_{s-}|_c^{\gamma}\big) N(\ud s \ud z)\Big ] \\
&\qquad + \int_0^t \EE\Big[\int_{|z|< 1} \big( |X_{s-}+ A z|_c^{\gamma} - |X_{s-}|_c^{\gamma} - \< \gamma|X_s|^{\gamma-2}_c
X_s, A z\> \big) \nu(\ud z)\Big] \ud s.
\end{align*}
First, since the moment of order $\beta$ is finite, a localization argument yields
\[
\EE\Big[\int_0^t  \int_{|z|< 1} \big( |X_{s-}+ A z|_c^{\gamma} - |X_{s-}|_c^{\gamma} \big) \ti N(\ud s \ud z)\Big]=0\quad \textrm{ for } t\gqq 0.
\]
Secondly, by the It\^o isometry for Poisson random measures (see \cite{Appl}) we obtain
\begin{align*}
\EE\Big[\int_0^t \int_{|z|\gqq 1} \big(|X_{s-}+ A z|_c^{\gamma} - |X_{s-}|_c^{\gamma}\big) N(\ud s \ud z)\Big ]&=\EE\Big[\int_0^t \int_{|z|\gqq 1} \big(|X_{s-}+ A z|_c^{\gamma} - |X_{s-}|_c^{\gamma}\big) \nu(\ud z) \ud s\Big ]\\
&\hspace{-0cm}=\int_0^t \EE\Big[ \int_{|z|\gqq 1} \big(|X_{s-}+ A z|_c^{\gamma} - |X_{s-}|_c^{\gamma}\big) \nu(\ud z)\Big ] \ud s.
\end{align*}
Hence, for almost all $t$ we have
\begin{align*}
\EE\big[|X_t|_c^{\gamma}\big] 
&= |x|_c^{\gamma} -\gamma \int_0^t \EE\big[\< |X_s|^{\gamma-2}_c
X_s, b(X_s)\>\big] \ud s \\
&\qquad + \int_0^t \EE\Big[\int_{|z|\gqq 1} \big(|X_{s-}+ A z|_c^{\gamma} - |X_{s-}|_c^{\gamma}\big) \nu(\ud z)\Big ] \ud s \\
&\qquad + \int_0^t \EE\Big[\int_{|z|< 1} \big( |X_{s-}+ A z|_c^{\gamma} - |X_{s-}|_c^{\gamma} - \< \gamma|X_s|^{\gamma-2}_c
X_s, A z\> \big) \nu(\ud z)\Big] \ud s.
\end{align*}
Taking derivatives we obtain
\begin{align*}
\frac{\ud}{\ud t}\EE\big[|X_t|_c^{\gamma}\big] 
&= -\gamma\EE\big[\< |X_s|^{\gamma-2}_c
X_s, b(X_t)\>\big] +  \EE\Big[\int_{|z|\gqq 1} \big(|X_{t-}+ A z|_c^{\gamma} - |X_{t-}|_c^{\gamma}\big) \nu(\ud z)\Big ] \\
&\qquad +  \EE\Big[\int_{|z|< 1} \big( |X_{t-}+ A z|_c^{\gamma} - |X_{t-}|_c^{\gamma} - \< \gamma|X_s|^{\gamma-2}_c
X_s, A z\> \big) \nu(\ud z)\Big].
\end{align*}
By Hypothesis~\ref{hyp: potential} it follows that
\begin{align*}
-\gamma |X_s|^{\gamma-2}_c\< 
X_s, b(X_s)\> 
& \lqq 
-\delta \gamma |X_s|^{\gamma-2}_c|X_s|^2
=-\delta \gamma |X_s|^{\gamma-2}_c(|X_s|^2_c - c^2)\\[2mm]
&=-\delta \gamma |X_s|^{\gamma}_c+\delta \gamma c^2 |X_s|^{\gamma-2}_c\lqq -\delta \gamma |X_s|^{\gamma}_c+\delta \gamma c^\gamma.
\end{align*}
Hence 
\begin{align*}
\frac{\ud}{\ud t}\EE\big[|X_t|_c^{\gamma}\big] 
&\lqq - \delta \gamma \EE\big[|X_t|_c^{\gamma}\big] + \delta \gamma c^\gamma    + \EE\Big[\int_{|z|\gqq 1} \big(|X_{t-} + A z|_c^{\gamma} - |X_{t-}|_c^{\gamma}\big)\nu(\ud z) \Big] \\
&\quad + \EE\Big[ \int_{|z|< 1} \big(|X_{t-} + A z|_c^{\gamma} - |X_{t-}|_c^{\gamma} - \< \gamma|X_s|^{\gamma-2}_c
X_s, A z\>\big)\nu(\ud z) \Big].
\end{align*}
For $\gamma \in (0,1]$, the subadditivity of  the power of order $\gamma$ yields
\begin{align*}
\EE\Big[\int_{|z|\gqq 1} \big(|X_{t-} + A z|_c^{\gamma} - |X_{t-}|_c^{\gamma}\big)\nu(\ud z)\Big]
&\lqq   \int_{|z|\gqq 1} |Az|_c^{\gamma} \nu(\ud z)\\
&\lqq \|A\|^\gamma \int_{|z|\gqq 1} |z|^\gamma \nu(\ud z) =\|A\|^\gamma C_1,
\end{align*}
and 
\begin{align*}
&\EE \Big[\int_{|z|< 1} \big(|X_{t-} + A z|_c^{\gamma} - |X_{t-}|_c^{\gamma} - \< \gamma|X_s|^{\gamma-2}_c
X_s, A z\>\big)\nu(\ud z)\Big]\\
&\hspace{2cm}\lqq  (\gamma 
(2-\gamma)d+d\gamma)c^{\gamma-2} \int_{|z|<1} |Az|^2 \nu(\ud z) \\
&\hspace{2cm}\lqq  
 \|A\|^2 c^{\gamma-2} (\gamma(2-\gamma)d+d\gamma)\int_{|z|<1} |z|^2 \nu(\ud z)) 
=:\|A\|^2 c^{\gamma-2} C_2. 
\end{align*}
Therefore we have 
\begin{align*}
\frac{\ud}{\ud t}\EE\big[|X_t|_c^{\gamma}\big] 
&\lqq - \delta \gamma \EE\big[|X_t|_c^{\gamma}\big] + \delta \gamma c^\gamma +\|A\|^\gamma C_1 +\|A\|^2 c^{\gamma-2}C_2 
\end{align*}
and the Gr\"onwall lemma yields 
\begin{align}\label{eq:99 ergoestimate}
\EE\big[|X_t|^{\gamma}\big]  & \lqq \EE\big[|X_t|_c^{\gamma}\big] 
\lqq |x|_c^{\gamma}\, e^{-\delta \gamma t} + \frac{1-e^{-\delta \gamma t}}{\delta\gamma}(\delta \gamma c^\gamma +\|A\|^\gamma C_1 +\|A\|^2 c^{\gamma-2}C_2  )\nonumber\\
& \lqq |x|_c^{\gamma}\, e^{-\delta \gamma t} +\frac{1}{\delta\gamma}(\delta \gamma c^\gamma +\|A\|^\gamma C_1 +\|A\|^2 c^{\gamma-2}C_2  )
\nonumber\\
& \lqq |x|^\gamma \, e^{-\delta \gamma t}+C_3,
\end{align}
where $C_3=c^\gamma+\frac{1}{\delta\gamma}(\delta \gamma c^\gamma +\|A\|^\gamma C_1 +\|A\|^2 c^{\gamma-2}C_2 )$.

\noindent \textbf{Condition $\mathbf{H_{1}}$ in \cite{Peng}}. 
We emphasize that Condition  $\mathbf{H_{1}}$ also only requires
Hypothesis~\ref{hyp: potential} and Hypothesis~\ref{hyp: moment condition}. 
In the sequel we consider the solution 
$(X_t(x))_{t\gqq 0}$ of 
 \eqref{ddeduf} with initial condition~$x$.
By Hypothesis~\ref{hyp: potential}
we have for all $x, y \in \RR^d$ 
\begin{align}
\frac{\ud}{\ud t} |X_t(x) - X_t(y)|_c^{\gamma} 
&= 
-\gamma |X_t(x) - X_t(y)|^{\gamma-2}_c 
\< 
X_t(x) - X_t(y)
, b(X_t(x)) - b(X_t(y))\>
\nonumber
\\[2mm]
&\lqq -\gamma \delta |X_t(x) - X_t(y)|^{\gamma-2}_c
|X_t(x) - X_t(y)|^{2}
\nonumber
\\[2mm]
&= -\gamma \delta |X_t(x) - X_t(y)|^{\gamma-2}_c
(|X_t(x) - X_t(y)|^2_c - c^2)
\nonumber\\[2mm]
&= -\gamma \delta |X_t(x) - X_t(y)|^{\gamma}_c+  \gamma \delta |X_t(x) - X_t(y)|^{\gamma-2}_c c^2
\nonumber
\\[2mm]
&\lqq -\gamma \delta |X_t(x) - X_t(y)|^{\gamma}_c+  \gamma \delta c^{\gamma},\label{eq:xy}
\end{align}
where in the last inequality we use that $|x|_c\gqq c$ and $\gamma \in (0,1]$.
Gr\"onwall's lemma applied to \eqref{eq:xy} yields
\begin{align}
\EE\big[|X_t(x) - X_t(y)|^\gamma\big]& \lqq \EE\big[|X_t(x) - X_t(y)|_c^{\gamma}\big] \lqq |x-y|_c^{\gamma} e^{- \delta \gamma t}+c^\gamma
 \lqq |x-y|^{\gamma} e^{- \delta \gamma t}+2c^\gamma.
\label{e:99H1b}
\end{align}
Let $R>0$ and $\Delta>0$. Here, we analyze the quantity
\[
\PP\big(\underbrace{|X_t(x) - X_t(y)|\lqq \Delta}_{=:D}, 
\underbrace{|X_t(x)|\lqq R_0}_{=:B}, \underbrace{|X_t(y)|\lqq R_0}_{=:C} \big),
\]
where $x,y\in \overline{B}_{R_0}(0)$ for a suitable $R_0>0$.
Observe that
\begin{equation}\label{eq:99subad}
\PP((D\cap B\cap C)^c)=\PP(D^c \cup B^c \cup C^c)\lqq \PP(D^c)+\PP(B^c)+\PP(C^c).
\end{equation}
For any $\Delta>0$ and $R>0$ we set $R_0=\max\{R,(4C_3)^{1/\gamma}\}$, where $C_3$ is the positive constant in estimate \eqref{eq:99 ergoestimate} and take $x,y\in \overline{B}_{R}(0)$.
By estimate \eqref{eq:99 ergoestimate} we obtain for any $t\gqq 0$ 
\begin{align}\label{eq:99largeball1}
\PP(B^c)& =\PP(|X_t(x)|> R_0)  \lqq \frac{\EE[|X_t(x)|^\gamma]}{R^\gamma_0} \nonumber\\
& \lqq \frac{|x|^\gamma \, e^{-\delta \gamma t}+C_3}{R^\gamma_0}\lqq \frac{R^\gamma e^{-\delta \gamma t}+C_3}{R^\gamma_0}
=\frac{R^\gamma}{R^\gamma_0}e^{-\delta \gamma t}+\frac{C_3}{R^\gamma_0}\lqq e^{-\delta \gamma t}+\frac{1}{4}.
\end{align}
Switching the role of $x$ and $y$ we have
\begin{equation}\label{eq:99 largeball2}
\PP(C^c)=\PP(|X_t(y)|> R_0)
\lqq e^{-\delta  \gamma t}+\frac{C_3}{R^\gamma_0}=e^{-\delta \gamma  t}+\frac{1}{4}.
\end{equation}
We continue with the  analysis of $\PP(D^c)$.
Again, let $x,y\in \overline{B}_{R}(0)$.
Define $T'_0(\Delta,R)$ as the unique positive solution of  
\begin{align*}
\frac{((2R)^\gamma+c^\gamma) e^{-\delta \gamma T'_0}}{\Delta^\gamma} = \frac{1}{4},
\end{align*}
where $c^\gamma=\frac{\Delta^\gamma}{8}$.
Hence for any $x,y\in \overline{B}_{R}(0)$ and $t\gqq T_0$ we have by \eqref{e:99H1b} 
the estimate
\begin{align}\label{eq:99proximity}
\PP(D^c)=\PP\big(|X_t(x) - X_t(y)| > \Delta \big) 
&\lqq \frac{\EE\big[|X_t(x) - X_t(y)|^\gamma\big]}{\Delta^\gamma}\nonumber \lqq \frac{|x-y|_c^{\gamma} e^{- \delta \gamma t}+c^\gamma}{\Delta^\gamma } \nonumber\\
&\lqq \frac{((2R)^\gamma+c^\gamma) e^{- \delta \gamma  t}+c^\gamma }{\Delta^\gamma }\lqq \frac{3}{8}.
\end{align}
By taking $T_0=\max\{T'_0,\frac{\ln(32)}{\delta \gamma}\}$  and 
combining \eqref{eq:99subad}, 
\eqref{eq:99largeball1}, \eqref{eq:99 largeball2}    and \eqref{eq:99proximity} it follows uniformly for any $x,y\in \overline{B}_{R}(0)$ and $t\gqq T_0$ that
\[
\mathbb{P}((D\cap B\cap C)^c)\lqq 2e^{-\delta \gamma t}+\frac{7}{8}\lqq \frac{15}{16}<1.
\]
The preceding inequality yields the weak form of irreducibility condition $\mathbf{H_{1}}$ in \cite{Peng} for the canonical coupling.

\noindent
\textbf{Condition $\mathbf{H_{2}}$ in \cite{Peng}.} The proof is virtually identical to  \cite{Peng} p.15-16. 
\end{proof}

The following corollaries are Taylor-made statements for the error estimates in Subsection~\ref{sec: local}.

\begin{cor}
For any $0< \gamma \lqq \beta \wedge 1$ there exists a positive constant $C$ such that for all $\e>0$, $x\in \RR^d$ and $t\gqq 0$ we have 
\begin{align}\label{eq: momentL1}
\EE\big[|X^{\e,x}_t|^\gamma \big]  & \lqq |x|^\gamma\, e^{-\delta \gamma t}  +C\e^\gamma.
\end{align}
\end{cor}
\begin{proof}
The statement follows taking $A=\e I_d$ and $c=\e$ in inequality \eqref{eq:99 ergoestimate}.
\end{proof}

\begin{cor}\label{lem: demomento}
For any $0< \gamma \lqq \beta \wedge 1$ there exists a positive constant $C=C(\delta,d, \gamma)$ such that
\begin{align}\label{eq: gamma moment}
\EE\left[|X^{\e,x}_t|^\gamma\right] \lqq C\e^\gamma  +|\varphi^x_t|^\gamma
\end{align}
for all $t\gqq 0$ and $\e\in (0,1]$.
\end{cor}
\begin{proof}
Note that the difference $\ti X^{\e,x}_t := X^{\e,x}_t - \varphi^x_t$ satisfies 
\[
\ud \ti X^{\e,x}_t = - \Big(\int_0^1 Db(\varphi^x_t + \theta \ti X^{\e,x}_t) \ud \theta\Big) \ti X^{\e,x}_t \ud t + \e \ud L_t, \qquad \ti X^{\e,x}_0 = 0.
\]
Hypothesis~\ref{hyp: potential} together 
with the analogous computations to  the proof of Condition \textbf{LC} 
in Theorem~\ref{Ext PengZhang} in Appendix~\ref{Appendix D} yields that for $A = \e I_d$, $c = \e$ and $0 <\gamma \lqq 1 \wedge \beta$ 
there is a constant $C>0$ such that  $\e \in (0,1]$, $x\in \RR^d$ and $t\gqq 0$ imply
\begin{align*}
\EE\big[|\ti X^{\e,x}_t|^\gamma \big]  & \lqq C\e^\gamma.
\end{align*}
Using the subadditivity of the $\gamma$-power we obtain \eqref{eq: gamma moment}. 
\end{proof}
\begin{cor}\label{cor: demomento}
For any $x\in \mathbb{R}^d$ and $0 < \gamma \lqq \beta \wedge 1$ 
there exists a positive constant $C=C(|x|, \gamma)$ such that for all $\vt\in (0,1)$ and $\e\in (0,1)$ we have 
\[
\PP\big(|X^{\e,x}_{T^{x}_\e}|\gqq r_\e\big)\lqq C(|x|)\e^{\gamma\vartheta}.
\]
\end{cor}
\begin{proof}
By Corollary~\ref{lem: demomento} and Lemma~\ref{lem:ordenepsilon}  we have
\[
\EE\big[|X^{\e,x}_{T^{x}_\e}|^\gamma\big]\lqq   C_1\e^\gamma  +C_2(|x|)\e^\gamma
\]
for some positive constants $C_1$ and $C_2(|x|)$.
The preceding inequality with the help of 
Markov's inequality yields
\begin{align*}
\PP\big(|X^{\e,x}_{T^{x}_\e}|\gqq r_\e\big)\lqq \frac{\EE\big[|X^{\e,x}_{T^{x}_\e}|^\gamma\big]}{\e^{\gamma(1-\vartheta)}}\lqq 
(C_1 +C_2(|x|))\e^{\gamma\vartheta},
\end{align*}
which concludes the statement.
\end{proof}

\bigskip 
\section*{Acknowledgments}
The research of GBV has been supported by the Academy of Finland, via
the Matter and Materials Profi4 university profiling action.
GBV also would like to express
his gratitude to University of Helsinki for all the facilities used along
the realization of this work. 
The research of MAH has been supported by the 
Proyecto de la Convocatoria 2020-2021:
``Stochastic dynamics of systems perturbed with small Markovian noise with applications in biophysics, climatology and statistics"
of the School of Sciences (Facultad de Ciencias) at Universidad de los Andes. 
JCP acknowledges support from  CONACyT-MEXICO CB-250590.
The authors would like to thank professor M. Jara 
and professor R. Imbuzeiro Oliveira both at IMPA for ideas how to construct the example given in Subsubsection \ref{subsub:counter}.

\bibliographystyle{amsplain}

\end{document}